\documentclass[10pt]{article}
\usepackage{amsmath,amsfonts}
\usepackage[mathscr]{eucal}
\usepackage{amssymb,amsmath}
\usepackage{latexsym}
\usepackage{amsfonts}
\usepackage{amscd}
\usepackage{amsthm}
\usepackage{graphics}
\usepackage{epsfig}
\usepackage{esint}
\usepackage{ulem}
\usepackage{hyperref}

\newcommand {\localToGlobalR}{r} 

\newcommand {\justB}{B} 

\newcommand{\Rzresc}[1]{}               

\newcommand{\Rresc}[1]{}                    
\newcommand{\Rrescb}{R}
\newcommand{\Rrescbp}[1]{R^{#1}}                    

\newcommand {\alphawedgeone}{{\alpha}} 
\newcommand {\thmConst}{{\kappa}} 
\newcommand {\injRad}{r_U} 
\newcommand {\chartMap}{F}
\newcommand {\distortion}{{\rm Distortion}}

\newcommand {\efpoly}{{P_1}}
\newcommand {\gefpoly}{{P_2}}
\newcommand {\ggefpoly}{{P_3}}

\newcommand {\localPK}{{{\tilde K}}} 
\newcommand {\globalEF}{{\varphi}}
\newcommand {\globalPEF}{{\tilde\varphi}}
\newcommand {\globalEV}{{\lambda}}
\newcommand {\globalPEV}{{\tilde\lambda}}
\newcommand {\localEF}{{\xi}}
\newcommand {\localPEF}{{\tilde\xi}}   
\newcommand {\localEV}{{\mu}}
\newcommand {\localPEV}{{\tilde\mu}}   
\newcommand {\localTangentBall}{{B_\Rrescb(z)}} 
\newcommand {\localgquadratic}{g_\justB}  
\newcommand {\localgquadraticr}{g_\justB}  
\newcommand {\assumptionA}{{\bf A.1}}

\newcommand {\grad}{\nabla}

\newcommand {\R}{\mathbb {R}}

\newcommand {\cD}{\mathcal {D}}

\newcommand {\cM}{\mathcal{M}}

\newcommand {\E}{\mathbb {E}}

\newcommand {\suml}{\sum\limits}

\newcommand{\pid}{(4\pi)^{-d\over 2}}

\newcommand{\vol}[1]{{|#1|}}
\newcommand{\dist}{{\rm dist}}
\newcommand{\ball}{{B}}

\newcommand{\half}{{1 \over 2}}

\newcommand{\dir}{p}

\newcommand{\efk}{{\beta_{loc}}}
\newcommand{\Cweyl}{{C_{count}}}

\newtheorem{theorem}{Theorem}[subsection]
\newtheorem{lemma}[theorem]{Lemma}

\newtheorem{proposition}[theorem]{Proposition}
\newtheorem{remark}[theorem]{Remark}
\newtheorem{example}[theorem]{Example}

\numberwithin{equation}{subsection}

\newlength{\originalbase}
\begin{document}
%
%

\title{Universal Local Parametrizations via Heat Kernels and Eigenfunctions of the Laplacian}

\author{
    Peter W Jones \footnote{Department of Mathematics, Yale University, 10 Hillhouse Ave, New Haven, CT, 06510, U.S.A., +1-(203)-432-1278}
 \and
    Mauro Maggioni \footnote{Department of Mathematics, Duke University, BOX 90320, Durham, NC, 27708, U.S.A., +1-(919)-660-2825}
 \and
    Raanan Schul \footnote{Department of Mathematics, UCLA, Box 951555 Los Angeles, CA 90095-1555, U.S.A., +1-(310)-825-3855}
 }

\maketitle

\begin{abstract}
We use heat kernels or eigenfunctions of the Laplacian to construct local coordinates
on large classes of Euclidean domains and Riemannian manifolds (not necessarily smooth, e.g. with $\mathcal{C}^\alpha$ metric).
These coordinates are bi-Lipschitz on embedded balls of the domain or manifold, with
distortion constants   that depend only on natural geometric properties of the domain or manifold.
The proof of these results relies on  estimates, from above and below, for the heat kernel and its gradient,
as well as for the eigenfunctions of the Laplacian and their gradient.  These estimates  hold in the non-smooth category,
and are stable with respect to perturbations within this category.
Finally, these coordinate systems are intrinsic and efficiently computable, and are of value in applications.
\end{abstract}


\tableofcontents

\section{Introduction}

The concept of a coordinate chart for a manifold is quite old, but it has only recently become a subject of intensive study for data sets.
In this paper we will state and prove a new theorem for coordinate charts on Riemannian manifolds.
This result is meant to explain the empirically observed robustness of certain coordinate charts for data sets,
Before stating our results, we  explain in more detail the setting, first for manifolds, and then for data sets.

Let $\cM$ be a Riemannian manifold.
A coordinate chart (more precisely, a restriction of one) can be viewed as a mapping from a  metric ball
$B\subset\cM$ into $\R^d$, where $d$ is the topological dimension of $\cM$.
This mapping has the form
\begin{equation*}
F(x)=\big(f_1(x), f_2(x),...,f_d(x)\big)\,.
\end{equation*}
It is natural to ask for $F$ to have low distortion.
Let $F(B)=\tilde B\subset \R^d$.
By assumption $F$ is a one to one mapping from $B$ to $\tilde B$.
The Lipschitz norm of $F$ is defined as
\begin{equation*}
\|F\|_{Lip}=\sup\limits_{x,y \in B \atop x\neq y} \frac{\|F(x)-F(y)\|}{d_{\cM}(x,y)}
\end{equation*}
where $d_{\cM}(\cdot,\cdot)$ is the metric on $\cM$ and
$\|\cdot\|$ is the usual Euclidean metric on $\R^d$.
Similarly, one sets
\begin{equation*}
\|F^{-1}\|_{Lip}=\sup\limits_{ x,y \in B \atop x\neq y} \frac{d_{\cM}(x,y)}{\|F(x)-F(y)\|}\,.
\end{equation*}
Then the distortion of $F$ on $B$ is defined to be
\begin{equation}
\distortion(F,B):=\|F\|_{Lip}\times \|F^{-1}\|_{Lip}\,.
\end{equation}

It is worth recalling at this point a prime example of a coordinate chart, namely the coordinate chart on a simply connected planar domain $\mathcal{D}$ given by a Riemann mapping $F$ from $\mathcal{D}$ to the unit disc $\mathbb{D}$.
Let $z_0\in \mathcal{D}$ and define $r=\dist(z_0,\partial \mathcal{D})$.
If we choose our Riemann map $F$ to satisfy $F(z_0)=0$, then
the distortion theorems of classical complex analysis
(see e.g. \cite{pommerenke} page 21)
state that $F$ maps the disc
$\ball(z_0,\frac{r}2)$ onto "almost" the unit disc, with low distortion:
\begin{equation}\label{e:distortion-estimates-1}
\ball(0,\thmConst^{-1})\subset F\big(\ball(z_0,\frac{r}2)\big) \subset \ball(0,1-\thmConst^{-1})\,,
\end{equation}
\begin{equation}\label{e:distortion-estimates-2}
\distortion\big(F,\ball(z_0,\frac{r}2)\big)\leq \thmConst\,.
\end{equation}
In other words, on $\ball(z_0,\frac{r}2)$, $F$ is a perturbation (in the proper sense) of the linear map
given by $z\to F'(z_0)(z-z_0)$, and $|F'(z_0)| \sim \frac1r$.

In this paper we will look for an analogue of \eqref{e:distortion-estimates-1} and \eqref{e:distortion-estimates-2} above, but in the setting of  Riemannian manifolds.
We will show that on Riemannian manifolds of finite volume there is a {\bf locally} defined $F$ that
has these properties, and that this choice of $F$ will come from {\bf globally}  defined Laplacian eigenfunctions.
On a metric embedded ball $B\subset \cM$ we will choose {\bf global} Laplacian eigenfunctions
$\globalEF_{i_1}, \globalEF_{i_2},...,\globalEF_{i_d}$ and constants
$\gamma_1,\gamma_2,...,\gamma_d\leq \thmConst$ (for a universal constant $\thmConst$) and define
\begin{equation}
\Phi:=\big( \gamma_1 \globalEF_{i_1}, \gamma_2\globalEF_{i_2},...,\gamma_d\globalEF_{i_d}\big)\,.
\end{equation}
This choice of $\Phi$, depending heavily on $z_0$ and $r$, is globally defined,
and on $\ball(z_0,\thmConst^{-1} r)$ enjoys the same properties as the Riemann map does in \eqref{e:distortion-estimates-1} and \eqref{e:distortion-estimates-2}.
In other words, $\Phi$ maps $\ball(z_0,\thmConst^{-1} r)$ to, roughly, a ball of unit size, with low distortion.
Here we should point out the 1994 paper of
B\'erard et al.\cite{BBG:EmbeddingRiemannianManifoldHeatKernel} where a weighted infinite sequence of eigenfunctions is shown to provide a global coordinate system (points in the manifold are mapped to $\ell_2$). To our knowledge this was the first result of this type in Riemannian geometry.
Our results can be viewed as a strengthening of their work, and have as a consequence the statement that for a compact manifold without boundary, a good global coordinate system is given by the eigenfunctions $\globalEF_j$ with eigenvalues $\globalEV_j<\thmConst R_{\rm inj}^{-2}$.
Here $R_{\rm inj}$   is the inradius of $\cM$, i.e. the largest $r>0$ such that for all $x\in \cM$,  $\ball(x,r)$ is an embedded ball.

The impetus for this paper and its results comes from certain recent results in the analysis of data sets.
A recurrent idea is to approximate a data set, or a portion of it, lying in high dimensional space, by a manifold of low dimension, and find
a parametrization of such data set or manifold. This process sometimes goes under the name of manifold learning, or linear or nonlinear dimensionality reduction.
This type of work has been in part motivated by spectral graph theory \cite{Chung} and spectral geometry \cite{Cheeger70,GHL_RiemannianGeometry,Donnelly-Fefferman} (and references therein).
Let $\{x_j\}_1^N$ be a collection of data points in a  metric space $\mathbb{X}$.
It is frequently quite difficult to extract any information from the data as it is presented.
One solution is to embed the points in $\R^n$ for $n$ perhaps quite large, and then use linear
methods (e.g. those using singular value decomposition) to obtain a dimensional reduction of the data set.
In certain situations however linear methods are insufficient.
For this reason, there has recently been great interest in nonlinear methods
\footnote{
Examples of such disparate applications include document analysis \cite{CMDiffusionWavelets}, face recognition \cite{niyogi},
clustering \cite{ng01spectral,BN:SpectralTechniquesEmbeddingClustering},
machine learning \cite{NB,NMB,SMC:GeneralFrameworkAdaptiveRegularization,smmm:ValueFunction,smkfsomm:SimLearningReprControlContinuous,MSCB:MultiscaleManifoldMethods},
nonlinear image denoising and segmentation \cite{shi-malik:pami,SMC:GeneralFrameworkAdaptiveRegularization},
processing of articulated images \cite{DG_HessianEigenmaps}, cataloguing of galaxies \cite{donoho2}, pattern analysis of brain potentials
\cite{pclo} and EEG data \cite{Shen:fMRIDiffusionMaps}, and the study of brain tumors \cite{ganesh}.
A variety of algorithms for manifold learning have been proposed \cite{RSLLE,BN,NB,LThesis,CLAcha1,DiffusionPNAS,DiffusionPNAS2,isomap,ZhaZha,DG_HessianEigenmaps,Weinberger:MaximumVarianceUnfolding,ZhaZha,Saul:LearningKernelMatrixNonlineadDimReduction,TSL,Saul:AnalysisExtensionSpectralMethods,Saul:SpectralMethodsDimReduction}.
}.
Unfortunately such techniques seldomly come with guarantees on their capabilities of indeed finding local parametrization (but see, for example,
\cite{DoGri:WhenDoesIsoMap,DG_HessianEigenmaps,TSL}), or on quantitative statements
on the quality of such parametrizations.

One of these methods, diffusion geometry, operates by first defining a kernel $K(x_j,x_k)$
on the data set, and then altering this slightly to obtain a self-adjoint matrix $(m_{j,k})$ that roughly corresponds to the generator of a diffusion process.
The eigenvectors of the matrix, should be seen as corresponding to Laplacian eigenfunctions on a  manifold.
One (judiciously) selects a collection $v_{i_1}, v_{i_2},...,v_{i_m}$ of eigenvectors and maps
\begin{equation}
x_k\to (v_{i_1}, v_{i_2},...,v_{i_m})\in\R^m
\end{equation}
Careful choices of collections of eigenvectors have been empirically observed to give excellent representations
of the data in a  very low dimensional Euclidean space. What has been  unclear is why this method should prove so successful.
Our results show that in the case of Riemannian manifolds, one can prove that this philosophy is not just correct,
but also robust.
It is to be said that researchers so far have restricted their attention to the case when the lowest frequency
eigenfunctions are selected, i.e. $i_1=1,i_2=2,\dots,i_m=m$
\cite{Spielman:SpectralPartitioningWorks,BNeigenmaps,NB,DiffusionMaps,DiffusionPNAS,CKLMN:DiffusionMapsReductionCoordinates}.

Given these results, it is plausible to guess that an analogous result should hold for a local piece of a data set if that piece has in some sense a ``local dimension'' approximately $d$.
There are certain difficulties with this philosophy.
The first is that graph eigenfunctions are global objects and any definition of  ``local dimension'' may change from point to point in the data set.
A second difficulty is that our results for manifolds depend on classical estimates for eigenfunctions.
This smoothness may be lacking in graph eigenfunctions.

It turns out that another of our manifold results does not suffer from these serious problems when working on a data set.
We introduce simple ``heat coordinate'' systems on manifolds.
Roughly speaking (and in the language of the previous paragraph) these are $d$ choices of manifold heat kernels that form a robust coordinate system on $\ball(z_0,\thmConst^{-1}r)$.
We call this method ``heat triangulation''  in analogy with triangulation as practiced in surveying, cartography, navigation, and modern GPS.
Indeed our method is a simple translation of these classical triangulation methods, and has a closed formula on $\R^d$, which we note has infinite volume! (Our result on heat kernels  makes no assumptions on the volume of the manifold.)
For data sets, {\it heat triangulation} is a much more stable object than  eigenfunction coordinates because:
\begin{itemize}
\item Heat kernels are local objects (see e.g. Proposition \ref{p:LocalToGlobalHeatKernel})
\item If a manifold $\mathcal{M}$ is approximated by discrete sets $X$, the corresponding graph heat kernels converge rather nicely to the manifold heat kernel.
This is studied for example in \cite{LThesis,CLAcha1,CLAcha2,BN}.
\item One has good statistical control on smoothness of the heat kernel, simply because one can easily examine it and because one can use the Hilbert space
$\{f\in L^2:\grad f\in L^2\}$.
\item Our results that use eigenfunctions rely in a crucial manner on Weyl's Lemma, whereas  heat kernel  estimates do not.
\end{itemize}
In a future paper we will return to applications of this method to data sets.

The philosophy used in this paper is as follows.
\begin{itemize}
\item[Step 1.] Find suitable points $y_j$, $1\leq j\leq d$ and a time $t$ so that the mapping given by heat kernels
$( x\to K_t(x,y_1),..., K_t(x,y_d) )$ is a good local coordinate system on $\ball(z,\thmConst^{-1}r)$. (This is heat triangulation.)
\item[Step 2.] Use Weyl's Lemma to find suitable eigenfunctions $\globalEF_{i_j}$ so that
(with $K_j(x)=K_t(x,y_j)$) one has large gradient.
\end{itemize}
Each point $y\in \cM$ gives rise to a heat kernel $K_t(x,y)$.
One may think of Step 1 as sampling this family of heat kernels $K_t(x,y)$ at
$d$ different choices $y_1,...,y_d$.
Indeed, with high probability, randomly chosen points from the appropriate annulus will be suitable.
Step 2 corresponds to sampling the vector $\{\globalEF_j(x)e^{\lambda_jt}\}_j$ $d$ times, once for each point
$y_1,...,y_d$.
This last sampling, where we choose an index $j$, cannot be performed randomly! (See example in section \ref{s:localized-ef}).

At this point we would like to note an advantage that local parametrization by eigenfunctions has over heat kernel triangulation (which we do not discuss in this paper).
Consider  the planar domain $[0,3\epsilon]\times [0,3]$. Then, using only two Neumann eigenfunctions, one gets a good parametrization of  the rectangle
$[\epsilon,2\epsilon]\times[1,2]$.
On the other hand, in order to get parametrization of similar distortion using heat kernel triangulation, on needs to use $\sim \frac1\epsilon$ different heat kernels.

To see where our philosophy comes from, we return for a moment to the setting of a simply connected planar  domain $\cD$ of area $=1$.
Let $z_0\in \cD$ and $r$ be as in the discussion before equation \eqref{e:distortion-estimates-1}.
With the choice of Riemann mapping $F$, with $F(z_0)=0$ we have  the classical formula known to Riemann:
\begin{equation}
F(z)={\rm exp}\big\{-G(z,z_0)- iG^*(z,z_0)\big\}\,.
\end{equation}
Here $G(\cdot,z_0)$ is Green's function for the domain $\cD$, with pole at $z_0$, and $G^*$ is the multivalued conjugate of $G$.
Thus, all information about $F$ on $\ball(z_0,\frac{r}2)$ is encoded in $G(z,z_0)$.
Recall that
\begin{equation}
G(z,z_0)=\int\limits_0^\infty K(z,z_0,t)dt\,,
\end{equation}
where $K$ is the (Dirichlet) heat kernel for $\cD$.
Thus the behavior of $F$ can be read off the information on $K(z,z_0,t)$.
Now write,
\begin{equation}
K(z,z_0,t)=\sum\limits_{j=1}^\infty \globalEF_j(z) \globalEF_j(z_0) e^{\globalEV_jt}
\end{equation}
where $\{\globalEF_j\}$ is the collection of Dirichlet eigenfunctions (normalized to have $L^2$ norm $=1$) and $\Delta \globalEF_j= \globalEV_j\globalEF_j$.
Notice that
\begin{equation}
|F'(z)|=|\grad G(z,z_0)|e^{-G(z,z_0)}.
\end{equation}
Since $|F'(z)|\sim \frac1r$ on $\ball(z_0,\frac{r}2)$ it is reasonable to guess from the above identities
that there are eigenfunctions $\globalEF_j$ such that
\begin{equation}
|\grad \globalEF_j|\gtrsim \frac1r
\end{equation}
on $\ball(z_0,\thmConst^{-1} r)$, for some $\thmConst>1$, independent of $\cD$.
(More precisely, a short calculation with Weyl's estimates makes this reasonable.)
This simple reasoning turns out to be correct and the main idea of this paper.
The proof does not depend on any properties of holomorphic functions, but runs with equal ease in any dimension.
This is because it  only requires estimates on the heat kernel, Laplacian eigenfunctions and their derivatives, all of which are real variable objects.

\medskip

The paper is organized in a top-bottom fashion, as follows. In Section \ref{s:results} we state the main results, in Section \ref{s:proofs-of-EF-thms} we present the main Lemmata,
the proofs of the main results, and important estimates on the heat kernel and eigenfunctions of the Laplacian, together with their proofs, but mostly only in the Euclidean case. For the purpose of completeness we have recorded here proofs of several known estimates, over which the experts may wish to skip. In Section \ref{s:supplementalmanifoldcase} we present the material for generalizing most estimates to the manifold case.
Finally, we discuss some examples in Section \ref{s:examples}.
We include a {\it Table of notation} at the end of the manuscript, see Section \ref{appendix-table-notation}.

\section{Results}
\label{s:results}
\allowdisplaybreaks

\subsection{Euclidean domains}
We first present  the case of Euclidean domains.
While our results in this setting follow from the more general results for manifolds discussed in the next section,
the case of Euclidean domains is of independent interest, and the exposition of the main result as well as the proof in this case is simpler in the several technical respects.

We consider the heat equation in $\Omega$, a finite volume domain in $\mathbb{R}^d$, with either Dirichlet or Neumann boundary conditions i.e., respectively,
\begin{gather*}
\begin{array}{ccc}
\begin{cases}
(\Delta -{\partial\over\partial t})u(x,t)=0\\
u|_{\partial\Omega} = 0
\end{cases}
&
\textrm{ or }
&\begin{cases}
(\Delta -{\partial\over\partial t})u(x,t)=0\\
\partial_\nu u|_{\partial\Omega} = 0
\end{cases}.
\end{array}
\end{gather*}
Here  $\nu$ is the outer normal on $\partial\Omega$.
Independently of the boundary conditions, $\Delta$ denotes the Laplacian on $\Omega$.
In this paper we restrict our attention to domains where the spectrum is discrete and the corresponding heat kernel can be written as
\begin{gather}\label{e:heat-kernel-spectral-expansion}
K_t(z,w)=K^\Omega_t(z,w)=\sum_{j=0}^{+\infty}\globalEF_j(z)\globalEF_j(w)e^{-\globalEV_j t}\,.
\end{gather}
where the $\{\globalEF_j\}$ form an orthonormal basis of eigenfunctions of $\Delta$,  with eigenvalues $0\le\globalEV_0\le\dots\le\globalEV_j\le\dots$.
We also require a (non-asymptotic)  Weyl-type  estimate:
there is a constant $\Cweyl$
such that for any $T>0$
\begin{equation}\label{e:weyl-bounds-omega}
\#\{j:0<\globalEV_j\leq T\}\leq \Cweyl T^\frac{d}2|\Omega|\,.
\end{equation}
In the Dirichlet case  $\Cweyl$ does not depend on $\Omega$ (see remark \ref{r:Weyl-constant}).
For the Dirichlet case the only substantial problem is that the eigenfunctions may fail to vanish at the boundary.  This in turn only occurs if there are boundary points where the Wiener series (for the boundary) converges \cite{Wiener:Series,Kellogg:PotentialTheory}.
For the Neumann case the situation is more complicated
\cite{Safarov:WeylAsymptoticFormula,
Simon:NeumannEssentialSpectrum,
Netrusov:SharpRemainderEstimatesNeumannPlanar}.
In particular there are  domains with arbitrary closed continuous Neumann spectrum \cite{Simon:NeumannEssentialSpectrum}.
We therefore restrict ourselves in this paper to domains (and, later, manifolds) where conditions
\eqref{e:heat-kernel-spectral-expansion} and \eqref{e:weyl-bounds-omega} are valid.
More general boundary conditions can be handled in similar fashion, since our analysis is local and depends on the boundary conditions only through
the properties above.

Finally,  here and throughout the manuscript, we define $\fint\limits_B f :=\frac{1}{\vol{B}}\int\limits_B f$.


\begin{theorem}[Embedding via Eigenfunctions, for Euclidean domains]
\label{t:datheorem}
Let $\Omega$ be a finite volume domain in $\mathbb{R}^d$, rescaled so that $|\Omega|=1$.
Let $\Delta$ be the Laplacian in $\Omega$, with Dirichlet or Neumann boundary conditions, and assume that \eqref{e:heat-kernel-spectral-expansion}
and \eqref{e:weyl-bounds-omega} hold.
Then is a  constant $\thmConst>1$  that depends only on $d$
such that the following hold.

For any $z\in\Omega$, let $\rho\le \mathrm{dist}\,(z,\partial\Omega)$.
Then there exist integers $i_1,\dots,i_d$ such that,
if we let $$\gamma_{l}=\left(\fint\limits_{\ball(z,\thmConst^{-1}\rho)}\varphi_{i_l}^2\right)^{-\frac12}\ ,\  l=1,\dots,d\,,$$ we have that:
\begin{itemize}
\item[(a)] the map
\begin{eqnarray}
\Phi : \ball(z,\thmConst^{-1}\rho) &\rightarrow& \mathbb{R}^d\\
x&\mapsto& (\gamma_1\globalEF_{i_1}(x),\dots,\gamma_d\globalEF_{i_d}(x))
\end{eqnarray}
satisfies, for any $x_1,x_2\in\ball(z,\thmConst^{-1}\rho)$,
\begin{equation}
\frac{\thmConst^{-1}}{\rho}||x_1-x_2||\le||\Phi(x_1)-\Phi(x_2)||\le \frac{\thmConst}{\rho}||x_1-x_2||\,;
\label{e:bilip_bound_euclidean}
\end{equation}
\item[(b)] the associated eigenvalues satisfy $$\thmConst^{-1} \rho^{-2}\le \globalEV_{i_1},\dots,\globalEV_{i_d}\le \thmConst \rho^{-2}\,;$$
\item[(c)] the constants $\gamma_l$ satisfy $$\gamma_1,\dots,\gamma_d\le \thmConst\, (\Cweyl)^\frac12.$$
\end{itemize}
\end{theorem}

\begin{remark}
In item (c) above, it will also be the case that $\thmConst^{-1} \rho^{\frac d2}\le \gamma_j$.
\end{remark}

\begin{remark}
The dependence on $\Cweyl$ is only needed in the Neumann case because, unlike the Dirichlet case, the upper bound in Weyl's Theorem depends on the domain. See Remark \ref{r:Weyl-constant} for a more precise statement.
\end{remark}

\subsection{Manifolds with $\mathcal{C}^{\alpha}$ metric}
\label{s:manifold}

The results above can be extended to certain classes of manifolds.
In order to formulate a result corresponding to Theorem \ref{t:datheorem} we must first carefully define the manifold analogue of $\dist(z,\partial\Omega)$.
Let $\mathcal{M}$ be a smooth, $d$-dimensional compact manifold,
possibly with boundary. Suppose we are given a metric tensor $g$ on
$\mathcal{M}$ which is $\mathcal{C}^\alpha$ for some $\alpha\in(0,1]$.
For any $z_0\in\mathcal{M}$, let $(U,\chartMap)$ be a coordinate chart such that $z_0\in U$ and normalized so that
\begin{itemize}
\item[(i)] $g^{il}(\chartMap(z_0))=\delta^{il}$.
\end{itemize}
Then we assume that
\begin{itemize}
\item[(ii)] for any $x\in U$, and any $\xi,\nu\in\mathbb{R}^d$,
\begin{equation}
c_{\min}(g)||\xi||_{\mathbb{R}^d}^2\le \sum_{i,j=1}^d g^{ij}(\chartMap(x))\xi_i \xi_j\,\,\,\,\mathrm{and }\,\,\,\,\sum_{i,j=1}^d g^{ij}(\chartMap(x))\xi_i\nu_j\le c_{\max}(g)||\xi||_{\mathbb{R}^d}\,||\nu||_{\mathbb{R}^d}\,.
\label{e:uniformlyelliptic}
\end{equation}
\end{itemize}
We let
\begin{equation}
\injRad(z_0) = \sup \{r>0 : B_r(\chartMap(z_0))\subseteq \chartMap(U)\}\,.
\label{e:rM}
\end{equation}
Observe that, when $g$ is at least $\mathcal{C}^2$,  $\injRad$ can be taken to be less than the inradius, with local coordinate chart $(U,\chartMap)$  given by the exponential map at $z$.
The chart $(U,\chartMap)$ may intersect the boundary with no consequence, as all of the work will be done inside  $\ball(z_0,\injRad)$.
We denote by $\|g\|_{\alphawedgeone}$ the maximum over all $i,j$ of
$$\sup\limits_{x\neq y} \frac{|g^{ij}(\chartMap(x))-g^{ij}(\chartMap(y))|}{|\chartMap(x)-\chartMap(y)|^{\alphawedgeone}}$$
for $x,y$ in $U$.
The natural volume measure $d\mu$ on the manifold is given, in any local chart, by $\sqrt{\mathrm{det}\,g}\,$; conditions \eqref{e:uniformlyelliptic} guarantee in particular that $\mathrm{det}g$ is uniformly bounded below from $0$.
Let $\Delta_{\mathcal{M}}$ be the Laplace Beltrami operator on $\mathcal{M}$. In a local chart, we have
\begin{equation}
\Delta_{\mathcal{M}} f(x) = -\frac{1}{\sqrt{\mathrm{det}\,g}}\sum_{i,j=1}\partial_j\left(\sqrt{\mathrm{det}\,g}\,g^{ij}(\chartMap(x))\partial_i f\right)(\chartMap(x))\,,
\label{e:DeltaM}
\end{equation}
when $g$ is smooth enough (e.g. $g\in\mathcal{C}^1$). In general one defines the Laplacian through its associated quadratic form \cite{DaviesHeatKernels,Davies:SpectraPropertiesChangesMetric}.
Conditions \eqref{e:uniformlyelliptic} are the usual uniform ellipticity conditions for the operator \eqref{e:DeltaM}.
With Dirichlet or Neumann boundary conditions, $\Delta_{\mathcal{M}}$ is self-adjoint on $L^2(\mathcal{M},\mu)$.  We will assume that the spectrum is discrete, denote by $0\le\globalEV_0\le\dots\le\globalEV_j\le$ its eigenvalues and by $\{\globalEF_j\}$ the corresponding orthonormal basis of eigenfunctions, and write equations \eqref{e:heat-kernel-spectral-expansion} and \eqref{e:weyl-bounds-omega} with $\Omega$ replaced by $\mathcal{M}$.


\begin{theorem}[Embedding via Eigenfunctions, for Manifolds]
\label{t:datheoremmanifold}
Let $(\mathcal{M},g)$, $z\in\mathcal{M}$ be a $d$ dimensional manifold and $(U,\chartMap)$ be a chart as above. Assume $|\mathcal{M}|=1$.
There is a  constant $\thmConst>1$,  depending on $d$, $c_{\min}$, $c_{\max}$, $||g||_{\alphawedgeone}$, ${\alphawedgeone}$,
such that the following hold.

Let $\rho\leq\injRad(z)$.
Then there exist integers $i_1,\dots,i_d$ such that  if we let
$$\gamma_{l}=\left(\fint\limits_{\ball(z,\thmConst^{-1}\rho)}\varphi_{i_l}^2\right)^{-\frac12}\ ,\  l=1,\dots,d\,,$$
we have that:
\begin{itemize}
\item[(a)] the map
\begin{eqnarray}
\Phi : \ball(z,\thmConst^{-1}\rho) &\rightarrow& \mathbb{R}^d\\
x&\mapsto& (\gamma_1\globalEF_{i_1}(x),\dots,\gamma_d\globalEF_{i_d}(x))
\end{eqnarray}
satisfies for any $x_1,x_2\in\ball(z,\thmConst^{-1}\rho)$
\begin{equation}
\frac{\thmConst^{-1}}{\rho}\,d_{\mathcal{M}}(x_1,x_2)\le ||\Phi(x_1)-\Phi(x_2)||\le \frac{\thmConst}{\rho}\,d_{\mathcal{M}}(x_1,x_2)\,.
\label{e:bilip_bound_manifold}
\end{equation}
\item[(b)] the associated eigenvalues satisfy $$\thmConst^{-1}\rho^{-2}\le \globalEV_{i_1},\dots,\globalEV_{i_d}\le \thmConst\rho^{-2}\,.$$
\item[(c)] the constants $\gamma_l$ satisfy $$\gamma_1,\dots,\gamma_d\le \thmConst (\Cweyl)^\frac12\,.$$
\end{itemize}
\end{theorem}

\begin{remark}
As in the Euclidean case, in item (c) above, it will also be the case that $\thmConst^{-1} \rho^{\frac d2}\le \gamma_j$
\end{remark}

\begin{remark}\label{r:rem-1}
Most of the proof is done on the local chart $(U,\chartMap)$ containing $z$.
An inspection of the proof shows that we use only the norm $\|g\|_{\alphawedgeone}$ of the $g$ restricted to this chart.
\end{remark}

\begin{remark}\label{r:hoelder-norm}
When rescaling Theorem \ref{t:datheoremmanifold}, it is important to note that if $f$ is a H\"older function with $\|f\|_{\mathcal{C}^{\alphawedgeone}}=A$ and
$f_r(z)=f(r^{-1}z)$ then $\|f_r\|_{\mathcal{C}^{\alphawedgeone}}=Ar^{\alphawedgeone}$.
Since we will have $r<1$, $f_r$ satisfies a better H\"older estimate then $f$, i.e.
$$\|f_r\|_{\mathcal{C}^{\alphawedgeone}}=Ar^{\alphawedgeone}< A=\|f\|_{\mathcal{C}^{\alphawedgeone}}\,.$$
We will repeatedly use this observation when discussing manifolds with $\mathcal{C}^\alpha$ metric.
\end{remark}

\begin{remark}
We do not know, in both Theorem  \ref{t:datheorem} and Theorem \ref{t:datheoremmanifold}, whether it is possible to choose eigenfunctions such that $\gamma_1\sim\gamma_2\sim...\sim\gamma_d$.
If this were so, the map
$x\mapsto (\globalEF_{i_1}(x),\dots,\globalEF_{i_d}(x))$
would be a low distortion map whose image has diameter $\geq \thmConst^{-1}$.
\end{remark}

\begin{remark}
As was noted by L. Guibas, when $\cM$ has a boundary, in the case of Neumann boundary values,
one may consider the  ``doubled" manifold, and may apply our result for a possibly larger  $\injRad(z)$.
\end{remark}

Clearly Theorem \ref{t:datheorem} is a particular case of Theorem \ref{t:datheoremmanifold}, but the proof of the former is significantly easier in that one can use standard estimates on eigenfunctions of the Laplacian and their derivatives.
For the sake of presentation we present one proof for both Theorems, but two sets of required Lemmata for those estimates which are significantly different
in the two cases.

\begin{remark}
The method of the proofs also gives a result independent of the constant $\Cweyl$:
Let $(\mathcal{M},g)$ and $z\in\mathcal{M}$   be  as in
Theorem \ref{t:datheoremmanifold}.
Let $\eta>0$, and assume that for any $x\in \mathcal{M}$ we have a chart $(U,\chartMap)$
such that $\injRad(x)\geq \eta>0$
(in particular, $\mathcal{M}$ has no boundary).
Then for $\rho\leq \eta$ the same results as in Theorem \ref{t:datheoremmanifold} hold, except the   constant $\thmConst$ depends only on  $d$, $c_{\min}$, $c_{\max}$, $||g||_{\alphawedgeone}$, ${\alphawedgeone}$ and not on  $\Cweyl$.
This is due to the fact that $\Cweyl$ becomes universal for values of $T>\eta^{-2}$.
\end{remark}

Another, in some sense stronger, result is true.  One may replace the $d$ eigenfunctions in Theorem \ref{t:datheoremmanifold} by $d$ heat kernels $\{K_t(z,y_i)\}_{i=1,...,d}$.
In fact such heat kernels arise naturally in the main steps of the proofs of Theorem  \ref{t:datheorem} and Theorem \ref{t:datheoremmanifold}.
This leads to an embedding map with even stronger guarantees:

\begin{theorem}[Heat Triangulation Theorem]
\label{t:heatkernelmapping}
Let $(\mathcal{M},g)$, $z\in\mathcal{M}$ and $(U,\chartMap)$ be as above, with the exception we now make no assumptions on the finiteness of the volume of $\cM$ and the existence of $\Cweyl$.
Let $\rho\leq \injRad(z)$.
Let $p_1,...,p_d$ be $d$ linearly independent directions.
There are constants $c>0$ and $c',\thmConst>1$, depending on
$d$, $c_{\min}$, $c_{\max}$, $\rho^{\alphawedgeone}||g||_{\alphawedgeone}$, ${\alphawedgeone}$,
and the smallest and largest eigenvalues of the
Gramian matrix $(\langle p_i,p_j\rangle)_{i,j=1,\dots,d}$, such that the following holds.
Let $y_i$ be so that $y_i-z$ is in the direction $p_i$, with $c\rho\le d_\mathcal{M}(y_i,z)\le 2c \rho$ for each $i=1,\dots,d$
and let $t=\thmConst^{-1}\rho^2$.
The map
\begin{equation}
x\mapsto (\rho^d K_{t}(x,y_1)),\dots,\rho^d K_{t}(x,y_d))
\label{e:heatkernelmapping}
\end{equation}
satisfies, for any $x_1,x_2\in\ball(z,\thmConst^{-1}\rho)$,
\begin{equation}
\frac{\thmConst^{-1}}{c'\rho}\,d_{\mathcal{M}}(x_1,x_2) \le ||\Phi(x_1)-\Phi(x_2)||\le \frac{\thmConst c'}{\rho}\,d_{\mathcal{M}}(x_1,x_2)\,.
\label{e:heatkerneldistortionestimate}
\end{equation}
\end{theorem}

\noindent The reason for the factor $\rho^{\alphawedgeone}$ which we have in  $\rho^{\alphawedgeone}||g||_{\alphawedgeone}$ above is to get scaling invariance.

This theorem holds for the manifold and Euclidean case alike, and depends only on the heat kernel estimates (and its gradient).
We again note that for this particular Theorem we require no statement about the volume of the manifold, the existence of $L^2$ Laplacian eigenfunctions, or their number.
The constants for the Euclidean case, depend only on dimension, and not on the domain.
The content of this theorem is that one is able to choose the directions $y_i-z$ randomly on a sphere, and with high probability on gets a  low distortion map.  This gives rise to a sampling theorem.

One may replace the (global) heat kernel above with a local heat kernel, i.e. the heat kernel for the ball $\ball(z,\rho)$ with the metric induced by the manifold and Dirichlet boundary conditions.  In fact, this is a key idea in the proof of all of the above Theorems.  Thus, on the one hand our results are local, i.e. independent of the global geometry of the manifold, yet on the other hand they are in terms of {\it{global}} eigenfunctions.

As  is clear from the proof, all theorems  hold for more general boundary conditions. This is especially true for the Heat Triangulation Theorem, which does not even depend on the existence of a spectral expansion for the heat kernel.

\begin{example}
It is a simple matter to verify this Theorem for the case where the manifold in $\R^d$.  For example if $d=2$, $\rho=1$, and $z=0$,  $y_1=(-1,0)$ and  $y_2=(0,-1)$.  Then if $K_t(x,y)$ is the Euclidean heat kernel,
$$x\to (K_1(x,y_1),\ K_1(x,y_2))$$
is a (nice) biLipschitz map on $\ball\big((0,0),\ \frac12)$. (The result for arbitrary radii then follows from a scaling argument).
This is because on can simply evaluate the heat kernel
$$K_t(x,y)=\frac1{4\pi t}e^{-\frac{|x-y|^2}{4t}}\,.$$
So in $\ball_\frac12((0,0))$
$$\grad K_1(x,y_1)\sim \frac1{2\pi}e^{-\frac{1}{4}}(1,0)\,\,\,\mathrm{\ and\ }\,\,\,\grad K_1(x,y_2)\sim \frac1{2\pi}e^{-\frac{1}{4}}(0,1)\,.$$
\end{example}

\noindent{\large{\bf{Acknowledgments}}}\\
\noindent The authors would like to thank K. Burdzy, R.R. Coifman, P. Gressman, N. Tecu, H. Smith and A.D. Szlam for useful discussions during the preparation of the manuscript, as well as IPAM for hosting these discussions and more.
P. W. Jones is grateful for partial support from NSF DMS 0501300.
M. Maggioni is grateful for partial support from NSF DMS 0650413, NSF CCF 0808847 and ONR N00014-07-1-0625.
R. Schul is grateful for partial support from NSF DMS 0502747.
The main theorems were reported in the announcement \cite{jms:UniformizationEigenfunctions}.

\section{The Proof of Theorems \ref{t:datheorem} and  \ref{t:datheoremmanifold}}
\label{s:proofs-of-EF-thms}

The proofs in the Euclidean and manifold case are similar.
In this section we present the steps of the proofs of Theorems \ref{t:datheorem}, \ref{t:datheoremmanifold}
we will  postpone the technical estimates needed to later sections.

Because we may change base points, we will use $R_z$ (or similarly, $R_w$) in place of $\rho$.
We will also interchange between $\ball(x,r)$ and $\ball_r(x)$.

\begin{remark}[Some remarks about the Manifold case]
\begin{itemize}
\item[(a)]
As mentioned in Remark \ref{r:rem-1}, we will often restrict to working on a single (fixed!) chart in local coordinates.
When we discuss moving in a direction $p$, we mean in the local coordinates.
\item[(b)]
We will use Brownian motion arguments (on the manifold).  In order to have existence and uniqueness one needs smoothness assumptions on the metric (say, $\mathcal{C}^2$, albeit less would suffice, see e.g. \cite{Oksendal-book}).
Therefore we will first prove the Theorem in the manifold case in the $\mathcal{C}^2$ metric category, and then use perturbation estimates to obtain the result for $g\in \mathcal{C}^{\alpha}$.  To this end, we will often have dependence on $||g||_{\alpha}$ even though we will be (for a specific Lemma or Proposition) assuming the $g\in\mathcal{C}^2$.
\end{itemize}
\label{r:rem-initial-remarks}
\end{remark}

\smallskip\noindent
{\textbf{Notation.}}
\begin{itemize}
\item In what follows, we will write
$f(x)\lesssim_{c_1,\dots,c_n} g(x)$ if there exists a constant $C$
depending only on $c_1,\dots,c_n$, and not on $f,g$ or $x$, such
that $f(x)\le C g(x)$ for all $x$ (in a specified domain).
We will write $f(x)\sim_{c_1,\dots,c_n} g(x)$ if both $f(x)\lesssim_{c_1,\dots,c_n} g(x)$ and $g(x) \lesssim_{c_1,\dots,c_n} f(x)$.
If $f,g$ take values in $\mathbb{R}^d$ the inequalities are intended componentwise.
We will write $a\sim_{C_1}^{C_2} b$ if $C_1 b\le a\le C_2 b$ (componentwise for $a,b$ vectors).
\item
In what follows we will write
$\partial_\dir K_t(\cdot,\cdot)$ to denote the partial derivative with respect to the {\it second} variable of  a heat kernel at time $t$.
\end{itemize}

\bigskip

\subsection{The Case of $g\in\mathcal{C}^2$.}
\label{s:proofs-of-EF-C2-case}
We note that even though we assume $g\in\mathcal{C}^2$, we only use the
$\mathcal{C}^{\alphawedgeone}$ norm of $g$.
The idea of the proof of Theorems \ref{t:datheorem} and \ref{t:datheoremmanifold} is as follows.
We start by fixing a direction $p_1$ at $z$.
We would like to find an eigenfunction $\globalEF_{i_1}$ such that $\left|\partial_{p_1}\globalEF_{i_1}\right|\gtrsim R_z^{-1}$  on $\ball_{c_1R_z}(z)$.
In order to achieve this, we start by showing that the heat kernel has large gradient in an annulus  of inner and outer radius $\sim R_z^{-1}$ around $y_1$ ($y_1$ chosen such that $z$ is in this annulus, in direction $p_1$).
We then show that the heat kernel and its gradient can be approximated on this annulus by the partial sum of \eqref{e:heat-kernel-spectral-expansion}
over  eigenfunctions $\globalEF_\globalEV$ which satisfy both
$\globalEV\sim R_z^{-2}$
and $R_z^{-\frac d2}||\globalEF_\globalEV||_{L^2(\ball_{c_1R_z}(z))}\gtrsim 1$.
By the pigeon-hole principle, at least one such eigenfunction, let it be $\globalEF_{i_1}$ has a large partial derivative in the direction $p_1$.
We then consider $\grad\globalEF_{i_1}$ and pick $p_2\perp\grad\globalEF_{i_1}$ and by induction we select $\globalEF_{i_1},\dots,\globalEF_{i_d}$, making sure that at each stage
we can find $\globalEF_{i_k}$, not previously chosen,
satisfying $\left|\partial_{p_k}\globalEF_{i_k}\right|\sim R_z^{-1}$  on $\ball_{c_1R_z}(z)$.
We finally show that for the proper choice of constants $\gamma_1,..,\gamma_d\lesssim 1$, the map
$\Phi:=\left(\gamma_1\globalEF_{i_1},...,\gamma_d\globalEF_{i_d}\right)$ satisfies the desired properties.

When working on a manifold, we  assume in what follows that we fix a local chart containing $B_{R_z}(z)$, as at the beginning of section
\ref{s:manifold}.

{\bf{Step 1. Estimates on the heat kernel and its gradient.}}
Let $K$ be the Dirichlet or Neumann heat kernel on $\Omega$ or $\mathcal{M}$, corresponding to one of the Laplacian operators considered above associated with $g$ and the fixed $\alpha$.

\smallskip\noindent{\bf Assumption} \assumptionA.
Assume $g\in\mathcal{C}^2$, and let $\alpha\in(0,1]$ be given and fixed.
Let constants $\delta_0,\delta_1>0$ depend on $d$, $c_{\min}$, $c_{\max}$, $||g||_{\alphawedgeone}$, ${\alphawedgeone}$.
We consider $z,w\in\Omega$ satisfying $\frac{\delta_1}2 R_z< t^{\frac12} < \delta_1 R_z$ and $|z-w|<\delta_0 R_z$.

\begin{remark}\label{r:inf-vol}
Proposition \ref{p:non-smooth-kernel_estimates} below
makes no assumptions on the finiteness of the volume of $\cM$ and the existence of $\Cweyl$.
It is also used in the proof of Theorem \ref{t:heatkernelmapping}.
\end{remark}

\begin{proposition}
Assume Assumption \assumptionA,   $\delta_0$ sufficiently small, and $\delta_1$ is sufficiently small depending on $\delta_0$.
Then there are constants $C_1,C_2,C_1',C_2',C_9>0$, that depend on  $d$, $\delta_0$, $\delta_1$, $c_{\min}$, $c_{\max}$, $R_z^{\alphawedgeone}||g||_{\alphawedgeone}$, ${\alphawedgeone}$,
such that the following hold:
\begin{itemize}
\item[(i)] the heat kernel satisfies
\begin{gather} \label{e:kernel_estimates}
K_t(z,w) \sim_{C_1}^{C_2} t^{-d\over 2}\,;
\end{gather}
\item[(ii)] if $\frac12\delta_0 R_z< |z-w|$, $p$ is a unit vector in the direction of $z-w$, and $q$ is arbitrary unit vector, then
\begin{equation}
\begin{aligned}
\left|\grad K_t(z,w)\right| \sim_{C_1'}^{C_2'} t^{-d\over 2}\frac{R_z}{t}\,\,\,\mathrm{and}\,\,\,
\left|\partial_p K_t(z,w)\right| \sim_{C_1'}^{C_2'} t^{-d\over 2}\frac{R_z}{t}
\end{aligned}
\label{e:gradient_kernel_estimates}
\end{equation}
\begin{equation}
\left|\partial_q K_t(z,w)-\partial_q K^{\R^d}_t(z,w)\right|\le C_9 t^{\frac{-d}2}\frac{R_z}{t}\,,
\label{e:gradient_kernel_estimates_q}
\end{equation}
where $C_9\rightarrow0$ as $\delta_1\rightarrow 0$ (with $\delta_0$  fixed);
here, $K^{\R^d}_t(z,w)$ is the usual Euclidean heat kernel.
\item[(iii)]
if 
$\frac12\delta_0 R_z< |z-w|$, and $q$ is as above, then for $s\le t$,
\begin{equation}
K_s(z,w) \lesssim_{C_2} t^{-d\over 2}\,\,\,,\,\,\,
\left|\grad K_s(z,w)\right| \lesssim_{C_2'} t^{-d\over 2}\frac{R_z}{t}\,\,\,\mathrm{and}\,\,\,
\left|\partial_q K_s(z,w)\right| \lesssim_{C_2'} t^{-d\over 2}\frac{R_z}{t}\,;
\label{e:kernel_estimates_s}
\end{equation}
\item[(iv)]
$C_1,C_2$ both tend to a single function of
$d$, $c_{\min}$, $c_{\max}$, $||g||_{\alphawedgeone}$, ${\alphawedgeone}$,
as $\delta_1$ tends to $0$ with $\delta_0$ fixed;
\end{itemize}

\label{p:non-smooth-kernel_estimates}
\end{proposition}
\noindent The reason for the factor of $R_z^{\alphawedgeone}$ which we have in  $R_z^{\alphawedgeone}||g||_{\alphawedgeone}$ above is to get scaling invariance.
Proposition \ref{p:non-smooth-kernel_estimates} is proved in  subsection \ref{ss_Heat_kernel_estimates} for the Euclidean case and in subsection \ref{s:HeatKernelEstimatesManifold}.

We continue with  the proof of Theorem \ref{t:datheorem} and \ref{t:datheoremmanifold}.
From here on, unless explicitly stated,   we assume the existence $\Cweyl$.
We have the spectral expansion
\begin{equation}
K_t(x,y)=\sum_{j=0}^{+\infty} e^{-\globalEV_j t}\globalEF_j(x)\globalEF_j(y)\,.
\label{e:kernel_spectral}
\end{equation}

\begin{remark}
The assumptions of Theorems \ref{t:datheorem} and \ref{t:datheoremmanifold} say that $\vol{\cM}=1$ (manifold case) or $\vol{\Omega}=1$ (Euclidean domain case).
Thus, unless explicitly stated, we will assume in the lemmata below that we have $R_z\lesssim_{d,c_{\min},c_{\max},||g||_{\alphawedgeone},{\alphawedgeone}} 1$.
\end{remark}

The following steps aim at replacing appropriately chosen heat kernels by a set of eigenfunctions, by extracting the ``leading terms'' in their
spectral expansion.

{\bf{Step 2. Heat kernel and eigenfunctions.}}
We start by restricting our attention to eigenfunctions which do not have too high frequency.
Let
\begin{equation}\label{d:Lambda-L-H-definitions}
\Lambda_L(A) = \left\{ \globalEV_j : \globalEV_j\le At^{-1}\right\}\,\,\,\mathrm{and}\,\,\,\Lambda_H(A') = \left\{ \globalEV_j : \globalEV_j> A't^{-1}\right\}=\Lambda_L(A')^c
\end{equation}

A first connection between the heat kernel and eigenfunctions is given by the following truncation Lemma, which is proved in  subsection \ref{ss_Heat_kernel_and_eigenfunctions}:

\begin{lemma}
Under Assumption \assumptionA, for $A>1$ large enough and $A'<1$ small enough, depending on $\delta_0,\delta_1,C_1,C_2,C_1',C_2'$ (as in Proposition \ref{p:non-smooth-kernel_estimates}),
there exist constants $C_3,C_4$ (depending on $A,A'$ as well as $d$, $c_{\min}$, $c_{\max}$, $||g||_{\alphawedgeone}$, ${\alphawedgeone}$) such that:
\begin{itemize}
\item[(i)] The heat kernel is approximated by the truncated expansion
\begin{gather} \label{e:FA_kernel_estimates}
K_t(z,w)\sim_{C_3}^{C_4}\suml_{j\in \Lambda_L(A)}\globalEF_j(z)\globalEF_j(w)e^{-\globalEV_j t}\,.
\end{gather}
\item[(ii)]
If $\frac12\delta_0 R_z<|z-w|$,   and  $p$ is a unit vector parallel to $z-w$, then
\begin{eqnarray}
\partial_\dir K_t(z,w)&\sim_{C_3}^{C_4}&\suml_{j\in \Lambda_L(A)\cap \Lambda_H(A')}\globalEF_j(z)\partial_\dir\globalEF_j(w)e^{-\globalEV_j t} \label{e:FA_partial_kernel_estimates}\,.
\end{eqnarray}
Furthermore,
\begin{eqnarray}
\left\|\suml_{j\notin \Lambda_L(A)\cap \Lambda_H(A')}
    \globalEF_j(z)\grad\globalEF_j(w)e^{-\globalEV_j t}     \right\|
\leq C_{10} t^{\frac{-d}2}\frac{R_z}{t}\label{e:FA_grad_kernel_estimates}
\end{eqnarray}
where $C_{10}\rightarrow0$ as $A\rightarrow\infty$ and $A'\rightarrow0$.
\item[(iii)]
$C_3,C_4$ both tend to $1$ as $A\rightarrow\infty$ and $A'\rightarrow0$.
\end{itemize}\label{l:truncated_kernel}
\end{lemma}

This Lemma implies that in the heat kernel expansion we do not need to consider eigenfunctions corresponding to eigenvalues larger than $At^{-1}$.
However, in our search for eigenfunctions with the desired properties, we need to restrict our attention further, by discarding eigenfunctions that have too
small a gradient around $z$.
Let
\begin{equation}\label{d:Lambda-E-definitions}
\Lambda_E(\dir,z,R_z,\delta_0,c_0):=
\left\{\globalEV_j\in\sigma(\Delta): \ \ \
    \frac1{c_0}R_z
    |\partial_\dir\globalEF_j(z)|
    \geq
    \left(\fint_{\ball(z,\frac12\delta_0R_z)}\globalEF_j(z')^2\,dz'\right)^\frac12
    \right\}\,.
\end{equation}

Here and in what follows, $\fint_A f=|A|^{-1}\int_A f$.
The truncation Lemma \ref{l:truncated_kernel} can be strengthened, on average,  into

\begin{lemma}
Assume Assumption \assumptionA,
$\delta_0$ sufficiently small, and $\delta_1$ is sufficiently small depending on $\delta_0$.
For $C_3,C_4$ close enough to $1$ (as in Lemma \ref{l:truncated_kernel}), and $c_0$ small enough
(depending on $C_2,C_1',\delta_0,\delta_1$)
there exist constants $C_5,C_6$ (depending only on $C_3$, $C_4$, $C_9$, and $c_0$) such that
if $\frac12\delta_0 R_z< |z-w|$,
and $p$ is a unit vector parallel to $z-w$,
then
\begin{gather}
    \left|\partial_\dir K_t(w,z)    \right|
\sim_{C_5}^{C_6}
    \left| \sum_{\globalEV_j \in
        \Lambda_L(A)\cap \Lambda_H(A') \cap \Lambda_E(z,R_z,\delta_0,c_0)}
    \globalEF_j(w)\,\partial_\dir\globalEF_j(z)\,e^{-\globalEV_j t}    \right|
    \,.
\label{e:truncatelocalavepartialheatkernel}
\end{gather}
\label{l:truncated_truncated_kernel}
\end{lemma}

{\bf{Step 3. Choosing appropriate eigenfunctions.}}

Define  the constants  $\gamma_{\globalEF_j}$ as
\begin{equation}
\gamma_{\globalEF_j}=\left(\fint_{\ball_{\frac12\delta_0R_z}(z)}\globalEF_j(z')^2\,dz'\right)^{-\frac{1}2}\,.
\label{e:gamma-def}
\end{equation}
Note that since  $\globalEF_j$ is $L^2$ normalized, we have
$\gamma_{\globalEF_j}\gtrsim R_z^{d/2}$.
A subset of these constants and corresponding eigenfunctions  will soon  be chosen to give us the constants
$\{\gamma_j\}$ and corresponding eigenfunctions $\{\globalEF_{i_j}\}$ in the statement of Theorem \ref{t:datheorem} and Theorem  \ref{t:datheoremmanifold}.

The set of eigenfunctions with eigenvalues in $\Lambda_L(A)\cap \Lambda_H(A') \cap \Lambda_E(\dir,z,R_z,\delta_0,c_0)$ is well-suited for our purposes, in view of the following:

\begin{lemma}
Under Assumption \assumptionA, for $\delta_0$ small enough, there exists a constant $C_7$ depending on
$c_0$ and $C_8$ depending on
$\{\delta_0,c_{\min},c_{\max},||g||_{\alphawedgeone},{\alphawedgeone}\}$
and a constant $b>0$ which depends on
$c_0$, $d$, $c_{\min}$, $c_{\max}$, $||g||_{\alphawedgeone}$, ${\alphawedgeone}$
such that the following holds.
Let  $\dir$ be a direction.
For  all $j\in  \Lambda_E(\dir,z,R_z,\delta_0,c_0)$ we have
that for all $z'$ such that $||z-z'||\le b \delta_0R_z$
\begin{equation}
\left|\partial_\dir \globalEF_{j}(z') \right| \sim_{C_7}^{C_8} R_z^{-1}\left(\fint_{\ball_{\frac12\delta_0R_z}(z)}\globalEF_j^2\right)^\frac12\,.
\label{e:partial_lower_bound_nearby}
\end{equation}
\label{l:bigpartials}
Moreover, there exists a index  $j\in \Lambda_L(A)\cap \Lambda_H(A') \cap \Lambda_E(\dir,z,R_z,\delta_0,c_0)$,
so that we have
\begin{equation}\label{e:energy-lower-bound}
\gamma_{\globalEF_j}\lesssim (\Cweyl)^\frac12\,
\end{equation}
with constants depending on $A, C_1,C_1',C_2,C_2',C_6, \{d,c_{\min},c_{\max},||g||_{\alphawedgeone},{\alphawedgeone}\}$, $\delta_0,\delta_1$.
\end{lemma}

We can now complete the proof of Theorems \ref{t:datheorem}, \ref{t:datheoremmanifold}.
\begin{proof}[Proof of Theorems \ref{t:datheorem} and  \ref{t:datheoremmanifold} for the case $g\in\mathcal{C}^2$]
Lemma \ref{l:bigpartials} yields an eigenfunction that serves our purpose in a given direction.
To complete the proof of the Theorems, we need to cover $d$ linearly independent directions.
Pick an arbitrary direction $\dir_1$.
By Lemma \ref{l:bigpartials} we can find $j_1\in \Lambda_L(A)\cap \Lambda_H(A') \cap \Lambda_E(\dir,z,R_z,\delta_0,c_0)$, (in particular $j_1\sim t^{-1}$) such that
$\left|\gamma_{\globalEF_{j_1}}\partial_{\dir_1} \globalEF_{j_1}(z) \right|\geq c_0 R_z^{-1}$.
Let $\dir_2$ be a direction orthogonal to $\grad\globalEF_{j_1}(z)$.
We apply again Lemma \ref{l:bigpartials}, and find $j_2<At^{-1}$ so that
$\left|\gamma_{\globalEF_{j_2}}\partial_{\dir_2}\globalEF_{j_2}(z) \right|\geq c_0 R_z^{-1}$.
Note that necessarily $j_2\neq j_1$ and $p_2$ is linearly independent of $p_1$. In fact, by choice of $p_2$, $$\partial_{\dir_2} \globalEF_{j_1} = 0\,.$$
We proceed in this fashion.
By induction, once we have chosen $j_1,\dots,j_k$ ($k<d$), and the corresponding $p_1,\dots,p_k$, such that
$\left|\gamma_{\globalEF_{j_l}}\partial_{\dir_l} \globalEF_{j_l}(z) \right|\geq c_0 R_z^{-1}$, for $l=1,\dots,k$,
we pick $p_{k+1}$ orthogonal to $\langle\{\grad\globalEF_{j_n}(z)\}_{n=1,\dots,k}\rangle$ and apply Lemma \ref{l:bigpartials},
that yields $j_{k+1}$ such that $\left|\gamma_{\globalEF_{j_{k+1}}}\partial_{\dir_{k+1}} \globalEF_{j_{k+1}}(z) \right|\geq c_0 R_z^{-1}$.

>From here on we denote by $\gamma_{i}=\gamma_{\globalEF_{j_i}}$ for simplicity of notation.
These are the constants $\{\gamma_{i}\}$ appearing in the statement of Theorem \ref{t:datheorem} and Theorem  \ref{t:datheoremmanifold}.

We claim that the matrix
$$A_{k+1}:=\left(\gamma_{m}\partial_{p_n}\globalEF_{j_m}(z)\right)_{m,n=1,\dots,k+1}$$
is lower triangular and $\{p_1,\dots,p_{k+1}\}$ is linearly independent.
Lower-triangularity of the matrix follows by induction and the choice of $p_{k+1}$.
Assume $\sum_{n=1}^{k+1}a_n p_n=0$, then
$\langle \sum_{n=1}^{k+1}a_n p_n,\gamma_{l}\grad\globalEF_{j_l}(z)\rangle=0$ for all $l=1,\dots,k+1$, i.e.
$a\in\mathbb{R}^{k+1}$ solves the linear system
$$ A_{k+1}a = 0\,.$$
But $A_{k+1}$ is lower triangular with all diagonal entries non-zero, hence $a=0$.

For $l\le k$ we have $\langle\grad\globalEF_{j_l}(z),p_{k+1}\rangle=0$ and, by Lemma \ref{l:bigpartials},
$$|\langle\gamma_{l}\grad\globalEF_{j_l},p_{l}\rangle| \gtrsim R_z^{-1}\,.$$
Now let $\Phi_k=(\gamma_{1}\globalEF_{j_1},\dots,\gamma_{k}\globalEF_{j_k})$ and $\Phi=\Phi_d$.
We start by showing that
\begin{equation*}
||\grad\Phi|_z (w-z)||\gtrsim_d \frac{1}{R_z}||w-z||\,.
\end{equation*}
Indeed, suppose that $$||\grad\Phi_k|_z(w-z)||\le\frac{c}{R_z}||w-z||\,,$$ for all $k=1,\dots,d$. For $c$ small enough, this will lead to a contradiction. Let $w-z=\sum_la_l p_l$. We  have (using say Lemma \ref{l:bigpartials})
\begin{equation*}
\begin{aligned}
||\grad\Phi_k|_z(w-z)||=||\sum_la_l\partial_{p_l}\Phi_k|_z||=||\sum_{l\le k}a_l\partial_{p_l}\Phi_k|_z||\gtrsim\left(|a_k|-c\sum_{l<k}|a_l|\right)\frac1{R_z}\,.
\end{aligned}
\end{equation*}
By induction, $|a_k|\le \sum_{l=1}^k c^l||w-z||$. For $c$ small enough, $|a_i|\le\frac{||w-z||}{d}$.
This is a contradiction since $||\sum_ia_i p_i||=||w-z||$ and $||p_i||=1$.

We also have, by Proposition \ref{p:PhiLinftyEstimates},
\begin{equation}
||\grad\Phi|_{w}-\grad\Phi|_z||_{op}\lesssim \left(\frac{||z-w||}{R_z}\right)^{{\alphawedgeone}} \frac{1}{R_z}\label{e:LinftyHolderGraphPhij2}\,.
\end{equation}
Finally, by ensuring $\frac{||z-w_i||}{R_z}$ is smaller then a universal constant for $i=1,2$, we get from equation \eqref{e:LinftyHolderGraphPhij2}
\begin{equation*}
\begin{aligned}
||\Phi(w_1)-\Phi(w_2)||&=
\left|\int_0^1 \grad\Phi|_{tw_1+(1-t)w_2}(w_1-w_2)dt\right|\\
&=
\left|\int_0^1 \left(    \grad\Phi|_{w_1} + \left(\grad\Phi|_{tw_1+(1-t)w_2}-\grad\Phi|_{w_1}\right)\right)
    (w_1-w_2)dt\right|\\
&\gtrsim
\int_0^1 \frac{1}{R_z}||w_1-w_2||dt
\gtrsim
\frac{1}{R_z}c_0||w_1-w_2||\,,
\end{aligned}
\end{equation*}
which proves the lower  bound \eqref{e:bilip_bound_euclidean}.
To prove the upper bound of \eqref{e:bilip_bound_euclidean}, we observe that from Proposition \ref{p:PhiLinftyEstimates} we have the upper bound
$$| \gamma_{l}\partial_{p_l}\varphi_{i_l}(z)|\lesssim R_z^{-1}$$
This completes the proof for the Euclidean case.

We now turn to the manifold case.
Let $R_z$ be as in the Theorem.  We take  $c_1\leq \frac12 \delta_0$ chosen so that
\begin{equation}\label{e:epsilon_0-is-used}
|g^{il}(x)-\delta^{il}|=|g^{il}(x)-g^{il}(z)|<||g||_{\alphawedgeone} ||x-z||^{\alphawedgeone}<\epsilon_0
\end{equation}
for all $x\in B_{2c_1R_z}(z)$.
For this $g$, the above is carried on in local coordinates.
It is then left to prove that the Euclidean distance in the range of the coordinate map is
equivalent to the geodesic distance on the manifold.
We have for all $x,y\in B_{c_1 R_z}(z)$
\begin{equation*}
\begin{aligned}
d_{\mathcal{M}}(x,y)
&\le\int_0^1 \left\|\frac{x-y}{||x-y||}\right\|_g dt
\le\int_0^1 \left\|\frac{x-y}{||x-y||}\right\|_{\mathbb{R}^d}(1+\|g\|_{\alphawedgeone} t^{\alphawedgeone}) dt
\lesssim_{\alphawedgeone} (1+\|g\|_{\alphawedgeone})\,||x-y ||\,.
\end{aligned}
\end{equation*}
The converse can be proved as follows.
Let $\xi:[0,1]\rightarrow\mathcal{M}$ be the geodesic from $x$ to $y$.
$\xi$ is contained in $B_{2d_{\mathcal{M}}(x,y)}(x)$ on the manifold, whose image in the local coordinate chart is contained in
$B_{2(1+\|g\|_{\alphawedgeone})d_{\mathcal{M}}(x,y)}(x)$. We have
\begin{equation*}
\begin{aligned}
d_{\mathcal{M}}(x,y)
&=\int_\xi ||\dot\xi(t)||_g
\gtrsim (1-\|g\|_{\alphawedgeone})\int_\xi||\dot\xi(t)||_{\mathbb{R}^d}
\gtrsim (1-\|g\|_{\alphawedgeone})||x-y||\,.
\end{aligned}
\end{equation*}
\end{proof}

\subsection{The Case of $g\in\mathcal{C}^\alpha$.}
\label{s:proofs-of-EF-C-alpha-case}

\begin{proof}[Proof of Theorems \ref{t:datheorem} and  \ref{t:datheoremmanifold} for the case $g\in\mathcal{C}^{\alpha}$]
We can now give a short proof for the $g\in\mathcal{C}^{\alpha}$ case,
relying  on the $\mathcal{C}^2$ case.
We need the following Lemma, which we prove in section \ref{s:BoundsOnEigenfunctionsManifoldCase}.
\begin{lemma}
Let $J>0$ be given.
If
$$\|\tilde g^{il}_n- g^{il}\|_{L^\infty(\localTangentBall)}\to_n 0$$
with
$\|\tilde g^{il}_n\|_{\mathcal{C}^{\alpha}}$ uniformly bounded and with fixed ellipticity constants (as in \eqref{e:uniformlyelliptic}),
then for $j<J$
\begin{gather}\label{lemma-4'-eq-1}
\|\globalEF_j-\globalPEF_{j,n}\|_{L^\infty(\localTangentBall)}\to_n 0\,,
\end{gather}
\begin{gather}\label{lemma-4'-eq-2}
\|\grad(\globalEF_j-\globalPEF_{j,n})\|_{L^\infty(\localTangentBall)}\to_n 0\,,
\end{gather}
and
\begin{gather}\label{lemma-4'-eq-3}
|\globalEV_j-\globalPEV_{j,n}|\to_n 0\,.
\end{gather}
\label{lemma-4'}
\end{lemma}

Now, to conclude the proof of the Theorem for the $\mathcal{C}^\alpha$ case,
let $J=c_5({d,\frac12c_{\min},2c_{\max},||g||_{\alphawedgeone},{\alphawedgeone}})\cdot R_z^{-2}$.
We may approximate $g$ in $\mathcal{C}^{\alpha}$ norm arbitrarily well by a $\mathcal{C}^2(\mathcal{M})$ metric.
By the above Lemma, and the Theorem for the case of  $\mathcal{C}^2$ metric, we obtain the Theorem for the $\mathcal{C}^{\alpha}$ case.
\end{proof}

%
%
%
\subsection{Heat kernels estimates}

This section makes no assumptions on the finiteness of the volume of $\cM$ and the existence of $\Cweyl$.

%
%
%
%
\subsubsection{Euclidean Dirichlet heat kernel estimates}
\label{ss_Heat_kernel_estimates}

We will start by proving the heat kernel estimates of Proposition \ref{p:non-smooth-kernel_estimates} for the Dirichlet kernel $K^\Omega$.
These estimates are in fact well known, and we include their proof here for completeness, and  also to introduce in a simple setting the kind of probabilistic approach that will be used to obtain estimates in a more general context.

\begin{proof}[Proof of Proposition \ref{p:non-smooth-kernel_estimates} for the Euclidean Dirichlet heat kernel]

Let $B^z_\omega$ below be a Brownian path started at point $z\in \Omega$, and $\tau(\omega)$ its first hitting time of $\partial\Omega$.
We recall that as a consequence of the Markov property we have (see e.g. \cite{Doob_PotentialTheory}, eqn. (9.5) page 590)
\begin{gather}
K_t^\Omega(z,w)=K_t^{\R^d}(z,w)-\E_\omega\left(K_{t-\tau(\omega)}^{\R^d}(B^z_\omega(\tau(\omega)),w)\chi_{t>\tau(\omega)}\right)
\label{run-until-bdry}
\end{gather}
Then,
\begin{equation}
\begin{aligned}
\E_\omega&\left(K_{t-\tau(\omega)}^{\R^d}(B^z_\omega(\tau(\omega)(\omega)),w)\chi_{t>\tau(\omega)(\omega)}\right)
=
\E_\omega\left(\pid (t-\tau(\omega))^{-d\over 2} e^{-\|B(\tau(\omega))-w\|^2 \over 4(t-\tau(\omega))} \chi_{t>\tau(\omega)} \right)\\
&\leq
\E_\omega\left(\pid (t)^{-d\over 2} e^{-\|B(\tau(\omega))-w\|^2 \over 4t} \chi_{t>\tau(\omega)} \right)
\leq
\E_\omega\left( \pid (t)^{-d\over 2} e^{-(1-\delta_0)^2R_z^2 \over 4t} \chi_{t>\tau(\omega)} \right)\\
&\leq
\pid (t)^{-d\over 2} e^{-(1-\delta_0)^2R_z^2 \over 4t}
\leq
\pid (t)^{-d\over 2} e^{-(1-\delta_0)^2R_z^2 \over 4(\delta_1R_z)^2}
=
\pid (t)^{-d\over 2} e^{-(1-\delta_0)^2 \over 4\delta_1^2} \\
\label{e:bound-above-dir}
\end{aligned}
\end{equation}
where for the first inequality we require $\frac{\|B(\tau)-w\|^2}t$ sufficiently large,  which is implied by choosing $\delta_0<1$ and $\delta_1$ small enough.
The last term can be made arbitrarily small by choosing $\delta_1$ small enough, independently of $\delta_0$
(as long as, say, $\delta_0<\frac12$).
We also have
\begin{eqnarray*}
K_t^{\R^d}(z,w)=\pid (t)^{-d\over 2} e^{-\|z-w\|^2\over 4t}\le \pid (t)^{-d\over 2}
\end{eqnarray*}
and
\begin{equation}
\begin{aligned}
K_t^{\R^d}(z,w)
&=
\pid (t)^{-d\over 2} e^{-\|z-w\|^2\over 4t}
&\geq
\pid (t)^{-d\over 2} e^{-\delta_0^2R_z^2 \over 4t}
&\geq
\pid (t)^{-d\over 2} e^{-\delta_0^2 \over \delta_1^2}\,.
\label{e:bound_below_Rn}
\end{aligned}
\end{equation}
If $\delta_0<\frac13$, then $\frac{(1-\delta_0)^2}4 \geq \delta_0$ and so  by reducing $\delta_1$ (while fixing $\delta_0$) we can make the left-hand side of \eqref{e:bound-above-dir} arbitrarily small in comparison with
\eqref{e:bound_below_Rn}.
Now, from equation  \eqref{run-until-bdry} we get
 \eqref{e:kernel_estimates} for the Dirichlet kernel.

Note that the range we have for $t$ and $\|B^z_\omega(\tau(\omega))-w\|$ imply that
$\pid (t)^{-d\over 2} e^{-\|B^z_\omega(\tau(\omega))\|^2\over 4t}$ is monotone increasing in $t$.
Hence we also have
\begin{equation}
\begin{aligned}
\E_\omega\left(\chi_{s>\tau(\omega)}K^{\R^d}_{s-\tau(\omega)}(B^z_\omega(\tau(\omega)),w)\right)
&\leq
\pid (t)^{-d\over 2} e^{-(1-\delta_0)^2 \over 4\delta_1^2}\,.
\label{e:s-bound-above-dir}
\end{aligned}
\end{equation}
If we also have $\frac{\delta_0}{\delta_1}$ is large enough, then $K^{\R^d}_s(z,w)$ is monotone increasing in $s$, and therefore
\begin{equation*}
K^{\R^d}_s(z,w)\leq K_t^{\R^d}(z,w)\leq \pid t^{-\frac{d}2}\,.
\end{equation*}
For and fixed $\delta_0$, we may reduce $\delta_1$ so that by \eqref{e:s-bound-above-dir} is small, and thus we obtain the first part of  \eqref{e:kernel_estimates_s}
from \eqref{run-until-bdry}.

We now turn to estimates  \eqref{e:gradient_kernel_estimates} and second and third parts of \eqref{e:kernel_estimates_s}.
We differentiate equation \eqref{run-until-bdry} and then we bound as follows:
\begin{equation}
\begin{aligned}
\|\grad_w\E_\omega[\chi_{t>\tau(\omega)}&K_{t-\tau(\omega)}^{\R^d}(B^z_\omega(\tau(\omega),w))]\|\\
&=
\left\|\grad_w\E_\omega\left[\chi_{t>\tau(\omega)} \pid (t-\tau(\omega))^{-d\over 2} e^{-\|B_\omega^z(\tau(\omega))-w\|^2 \over 4(t-\tau(\omega))}\right]\right\|\\
&=
\left\|\E_\omega\left[\chi_{t>\tau(\omega)} \grad_w \pid (t-\tau(\omega))^{-d\over 2} e^{-\|B_\omega^z(\tau(\omega))-w\|^2 \over 4(t-\tau(\omega))}\right]\right\|\\
&=
\left\|\E_\omega\left[\chi_{t>\tau(\omega)} {B^z_\omega(\tau(\omega))-w \over 2(t-\tau(\omega))} \pid (t-\tau(\omega))^{-d\over 2} e^{-\|B^z_\omega(\tau(\omega))-w\|^2 \over 4(t-\tau(\omega))}\right]\right\|\\
&\leq
\left\|\E_\omega\left[\chi_{t>\tau(\omega)} {\|B^z_\omega(\tau)-w\| \over 2(t-\tau(\omega))} \pid (t-\tau(\omega))^{-d\over 2} e^{-\|B^z_\omega(\tau(\omega))-w\|^2 \over 4(t-\tau)}\right]\right\|\\
&\leq
\left\|\E_\omega\left[\chi_{t>\tau(\omega)} {\|B^z_\omega(\tau(\omega))-w\| \over 2t} \pid t^{-d\over 2} e^{-\|B^z_\omega(\tau(\omega))-w\|^2 \over 4t}\right]\right\|\\
&=
\pid t^{-d-1\over 2} \E_\omega\left[\chi_{t>\tau(\omega)} {\|B^z_\omega(\tau(\omega))-w\| \over 2\sqrt{t}}  e^{-\|B^z_\omega(\tau(\omega))-w\|^2 \over 4t}\right]\\
&\leq
\pid t^{-d-1\over 2} \E_\omega\left[\chi_{t>\tau(\omega)} {1-\delta_0 \over 2\delta_1}  e^{-(1-\delta_0)^2 \over 4\delta_1^2}\right]\\
&=
\pid t^{-d-1\over 2} {1-\delta_0 \over 2\delta_1}  e^{-(1-\delta_0)^2 \over 4\delta_1^2}=:C(\delta_0,\delta_1)\\
\label{e:heatDkernelgradtailupperbound}
\end{aligned}
\end{equation}
where for the second equality we use the dominated convergence Theorem, for the inequalities in the fifth and in the penultimate line
we choose $\delta_0<1$ and $\delta_1$ small enough. Note that $\delta_1\to 0$ implies that
$C(\delta_0,\delta_1)\to 0$.
Observe that these estimates hold also with $\grad_w$ replaced by $\frac\partial{\partial p}$.

We also have
\begin{eqnarray*}
\grad_w K^{\R^d}_t(z,w)&=&\grad_w \pid t^{-d\over 2} e^{-\|z-w\|^2 \over 4t}
= \pid t^{-d-1\over 2} {(z-w)\over 2\sqrt{t}}  e^{-\|z-w\|^2 \over 4t}\\
&\geq&
\pid t^{-d-1\over 2} {\delta_0R_z\over t}  e^{-\delta_0^2 \over \delta_1^2} (z-w)
\end{eqnarray*}
(with inequality understood entrywise) where as above the last inequality holds for $\delta_0<1$ and $\delta_1$ small enough.
If $\dir$ is parallel to $z-w$, the same estimates hold if we replace $\grad_w$ by $\partial_\dir$.
Hence, for any fixed $\delta_0$,  by reducing $\delta_1$  we get
\begin{gather*}
\|\grad_w K_t^\Omega(z,w)-\grad_w K_t^{\R^d}(z,w)\|\le
    \pid {R_z \over 2t}   t^{-d\over 2} e^{-(1-\delta_0)^2 \over 4\delta_1^2}
\end{gather*}
and therefore
\begin{equation*}
\|\grad_w K_t^\Omega(z,w)\|\sim\|\grad_w K_t^{\R^d}(z,w)\|\sim {t^{ -d\over 2} \frac{R_z}t}\,.
\end{equation*}
The estimate \eqref{e:gradient_kernel_estimates} involving $\partial_p$ is proven analogously.
The second and third parts of  \eqref{e:kernel_estimates_s} follow as above.
Finally, to prove \eqref{e:gradient_kernel_estimates_q} we use \eqref{e:heatDkernelgradtailupperbound} to obtain
\begin{equation*}
\begin{aligned}
\|\partial_q K^\Omega_t(z,w)-\partial_qK^{\R^d}_t(z,w)\|&\le
\pid {R_z \over 2t}   t^{-d\over 2} e^{-(1-\delta_0)^2 \over 4\delta_1^2}
\le C_9(\delta_0,\delta_1) t^{\frac{-d}2}\frac {R_z}t\,.
\end{aligned}
\end{equation*}
\end{proof}

\subsubsection{Local and global heat kernels}
\label{s:LocalAndGlobalHeatKernels}

In this section, let $K$ be the heat kernel, Dirichlet or Neumann, for:
\begin{itemize}
\item[(i)] a domain $\Omega$ (possibly unbounded), and a uniformly elliptic operator $\Delta$ as in \eqref{e:DeltaM}, with $g\in\mathcal{C}^2(\Omega)$;
\item[(ii)] a manifold $\mathcal{M}$ with $g\in\mathcal{C}^2$ satisfying the requirements in section \ref{s:manifold}, and let $\Delta$ be the associated Laplacian.
\end{itemize}
\begin{remark}
We emphasize that in this section we do not assume that the volume of $\mathcal{M}$ is finite.
\end{remark}

Observe that in both settings the existence of an associated Brownian motion is guaranteed (\cite{Oksendal-book} for the $\R^d$ case and the manifold case then follows from uniqueness).
The following result connects $K$ with the Dirichlet kernel on a ball, associated with $\Delta$, to which the estimates of the previous section apply: this will allow us
to extend estimates for the Dirichlet heat kernel on a ball to the general heat kernel $K$.
A more detailed account of the ideas in the following proposition appears in
Stroock's recent book \cite{stroock:PDEforProbabilists} (section 5.2 {\it Duhamel's Formula}).

\begin{proposition}
Let $z\in\Omega$ and $\localToGlobalR\le\mathrm{dist}(z,\partial\Omega)$, or $z\in\mathcal{M}$ and $\localToGlobalR\le \injRad(z)$.
Let $x,y\in \ball(z,\frac14\localToGlobalR)$.
For each path $B_\omega^x$ (starting at $x$), we define $\tau_1(\omega)\le \tau_2(\omega)\le\dots$ as follows.
Let $\tau_1(\omega)$ be the first time that $B_\omega^x$ re-enters $\ball(x,\frac38\localToGlobalR)$ after having exited $\ball(x,\frac12\localToGlobalR)$
(if this does not happen, let $\tau_1(\omega)=+\infty$). Let $x_1=B_\omega^x(\tau_1)$.
By induction, for $n>1$ let $\tau_n(\omega)$ be the first time after $\tau_{n-1}(\omega)$ that $B_\omega$ re-enters $\ball(x,\frac38\localToGlobalR)$ after having exited $\ball(x,\frac12\localToGlobalR)$,
or $+\infty$ otherwise. Let $x_n=B_\omega^x(\tau_n)$. If $\tau_n(\omega)=+\infty$, let $\tau_{n+k}(\omega)=+\infty$ for all $k\ge0$.
Then
\begin{equation}
K_s(x,y)=K_s^{D}(x,y)
+\sum_{n=1}^{\infty}\mathbb{E}_\omega\left[K_{s-\tau_n(\omega)}^{D}(x_n(\omega),y)\Big|\tau_n<s\right]P(\tau_n<s)\,.
\label{e:PetersMagic}
\end{equation}
where
\begin{equation*}
K^D_s=K_s^{Dir(\ball_{\frac12\localToGlobalR}(x))}
\end{equation*}
is the heat kernel at time s for the ball $\ball(x,\frac12\localToGlobalR)$ with Dirichlet boundary conditions.
Moreover there exists an $M=M(c_{\min},c_{\max})$ such that
\begin{equation}
P(\tau_n<s)\lesssim_{d,c_{\min},c_{\max}}
    \exp\{-n\frac{r^2}{Ms}\}\,.
\label{e:Ptaunlessthant}
\end{equation}
\label{p:LocalToGlobalHeatKernel}
\end{proposition}
\begin{remark}
In our applications of this proposition, we have $\localToGlobalR\sim \delta_0 R_z$.
In that case, if $s^{\frac12} < \delta_1 R_z$, for $\frac{\delta_0}{\delta_1}$
sufficiently large (depending only,  on $d,c_{\min},c_{\max}$),
the factor $\exp\{-n\frac{r^2}{Ms}\}$
can be made arbitrarily small.
This gives us  control (exponential in $n$) on the right-hand side of \eqref{e:Ptaunlessthant}.
Hence, for any fixed $\delta_0$, for $\delta_1\to 0$  the right-hand-side of
equation \eqref{e:PetersMagic} is dominated by the first summand.
\label{r:expsmall}
\end{remark}
\begin{proof}[Proof of Proposition \ref{p:LocalToGlobalHeatKernel}]
The proofs for the case of a domain $\Omega$ and the case of a manifold $\mathcal{M}$ are identical.
We have, for any fixed, small enough, $\epsilon>0$,
\begin{equation*}
\begin{aligned}
P^\Omega&(B^x_\omega(s)\in\ball_{\epsilon}(y))
=P^\Omega(B^x_\omega(s)\in\ball_{\epsilon}(y),\tau_1\ge s)+P^\Omega(B^x_\omega(s)\in\ball_{\epsilon}(y),\tau_1<s)\\
&=\int_{\ball_{\epsilon}(y)}K_s^{D}(x,y')dy'\\
    &\ \ \ \ +P^\Omega(B^x_\omega(s)\in\ball_{\epsilon}(y),\tau_2\ge s|\tau_1<s)\,P(\tau_1<s)\\
    &\ \ \ \ +P^\Omega(B^{x_1}_\omega(s)\in\ball_{\epsilon}(y),\tau_1<s,\tau_2<s)\\
&=\int_{\ball_{\epsilon}(y)}K_s^{D}(x,y')dy'\\
    &\ \ \ \ +P^\Omega(B^{x_1}_\omega(s-\tau_1)\in\ball_{\epsilon}(y),\tau_2\ge s|\tau_1<s)\,P(\tau_1<s)\\
    &\ \ \ \ +P^\Omega(B^{x_1}_\omega(s)\in\ball_{\epsilon}(y),\tau_1<s,\tau_2<s)\\
&=\int_{\ball_{\epsilon}(y)}K_s^{D}(x,y')dy'\\
    &\ \ \ \ +\int_{\ball_{\epsilon}(y)}\mathbb{E}_\omega\left[K_{s-\tau_1}^{D}(x_1,y')|\tau_1<s\right]dy\,P(\tau_1<s)\\
    &\ \ \ \ +P^\Omega(B^{x_1}_\omega(s)\in\ball_{\epsilon}(y),\tau_2<s)\\
&=\dots\\
&=\int_{\ball_{\epsilon}(y)}K_s^{D}(x,y')dy'\\
&\ \ \ \ +\sum_{n=1}^{+\infty}\int_{\ball_{\epsilon}(y)}\mathbb{E}_\omega\left[K_{s-\tau_n}^{D}(x_i,y')|\tau_n<s\right]dy'\,P(\tau_n<s).
\end{aligned}
\end{equation*}
By dividing by $|\ball(y,\epsilon)|$ and taking the limit as $\epsilon\rightarrow 0^+$, we obtain \eqref{e:PetersMagic}.

In order to estimate $P(\tau_n<s)$, we need the following

\begin{lemma}
Let  $\Omega$ be a domain  corresponding to a uniformly elliptic operator as in \eqref{e:DeltaM}.
Let $\tau$ be the first exit time from $B_R(z)\subseteq\Omega$ for the corresponding stochastic process started at $z$.
Then there exists $M=M({d,c_{\min},c_{\max}})>0$ such that,
\begin{equation}
P^z(\{\tau\le s\})\lesssim_{d,c_{\min},c_{\max}}\exp\{-R^2(2Ms)^{-1}\}\,.
\label{e:Dirtaudistribution}
\end{equation}
Similarly for $z\in\mathcal{M}$ and $R\le \injRad(z)$.
\label{Ptauleqt}
\end{lemma}
\begin{proof}
First note that without loss of generality we may replace $\Omega$ by $B_{2R}(z)$ with Dirichlet boundary conditions,
and then replace $B_{2R}(z)$ by $\R^d$ by extending the coefficients $g^{ij}$ to $\tilde g^{ij}$ defined on all of $\R^d$.
Let $\hat K$ be the associated heat kernel.
For any $s'>0$ and $x,y\in B_{R}(z)$
\begin{equation}
s'^{-\frac d2}\exp\{-\frac{|x-y|^2}{A_1 s'}\}
\lesssim_{c_{\min},c_{\max},d}
\hat K_{s'}(x,y)\lesssim_{d,c_{\min},c_{\max}}
s'^{-\frac d2}\exp\{-\frac{|x-y|^2}{A_2 s'}\}\,.
\label{e:heatkernelboundstaut}
\end{equation}
holds for $M=M({c_{\min},c_{\max},d})$ and  $A_i=A_i({c_{\min},c_{\max},d})$ (see \cite{DaviesHeatKernels}, Corollary 3.2.8 and Theorem 3.3.4).
We now follow a short proof by Stroock \cite{Stroock:DiffusionSemigroups}. By the strong Markov property, we have
\begin{equation*} \begin{aligned}
P^{\Omega}(B_\omega^z(s)\notin B_R(z))=
\mathbb{E}^{z}_\omega
    \left[P(B_\omega^{\omega(\tau(\omega))}(s-\tau(\omega))\notin B_R(z))\chi_{\{\tau(\omega)<s\}}\right]\,.
\end{aligned}
\end{equation*}

>From the lower bound in equation \eqref{e:heatkernelboundstaut}
we have  that if $x\in\partial B_R(z)$ and $s>0$  then
$P(B_\omega^{x}(s)\notin B_R(x))\ge\epsilon(c_{\min},c_{\max},d)$.
Combining this with the upper bound in equation \eqref{e:heatkernelboundstaut} we have
\begin{equation*}
\begin{aligned}
P^{z}(\tau\le s)&
\le \epsilon(d,c_{\min},c_{\max})^{-1}P(B_\omega^{z}(s)\notin B_R(z))\lesssim_{d,c_{\min},c_{\max}}
\exp\{-R^2(2Ms)^{-1}\}\,.
\end{aligned}
\end{equation*}
\end{proof}

We go back to the proof of Proposition \ref{p:LocalToGlobalHeatKernel}.
To estimate $P(\tau_n<s|\tau_{n-1}<s,\dots,\tau_1<s)$, we observe that between $\tau_{n-1}$ and $\tau_{n}$, the path $\omega$ has to cross both $\partial B_{\frac38\localToGlobalR}(z)$ and $\partial B_{\frac12\localToGlobalR}(z)$: let $\tau_n^*$ and
$\tau_n^{**}$ be the first time this happens, and let $y^*=\omega(\tau_n^*)$. Then
\begin{equation*}
\begin{aligned}
P(\tau_n<s|\tau_{n-1}<s,\dots,\tau_1<s)&\le P(\tau_n^{**}-\tau_n^*<s)\le \sup_{y^*\in B_{\frac38\localToGlobalR}(z)} P^{y^*}\left(\sup_{s'\in[0,s]}
||B_\omega(s')-y^*||>\frac18\localToGlobalR\right)\\
&\lesssim_{c_{\min},c_{\max},d} e^{-\frac {\left(\frac18\localToGlobalR\right)^2}{2Ms}}\,.
\end{aligned}
\end{equation*}
Therefore we have
\begin{equation}
\begin{aligned}
P(\tau_n<s)=P(\tau_1<s,\tau_2<s,\dots,\tau_n<s)
&=\left(\prod_{l=2}^n P(\tau_l<s|\tau_{l-1}<s,\dots,\tau_1<s)\right)P(\tau_1<s)\\
&\lesssim_{c_{\min},c_{\max},d}
\exp\{- n\left(\frac18\localToGlobalR\right)^2(2Ms)^{-1} \}\,.
\label{e:taunexpbound}
\end{aligned}
\end{equation}
Renaming $128M$ to $M$ we get the lemma.
\end{proof}

\begin{remark}
As it is clear from the proof, the proposition holds for any boundary condition on a manifold or  domain.
\end{remark}

\subsubsection{Euclidean Neumann heat kernel estimates}

We use the results of the previous two sections to prove the Neumann case of Proposition \ref{p:non-smooth-kernel_estimates}:

\begin{proof}[Proof of Proposition \ref{p:non-smooth-kernel_estimates} for the Euclidean Neumann heat kernel ]
The starting point is Proposition \ref{p:LocalToGlobalHeatKernel}, which allows us to localize.
We use Proposition \ref{p:non-smooth-kernel_estimates} for the case of $B_{2\delta_0R_z}(z)$.
For this proof, we denote by  $C_2[B]$ be the $C_2$ constant for the Dirichlet ball case.
For $s\le t$ we have using equation \eqref{e:taunexpbound},
\begin{equation}
\begin{aligned}
|K_s(x,y)- K_s^{2\delta_0R_z}(x,y)|&=
|\sum_{n=1}^{+\infty}\mathbb{E}_\omega\left[ K_{s-\tau_n}^{2\delta_0R_z}(x_n(\omega),y)|\tau_n<s\right]P_{\omega}(\tau_n<s)|\\
&\lesssim_{C_2[B]}
\sum_{n=1}^\infty t^{-\frac d2} \underbrace{\exp\{-n\left(\frac{\delta_0R_z}2\right)^2(Ms)^{-1}\}}_{\mathrm{equation\ } \eqref{e:taunexpbound}}\\
& \lesssim_{C_2[B],\delta_0,\delta_1} t^{-\frac d2} \exp\{-\left(\frac{\delta_0R_z}2\right)^2(Ms)^{-1}\}
\end{aligned}
\end{equation}
This proves \eqref{e:kernel_estimates} and the first part of \eqref{e:kernel_estimates_s} (see Remark \ref{r:expsmall}).
For the gradient estimates, i.e.  \eqref{e:gradient_kernel_estimates}, \eqref{e:gradient_kernel_estimates_q},  and the second and third part of \eqref{e:kernel_estimates_s},
\begin{equation*}
\begin{aligned}
\|\grad_y K_s(x,y)- \grad_y  K_s^{2\delta_0R_z}(x,y)\|
&\le \sum_{n=1}^{+\infty}\left\|\grad_y \mathbb{E}_\omega\left[ K_{s-\tau_n}^{2\delta_0R_z}(x_n(\omega),y)|\tau_n<s\right]\right\|P_{\omega}(\tau_n<s)\\
&\lesssim_{C_2'[B]}
\sum_{n=1}^\infty t^{-\frac d2}\frac{\delta_0R_z}t
\underbrace{\exp\{-n\left(\frac{\delta_0R_z}2\right)^2(Ms)^{-1}\}}_{\mathrm{equation\ } \eqref{e:taunexpbound}}\\
& \lesssim_{C_2'[B],\delta_0,\delta_1}  t^{-\frac d2}\frac{\delta_0R_z}t
\exp\{\left(\frac{\delta_0R_z}2\right)^2(Ms)^{-1}\}
\end{aligned}
\end{equation*}
giving us $C_9$.
By Remark \ref{r:expsmall} the exponential term from equation \eqref{e:taunexpbound} can be made small enough so that we obtain estimate \eqref{e:gradient_kernel_estimates}
as well as the second and third parts of \eqref{e:kernel_estimates_s}.
\end{proof}
The proof for the manifold case is postponed to Section \ref{s:HeatKernelEstimatesManifold}.

\subsection{Heat kernel and eigenfunctions}

%
%
%
%
\subsubsection{Bounds on Eigenfunctions}

We record some inequalities that will be used in what follows.

\begin{proposition}
Assume $g\in \mathcal{C}^{\alpha}$.
There exists $b_1<1$, and $C_P>0$ that depends on $d,c_{\min},c_{\max},||g||_{\alphawedgeone},{\alphawedgeone}$ such that
for any
eigenfunction $\globalEF_j$ of $\Delta_\mathcal{M}$ on $B_{R}(z)$, corresponding to the eigenvalue $\globalEV_j$, and $R\le R_z$, the following estimates hold.
For $w\in B_{b_1R}(z)$ and $x,y\in B_{b_1R}(z)$,
\begin{align}
|\globalEF_j(w)|
&\leq C_P
\efpoly(\globalEV_j R^2) \left(\fint_{B_{R}(z)}|\globalEF_j|^2\right)^\frac12\label{e:LinftyPhij}\\
||\grad\globalEF_j(w)||
&\leq C_P
\frac{\gefpoly(\globalEV_j R^2)}R\, \left(\fint_{B_{R}(z)}|\globalEF_j|^2\right)^\frac12\label{e:LinftyGradPhij}\\
||\grad\globalEF_j(x)- \grad \globalEF_j(y)||
&\leq C_P
\frac{\ggefpoly(\globalEV_j R^2)}{R^{1+{\alphawedgeone}}} \left(\fint_{B_{R}(z)}|\globalEF_j|^2\right)^\frac12 \!\!\! ||x-y||^{\alphawedgeone} \label{e:LinftyHolderGraphPhij}
\end{align}
where
$\efpoly(x)=(1+x)^{\frac12 + \beta}$,
$\gefpoly(x)=(1+x)^{\frac32 + \beta}$,
$\ggefpoly(x)=(1+x)^{\frac52 + \beta}$,
with $\beta$  the smallest integer larger than or equal to $\frac{d-2}4$.
\label{p:PhiLinftyEstimates}
\end{proposition}

We postpone the proof to Section \ref{s:BoundsOnEigenfunctionsManifoldCase}.
Related estimates can be found in \cite{Davies:SpectraPropertiesChangesMetric,sogge,smith:c11metrics,smith:sharplplq,Xu} and references therein.

%
%
%
%
\subsubsection{Truncated heat kernel and selecting eigenfunctions}\label{ss_Heat_kernel_and_eigenfunctions}

The goal of this section is to prove Lemma \ref{l:truncated_kernel} and \ref{l:truncated_truncated_kernel}.
All the results of this section and their proofs will be independent on whether we talk about the Dirichlet or Neumann heat kernel,
and on whether we talk about the standard Laplacian or about a uniformly elliptic operator satisfying our usual assumptions and whether we talk about a manifold $\mathcal{M}$ or a domain $\Omega$.
This is because the only tools we will need to obtain the results in this section are the
spectral expansion of the heat kernel \eqref{e:kernel_spectral}, the elliptic estimates of Proposition \ref{p:PhiLinftyEstimates},
the assumption on $\Cweyl$ \eqref{e:weyl-bounds-omega},
and the bound
\begin{equation}
K_t(z,w)\leq K_t(z,z)^\half K_t(w,w)^\half\,.
\label{e:CS_kernel}
\end{equation}
which is a straightforward application of Cauchy-Schwartz inequality to \eqref{e:kernel_spectral}.

\begin{proof}[Proof of Lemma \ref{l:truncated_kernel}]
We upper bound the tail of the heat kernel:
\begin{equation}
\begin{aligned}
\left|\suml_{\globalEV_j\geq At^{-1}}\globalEF_j(z)\globalEF_j(w)e^{-\globalEV_j t} \right|
&\leq
e^{-A\over 2}\left|\suml_{\globalEV_j\geq At^{-1}}   \globalEF_j(z)\globalEF_j(w)e^{-\globalEV_j t\over 2}\right|
\leq
e^{-A\over 2} |K_{t\over 2}(z,z)|^\half |K_{t\over 2}(w,w)|^\half\\
&\lesssim_{C_2}
t^{-d\over2} e^{-A\over 2}
\label{e:BoundOnHeatKernelTailGeneric}
\end{aligned}
\end{equation}
by \eqref{e:CS_kernel} and Proposition \ref{p:non-smooth-kernel_estimates}.
For $A$ large enough, this implies \eqref{e:FA_kernel_estimates}.

For the gradient, we use Proposition \ref{p:PhiLinftyEstimates}, and observe that $x^n e^{-x\over 4}$ is a decreasing function if $x$ is large enough.
We let $r_w=(1-\delta_0)R_z$.
\begin{equation*}
\begin{aligned}
&\left\|\suml_{\globalEV_j\geq At^{-1}}   \globalEF_j(z)\grad\globalEF_j(w)e^{-\globalEV_j t}\right\|
\leq
\left\|e^{-A\over 2}\suml_{\globalEV_j\geq At^{-1}}   \globalEF_j(z)\grad\globalEF_j(w)e^{-\globalEV_j t\over 2}\right\|\\
&\ \ \ \ \leq
e^{-A\over 2} |K_{t\over 2}(z,z)|^\half\left(\suml_{\globalEV_j\geq At^{-1}} \|\grad\globalEF_j(w)\|^2 e^{-\globalEV_j t\over 2}\right)^\half\\
&\ \ \ \ \lesssim_{C_p} %
e^{-A\over 2} |K_{t\over 2}(z,z)|^\half\left(\suml_{\globalEV_j\geq At^{-1}}\gefpoly(\globalEV_j r_w^2)^2r_w^{-2} \fint_{\ball(w,\half r_w)}|\globalEF_j|^2
    e^{-\globalEV_j t\over 2}\right)^\half\\
&\ \ \ \ \lesssim_{C_P,\delta_1} %
e^{-A\over 2} |K_{t\over 2}(z,z)|^\half \left(\suml_{\globalEV_j\geq At^{-1}} \fint_{\ball(w,\half r_w)}|\globalEF_j|^2 e^{-\globalEV_j t\over 2}
        \gefpoly(\globalEV_j t)^2\frac1t e^{-\globalEV_j t\over 2}\right)^\half\\
&\ \ \ \ \lesssim_{C_P,\delta_1} %
e^{-A\over 2} |K_{t\over 2}(z,z)|^\half
    \left(\suml_{\globalEV_j\geq At^{-1}}
       \fint_{\ball(w,\half r_w)}|\globalEF_j|^2 e^{-\globalEV_j t\over 4}
        {1\over t} \right)^\half\\
&\ \ \ \ \lesssim_{C_P,\delta_1} %
e^{-A\over 2} |K_{t\over 2}(z,z)|^\half
    \left( \fint_{\ball(w,\half r_w)}K_{t\over 4}(w',w')dw'
        {1\over t}\right)^\half\\
&\ \ \ \ \lesssim_{C_P,\delta_1,C_2} %
e^{-A\over 2}t^{-d\over2}{1\over \sqrt{t}}
\end{aligned}
\end{equation*}

Now we consider the contribution of the low frequency eigenfunctions to the gradient.
Proceeding as above, and recalling that in this case
$\globalEV_j r_w^2\leq \frac{A'}{\delta_1}$, we obtain
\begin{equation*}
\begin{aligned}
&\left\|\suml_{\globalEV_j\leq A't^{-1}}\globalEF_j(z)\grad\globalEF_j(w)e^{-\globalEV_j t}\right\|\\
&\ \ \ \ \lesssim_{C_P} %
|K_{t}(z,z)|^\half\left(\suml_{\globalEV_j\leq A't^{-1}}\gefpoly(\globalEV_j r_w^2)^2r_w^{-2}\fint_{\ball(w,\half r_w)}|\globalEF_j|^2
    e^{-\globalEV_j t}\right)^\half\\
&\ \ \ \ \lesssim_{C_P} %
|K_{t}(z,z)|^\half\frac1{r_w}\left(\suml_{\globalEV_j\leq A't^{-1}}\fint_{\ball(w,\half r_w)}|\globalEF_j|^2 P_2(\globalEV_jr_w^2)^2 e^{-\globalEV_j t}\right)^\half\\
&\ \ \ \ \lesssim_{C_P} %
P_2\left(\frac{A'}{\delta_1}\right)|K_{t}(z,z)|^\half\frac1{r_w}\left(\suml_{\globalEV_j\leq A't^{-1}}\fint_{\ball(w,\half r_w)}|\globalEF_j|^2 e^{-\frac{\globalEV_j t}2}\right)^\half\\
&\ \ \ \ \lesssim_{C_P} %
P_2\left(\frac{A'}{\delta_1}\right)|K_{t}(z,z)|^\half\frac1{r_w}\left(\suml_{\globalEV_j\leq A't^{-1}}\fint_{\ball(w,\half r_w)}|\globalEF_j|^2 e^{-\frac{\globalEV_j t}{2A'}}\right)^\half\\
&\ \ \ \ \lesssim_{C_P,C_2} %
P_2\left(\frac{A'}{\delta_1}\right)\frac1{r_w}|K_{t}(z,z)|^\half\left(\int_{\ball(w,\half r_w)}K_{t/2A'}(w',w')\right)^\half\\
&\ \ \ \ \lesssim_{C_P,C_2} %
P_2\left(\frac{A'}{\delta_1}\right)\frac1{r_w}t^{-\frac d2}A'^{\frac{d}4}\,.
\end{aligned}
\end{equation*}
Thus, by reducing $A'$ we get  the bound \eqref{e:FA_partial_kernel_estimates} and \eqref{e:FA_grad_kernel_estimates}.
(We note that an alternative approach  to the introduction of $e^{-\frac{\globalEV_j t}{2A'}}$ would have been to
 reduce $\delta_1$ and note that $\frac{r_w}t$ is as large as we want in comparison with $\frac1{r_w}$, and then to reduce $A'$ to compensate for the reduction in $\delta_1$.)
\end{proof}

%
%
For a domain with Dirichlet boundary conditions, we automatically have a bound on
$\Cweyl$ as in \eqref{e:weyl-bounds-omega}:
\begin{lemma}[Weyl's Lemma for Dirichlet boundary conditions]
\label{l:weyls-lemma}
Let $\Omega$ be a  domain in $\mathbb{R}^d$, $\Delta$ a uniformly elliptic operator on $\Omega$ as in \eqref{e:uniformlyelliptic},
with Dirichlet boundary conditions.
Let $\globalEV_0\le\globalEV_1\le\dots$ be the eigenvalues of $\Delta$.
Then
\begin{eqnarray}\label{e:proof-weyls-lemma}
\#\{j: \globalEV_j\leq \globalEV\} &\lesssim_{d,c_{\min},c_{\max}} |\Omega| \globalEV^{d\over 2}\,.
\end{eqnarray}
\end{lemma}
\begin{proof}
Let $K^\Omega$ the associated heat kernel.
Extend the coefficient $g^{ij}$ to $\mathbb{R}^d\smallsetminus \Omega$ by letting $g^{ij}=\delta^{ij}$, and let $\tilde K$ be the associated heat kernel on $\R^d$.
Then $K^\Omega$ is pointwise dominated by $\tilde K$.
Then by estimate \eqref{e:heatkernelboundstaut} we have, following \cite{GHL_RiemannianGeometry}:
\begin{eqnarray*}
\#\{j: \globalEV_j\leq \globalEV\}
\leq e \cdot \sum  e^{-\globalEV_j\over \globalEV}
=
e \cdot \int_\Omega K^\Omega_{\frac1\globalEV}(x,x)dx\leq
e \cdot \int_\Omega \tilde K_{\frac1\globalEV}(x,x)dx
\lesssim_{d,c_{\min},c_{\max}}  |\Omega|\globalEV^{d\over 2}\,.
\end{eqnarray*}
\end{proof}

\begin{remark}\label{r:Weyl-constant}
Notice that in the Dirichlet case  $\Cweyl$ is independent of $\Omega$.
In the Neumann case, if one has good enough estimates on the trace of the corresponding heat kernel, the same proof applies. In general these estimates will depend on $\Omega$, and $\Cweyl$ will not be independent of $\Omega$ (see e.g. \cite{Simon:NeumannEssentialSpectrum,Safarov:WeylAsymptoticFormula}).
\end{remark}

\begin{proof}[Proof of Lemma \ref{l:truncated_truncated_kernel}]
In view of Lemma \ref{l:truncated_kernel}, we will show that the terms in the eigenfunction series
corresponding to $j\in\Lambda_L(A)\cap\Lambda_H(A')$ but $j\notin \Lambda_E(\dir,z,R_z,\delta_0,c_0)$ do not contribute
significantly to the lefthand side of \eqref{e:truncatelocalavepartialheatkernel}.
Let $\Lambda_1=\Lambda_L(A)\cap\Lambda_H(A')\cap\left(\Lambda_E(\dir,z,R_z,\delta_0,c_0)\right)^c$. We thus  have
\begin{eqnarray*}
&&
    \left|\sum_{\globalEV_j \in \Lambda_1}
        \globalEF_j(w)\partial_\dir\globalEF_j(z)e^{-\globalEV_j t}\right|\\
&\leq&
    \sum_{\globalEV_j \in \Lambda_1}
        \left|\globalEF_j(w)\right|
        \left|\partial_\dir\globalEF_j(z)e^{-\globalEV_j t}\right|\\
&\leq&
\left(\sum_{\globalEV_j \in \Lambda_1}|\globalEF_j(w)|^2e^{-\globalEV_j t}\right)^\half
\left(\sum_{\globalEV_j \in \Lambda_1}
        |\partial_\dir\globalEF_j(z)|^2e^{-\globalEV_j t}\right)^\half\\
&\leq&
K_{{t}}(w,w)^\half
\left(\sum_{\globalEV_j \in \Lambda_1}
        |\partial_\dir\globalEF_j(z)|^2e^{-\globalEV_j t}\right)^\half\\
&\leq&
K_{{t}}(w,w)^\half
c_0\frac{1}R_z
\left(\sum_{\globalEV_j \in \Lambda_1}
    \ \fint\limits_{\ball(z,\frac12\delta_0R_z)}
        |\globalEF_j|^2e^{-\globalEV_j t}\right)^\half\\
&\leq&
c_0\frac{1}R_z
K_{{t}}(w,w)^\half
    \fint\limits_{z'\in \ball(z,\frac12\delta_0R_z)}
        K_{t}(z',z')^\half \ dz',.
\end{eqnarray*}
Hence by reducing $c_0$ and using Proposition \ref{p:non-smooth-kernel_estimates} together with Lemma \ref{l:truncated_kernel}
we conclude the proof.  We note that the constant $C_9$ comes into play  since we have to estimate the left hand side of equation \eqref{e:truncatelocalavepartialheatkernel}.
\end{proof}

\begin{remark}
The following  proof is the only place where we use the bound on $\Cweyl$ \eqref{e:weyl-bounds-omega}.
\end{remark}

\begin{proof}[Proof of Lemma \ref{l:bigpartials}]

For sufficiently small $b$, Equations \eqref{e:LinftyGradPhij} and \eqref{e:LinftyHolderGraphPhij} from
Proposition \ref{p:PhiLinftyEstimates},
together with the definition of $\Lambda_E(\dir,z,R_z,\delta_0,c_0)$ give
equation \eqref{e:partial_lower_bound_nearby}.
We turn to showing equation \eqref{e:energy-lower-bound}, which is where $\Cweyl$ will appear.

%
Since at this point all the constants are fixed, to ease the notation we let $\Lambda:=\Lambda_L(A)\cap \Lambda_H(A')\cap \Lambda_E(\dir,z,R_z,\delta_0,c_0)$.
Let $w\in\ball(z,\delta_0R_z)\smallsetminus \ball(z,\frac12\delta_0R_z)$ with
$w-z$ in the direciton of $\dir$.
Observe that Proposition \ref{p:non-smooth-kernel_estimates} and Lemma \ref{l:truncated_truncated_kernel} imply that
\begin{equation}
\begin{aligned}
&K_{\frac{t}2}(w,w)\frac{R_z}t\sim_{C_1}^{C_2}
t^{-d\over 2} \frac{R_z}t\sim_{C_1'}^{C_2'}
\left|{\partial_\dir}K_t(w,z)
\right|\\
&\sim_{C_5}^{C_6}
    \left|\suml_{\globalEV_j\in \Lambda}  \globalEF_j(w){\partial_\dir}\globalEF_j(z)e^{-\globalEV_j t}\right|
    \\
&\lesssim_{\{d,c_{\min},c_{\max},||g||_{\alphawedgeone},{\alphawedgeone}\}}
R_z^{-1}
    \suml_{\globalEV_j\in \Lambda}
        \left|\globalEF_j(w)
    e^{-\globalEV_j t}
    \right|
        \left(\fint_{\ball(z,{\frac12\delta_0R_z})}|\globalEF_j|^2\right)^\frac12
    \\
&\leq
R_z^{-1}
    \left(\suml_{\globalEV_j\in \Lambda}  \globalEF_j(w)^2
    e^{-2\globalEV_j t}
    \right)^\frac12
    \left(\suml_{\globalEV_j\in \Lambda}
        \fint_{\ball(z,{\frac12\delta_0R_z})}|\globalEF_j|^2\right)^\frac12\\
&\leq
R_z^{-1}
K_{2t}(w,w)^\half
    \left(\suml_{\globalEV_j\in \Lambda}
        \fint_{\ball(z,{\frac12\delta_0R_z})}|\globalEF_j|^2\right)^\frac12\,.
\end{aligned}
\end{equation}
giving, with constant depending on $C_1,C_1',C_2,C_2',C_6, \{d,c_{\min},c_{\max},||g||_{\alphawedgeone},{\alphawedgeone}\}$,
\begin{equation}
\begin{aligned}
t^{-d\over 2} \left(\frac{R_z^2}t    \right)^2
\lesssim
    \suml_{\globalEV_j\in \Lambda}
        \fint_{\ball_{\frac12\delta_0R_z}(z)}|\globalEF_j|^2\,.
\end{aligned}
\end{equation}
Thus, by the pigeon-hole principle and Weyl's bound \eqref{e:weyl-bounds-omega},
we have $\globalEV_j \in \Lambda_1$ with
\begin{equation}
\begin{aligned}
\frac1\Cweyl \left(\frac{R_z^2}t    \right)^2
\lesssim
        \fint_{\ball_{\frac12\delta_0R_z}(z)}|\globalEF_j|^2\,.
\end{aligned}
\end{equation}
This gives equation \eqref{e:energy-lower-bound}.
\end{proof}

\subsection{Supplemental Lemmata  for  the Manifold case}
\label{s:supplementalmanifoldcase}

We will initially be interested in localizing the manifold Laplacian $\Delta_{\mathcal{M}}$ to a ball $\justB=B_R(z)$, $R\le \injRad(z)$, in a coordinate chart about $z$,
satisfying the assumptions in the Theorem.
We will rescale up so that $R=1$ (and rescale the volume of $\mathcal{M}$ accordingly).
We impose Dirichlet boundary conditions on $\partial \justB$, and denote by  $\tilde\Delta^\justB$ this Laplacian, which has the expression \eqref{e:DeltaM}.
We will compare $\tilde\Delta^\justB$ with the Euclidean Laplacian $\Delta^\justB$ on $\justB$ (also with Dirichlet boundary conditions).
We will then compare $\tilde\Delta^\justB$ with the global Laplacian $\Delta_{\mathcal{M}}$ on the whole manifold (with Dirichlet or Neumann boundary conditions).
The first comparison is most conveniently done through the associated Green functions.
We use the following notation:
\begin{itemize}
\item[(i)]  $\Delta^\justB\localEF_j=\localEV_j\localEF_j$ is the eigen-decomposition of $\Delta^\justB$, with sorted eigenvalues
$0\le\localEV_0\le\localEV_1\le\dots$, $\tilde\Delta^\justB\localPEF_j=\localPEV_j\localPEF_j$ is the
analogous decomposition of $\tilde\Delta^\justB$, and $\Delta_\mathcal{M}\globalEF_j=\globalEV_j\globalEF_j$ the one for $\Delta_\mathcal{M}$.
The eigenfunctions are assumed to be normalized in the corresponding natural $L^2$ spaces.
\item[(ii)] $G^\justB$ is the Green function on $\justB$, associated with $\Delta^\justB$, with Dirichlet boundary conditions, and $K^\justB$ the corresponding heat kernel;
\item[(iii)] $\tilde G^\justB$ is the Green function on $\justB$, associated with $\tilde\Delta^\justB$, with Dirichlet boundary conditions, and $\tilde K^\justB$ the corresponding heat kernel;
\item[(iv)] the quadratic form associated with $g$, restricted to $\justB$, will be abbreviated as
\begin{equation}
\localgquadratic(u,v)=\int_{\justB}\sum_{i,j=1}^d g^{ij}\partial_i u \partial_j v\,.
\label{e:gquadratic}
\end{equation}
for suitable $u,v$.
\end{itemize}

We will use estimates from \cite{GruterWidman:GreenFunction}, where they are stated only for the case of $d\geq 3$.
Our Theorems are true also for the case $d=2$ (and trivially , $d=1$).
This can be seen indirectly by considering $\tilde{\mathcal{M}}:=\mathcal{M}\times \mathbb{T}$ and noting that the eigenfunctions of $\tilde{\mathcal{M}}$ and the heat kernel of $\tilde{\mathcal{M}}$ both factor.

We let
$$ (L^*)^p(\justB)=\{f:\justB\rightarrow\mathbb{R} \mathrm{\ measurable\ }: ||f||_{(L^*)^p}<+\infty\} $$
where
$$ ||f||_{(L^*)^p} = \sup_{t>0} t|\{x\in\justB:|f(x)|>t\}|^\frac 1p\,. $$

We recall the following Theorem from \cite{GruterWidman:GreenFunction}
\begin{theorem}\label{GW-thm}
Suppose $d\ge3$, and $g\in L^\infty$ and uniformly elliptic with $c_{\min}$ and $c_{\max}$ as in \eqref{e:uniformlyelliptic}.
There exists a unique nonnegative function $\tilde G^\justB:\justB\rightarrow\mathbb{R}\cup\{\infty\}$,
called the Green function, such that for each $y\in\justB$ and any $r>0$ such that $B_r(y)\subseteq \justB$,
\begin{equation*}
\tilde G^\justB(\cdot,y)\in W_c^{1,2}(\justB\setminus B_r(y))\cap W_c^{1,1}(\justB)\,,
\end{equation*}
$G^\justB|_{\partial\justB}=0$, and for all $\phi\in\mathcal{C}_c^\infty(\justB)$
\begin{equation*}
\localgquadratic(\tilde G^\justB(\cdot,y),\phi)=\phi(y)\,.
\end{equation*}
Moreover, for each $y\in\justB$, 
\begin{itemize}
\item[(i)]
$\tilde G^\justB(\cdot,y)\in (L^*)^{\frac d{d-2}},\,\,\,\mathrm{with\ }\,\,\,||\tilde G^\justB||_{(L^*)^{\frac d{d-2}}}\lesssim_{d,c_{\min}} 1$
\item[(ii)]
$\nabla \tilde G^\justB(\cdot,y)\in (L^*)^{\frac d{d-1}},\,\,\,\mathrm{with\ }\,\,\,||\nabla \tilde G^\justB||_{(L^*)^{\frac d{d-1}}}\lesssim_{d,c_{\max},c_{\min}} 1$
\item[(iii)] $\tilde G^\justB(x,y)\gtrsim_{d,c_{\max},c_{\min}} |x-y|^{2-d}\,\,\,\,\mathrm{for\ }\,\,\,|x-y|\le\frac12d(y,\partial\justB)$
\item[(iv)]
\begin{equation}
\tilde G^\justB(x,y)\lesssim_{d,c_{\max},c_{\min}} |x-y|^{2-d}
\label{G-estimate}
\end{equation}
\end{itemize}
If $g\in \mathcal{C}^{\alpha}$ we also have (see page 333 in \cite{GruterWidman:GreenFunction})
\begin{itemize}
\item[(v)]
\begin{equation}
\nabla_y \tilde G^\justB(x,y)\lesssim_{d,c_{\max},c_{\min},{\alphawedgeone},||g||_{\alphawedgeone}} |x-y|^{1-d}
\label{e:gradGestimate}
\end{equation}
\item[(vi)]
\begin{equation}
|\nabla_x \tilde G^\justB(x_1,y)-\nabla_x \tilde G^\justB(x_2,y)|
\lesssim_{d,c_{\max},c_{\min},R^{\alphawedgeone},{\alphawedgeone},||g||_{\alphawedgeone}} \frac{|x_1-x_2|^{\alphawedgeone}}{|x_1-y|^{d+{\alphawedgeone}-1}+|x_2-y|^{d+{\alphawedgeone}-1}}\,.
\label{e:gradGholderestimate}
\end{equation}
\end{itemize}
\label{t:WG}
\end{theorem}

Simple consequences of the bounds above are the following inequalities, which we record for future use:
\begin{align}
\int_{c_1 \Rrescbp{}\le||z-y||\le c_2\Rrescbp{}}|\tilde G^{\Rrescbp{}}(z,y)|^p\,dy&\lesssim_{c_1,c_2,d,c_{\min},c_{\max},p} \Rrescbp{(2-d)+\frac dp} \label{e:annulusintG}\\
\int_{c_1 \Rrescbp{}\le||z-y||\le c_2\Rrescbp{}}|\nabla_y \tilde G^{\Rrescbp{}}(z,y)|^p\,dy&\lesssim_{c_1,c_2,d,c_{\min},c_{\max},{\alphawedgeone},||g||_{\alphawedgeone},p} \Rrescbp{(1-d)+\frac dp}\,, \label{e:annulusintgradG}\\
\int_{\justB} |\nabla_y \tilde G^{\Rrescbp{}}(z,y)|\,dy&\lesssim_{d,c_{\min},c_{\max},{\alphawedgeone},||g||_{\alphawedgeone}} \Rrescbp{}\,, \label{e:annulusintgradG2}
\end{align}
which are an immediate consequence of \eqref{e:gradGestimate}, and are valid for $c_1,c_2>0$ and $0<R_z'<R_z$.

We recall that if we only assume uniform ellipticity, without any assumption on the modulus of continuity of $g$, then we have no pointwise estimates on $\grad G$.

%
%
%
%
\subsubsection{Perturbation of eigenfunctions}
We start by comparing eigenfunctions of the Euclidean $\Delta^\justB$ with eigenfunctions of $\tilde\Delta^\justB$.
We remind the reader that we have rescaled up to $R=1$.

\begin{lemma}
Let $J>0$ and $\eta>0$ be given.
There is an $\epsilon_0=\epsilon_0(J)$ so that if $\epsilon<\epsilon_0$ and
$Id: (\justB,\delta^{ij})\to (\justB, g^{ij})$
is $1+\epsilon$ bi-Lipschitz, then for $j<J$,
\begin{gather*}
\|\localEF_j-\localPEF_j\|_{L^2(\justB)}\leq \eta\,\,,\,\,
|\localEV_j-\localPEV_j|<\eta\localEV_j
\end{gather*}
\label{l:frenchperturbation}
\label{lemma-1}
\end{lemma}

\begin{proof}
This follows from Lemma 20 in \cite{GHL_RiemannianGeometry}.
\end{proof}

\begin{lemma}
There is an integer $\efk>0$ such that
the following bounds hold:
\begin{equation}
\begin{aligned}
\|\localPEF_j\|_{L^\infty(\justB)}&\lesssim_{d,c_{\min},c_{\max}} (\localPEV_j \Rresc{2})^{\efk}\Rresc{-\frac{d}2}
\label{e:Linftyeigen}
\end{aligned}
\end{equation}
and if $g^{ij}\in \mathcal{C}^{\alpha}$ we also have
\begin{align}
\|\grad_y\localPEF_j\|_{L^\infty(\justB)}&\lesssim_{d,c_{\min},c_{\max},{\alphawedgeone},||g||_{\alphawedgeone}} \localPEV_j \Rresc{}(\localPEV_j \Rresc{2})^{\efk}\Rresc{-\frac{d}2} \label{e:Linftygradeigen}\,,
\end{align}
with $\efk=\frac{d-1}2$ for $d$ odd and $\efk=\frac{d}2$ for $d$ even.
\label{lemma-ef-bound}
\end{lemma}

\begin{proof}

By the definition of $\tilde G$, and by recalling that $\tilde G|_{\partial \justB}=0$,
\begin{gather*}
\localPEF_j={\tilde G^\justB}\tilde\Delta^\justB \localPEF_j=\localPEV_j{\tilde G^\justB}\localPEF_j
=...=
\localPEV_j^k\underbrace{{\tilde G^\justB}\dots{\tilde G^\justB}}_{k}\localPEF_j
\end{gather*}
Let $p_i,q_i$ be such that
\begin{equation*}
\sum\limits_{1\leq i\leq k} p_i^{-1} - k +1 = q_k^{-1}\,.
\end{equation*}
Then, using Young's inequality we have
\begin{equation*}
\begin{aligned}
\|\underbrace{{\tilde G^\justB}*\dots*{\tilde G^\justB}}_{k}\|_{L^{q_k}}
&\lesssim_{q_{k-1},q_k,p_k}
\|\underbrace{{\tilde G^\justB}*\dots*{\tilde G^\justB}}_{k-1}\|_{L^{q_{k-1}}}\|{\tilde G^\justB}\|_{L^{p_k}}\\
&\lesssim_{q_{k-2},q_{k-1},p_{k-1}}...
\lesssim _{q_1,q_2,p_2}\|{\tilde G^\justB}\|_{L^{p_1}}...\|{\tilde G^\justB}\|_{L^{p_k}}\,.
\end{aligned}
\end{equation*}
We have
${\tilde G^\justB}\in L^{\frac{d-1}{d-2}}$ by Theorem \ref{t:WG};
we take $p_i=\frac{d-1}{d-2}$ and take $k=d-1$ and get
$$q_{d-1}^{-1}=(d-1)\frac{d-2}{d-1} - (d-1) + 1=0\,,$$
for odd $d$
$$q_\frac{d-1}2^{-1}=\frac{d-1}2 \frac{d-2}{d-1} - \frac{d-1}2 + 1=\frac12\,,$$
and for even  $d$
$$q_\frac{d}2^{-1}=\frac{d}2 \frac{d-2}{d-1} - \frac{d}2 + 1=
\frac{d}2 (\frac{d-2}{d-1} -1) + 1\leq  \frac12\,.$$
Now, for odd $d$,
\begin{gather*}
\|\localPEF_j\|_{\infty}\lesssim
    \localPEV_j^{\frac{d-1}2}\|\underbrace{{\tilde G^\justB}*...*{\tilde G^\justB}}_{\frac{d-1}2}\|_{L^2}
    \|\localPEF_j\|_{L^2}
    \leq
    \localPEV_j^{\frac{d-1}2}\|{\tilde G^\justB}\|_{L^\frac{d-1}{d-2}}\lesssim \localPEV_j^{\frac{d-1}2}
\end{gather*}
which gives the first desired bound.
If $d$ is even do the same with $\frac{d}2$ replacing $\frac{d-1}2$.

For the gradient estimate, we have
\begin{gather*}
|\grad_y\localPEF_j|=|\grad_y \localgquadratic({\tilde G^\justB},\localPEF_j)|=
|\grad_y\int_{\justB} {\tilde G^\justB} \tilde\Delta^\justB \localPEF_j|=
|\localPEV_j \grad_y\int {\tilde G^\justB}  \localPEF_j|
\leq
\localPEV_j\|\grad_y{\tilde G^\justB}\|_{L^1}\|\localPEF_j\|_{\infty}\,,
\end{gather*}
where we used the defining property of $\tilde G^\justB$ in Theorem \ref{t:WG} and Green's Theorem.
We estimate the last term by \eqref{e:Linftyeigen} and equation \eqref{e:gradGestimate} to get the desired result.
\end{proof}

We can now convert the $L^2$-estimates in Lemma \ref{lemma-1} into $L^\infty$-estimates.
We will need the following

\begin{lemma}
Assume that $|g^{ij}(x)-\delta^{ij}|<\epsilon$ for $x\in\justB$.
Then for $\psi\in C^\infty_c(\justB)$ we have
\begin{equation}
\bigg\|      \int_{\justB} \langle
    \nabla\left(\tilde G^\justB(z,w)-G^\justB(z,w)\right),\nabla\psi(z)\rangle dz
\bigg\|_{L^\infty(\justB)}
\leq
    \epsilon\|\grad_y{\tilde G^\justB}\|_{L^1} \|\grad\psi\|_{L^\infty(\justB)}
\label{e:Gpertubationestimate-1}
\end{equation}
and if $g^{ij}\in \mathcal{C}^{\alpha}$ we also have
\begin{equation}
\bigg\| \int_{\justB} \langle
    \nabla \left( \tilde G^\justB(z,w)-G^\justB(z,w)\right),\nabla\localPEF_l(z)\rangle dz
\bigg\|_{L^\infty(\justB)}
\lesssim_{d,c_{\min},c_{\max},{\alphawedgeone},||g||_{\alphawedgeone}}
    \epsilon
        \localPEV_l \Rresc{}(\localPEV_l \Rresc{2})^{\efk}\Rresc{-\frac{d}2}
\label{e:Gpertubationestimate-2}
\end{equation}
as well as
\begin{equation}
\bigg\|      \int_{\justB} \langle
    \nabla \left( \tilde G^\justB(z,w)-G^\justB(z,w)\right),\nabla\localEF_l(z)\rangle dz
\bigg\|_{L^\infty(\justB)}
\lesssim_{d,c_{\min},c_{\max},{\alphawedgeone},||g||_{\alphawedgeone}}
    \epsilon
        \localEV_l \Rresc{}(\localEV_l \Rresc{2})^{\efk}\Rresc{-\frac{d}2}
\label{e:Gpertubationestimate-3}
\end{equation}
with $\efk$ as in Lemma \ref{lemma-ef-bound}.
\end{lemma}
\begin{proof}
\begin{eqnarray*}
&&\bigg\|      \int_{\justB}
    \langle\nabla \left( G^\justB(z,w)-\tilde G^\justB(z,w)    \right),\nabla\psi(z)\rangle dz
\bigg\|_{L^\infty(\justB)}\\
&=&
\bigg\|      \int_{\justB} \sum_{i,j}
    \delta^{ij}\partial_i \left( G^\justB(z,w)-\tilde G^\justB(z,w)    \right)\partial_j\psi(z) dz
\bigg\|_{L^\infty(\justB)}\\
&=&
\bigg\|      \int_{\justB} \sum_{ i,j}
    \left(\delta^{ij}\partial_i G^\justB(z,w)-
        g^{ij}(z) \partial_i\tilde G^\justB(z,w)    \right)\partial_j\psi(z) dz\\
&\phantom{=}&     +
      \int_{\justB} \sum_{ i,j}
       \left(g^{ij}(z) \partial_i \tilde G^\justB(z,w)-\delta^{ij}\partial_i \tilde G^\justB(z,w)    \right)
                    \partial_j\psi(z) dz
\bigg\|_{L^\infty(\justB)}\\
&=&\bigg\|
\int_{\justB} \sum_{ i,j}
       \left((g^{ij}(z)-\delta^{ij}) \partial_i \tilde G^\justB(z,w) \right)
            \partial_j\psi(z) dz
\bigg\|_{L^\infty(\justB)}\\
&\lesssim&
    \epsilon\|\grad_y{\tilde G^\justB}\|_{L^1} \|\grad\psi\|_{L^\infty(\justB)}
\end{eqnarray*}
which gives \eqref{e:Gpertubationestimate-1}.
Using Lemma \ref{lemma-ef-bound} one also gets
\eqref{e:Gpertubationestimate-2} and \eqref{e:Gpertubationestimate-3}.
\end{proof}

\begin{lemma}
Let $J,\eta>0$ be given. Let $\efk$ be as in Lemma \ref{lemma-ef-bound}.
There is an $\epsilon_0$ which depends on $J$, $\eta$, $d$, $c_{\max}$, $c_{\min}$, $||g||_{\alphawedgeone}$, ${\alphawedgeone})$, so that if $\epsilon<\epsilon_0$, and $|g^{il}(x)-\delta^{il}|<\epsilon$, for $x\in\justB$,
then for $j<J$,
\begin{gather}\label{lemma-4-eq-3}
|\localEV_j-\localPEV_j|<\eta\localEV_j\,.
\end{gather}
\begin{gather}\label{lemma-4-eq-1}
\|\localEF_j-\localPEF_j\|_{L^\infty(\justB)}\lesssim_{d,c_{\min},c_{\max},||g||_{\alphawedgeone},{\alphawedgeone}}
    \eta Q_1 (\localPEV_l)
\end{gather}
where $Q_1$ is a polynomial of degree $2\efk$.
If $g\in \mathcal{C}^{\alpha}$ we also have
\begin{gather}\label{lemma-4-eq-2}
\|\grad(\localEF_j-\localPEF_j)\|_{L^\infty}\lesssim_{d,c_{\min},c_{\max},||g||_{\alphawedgeone},{\alphawedgeone}}\eta Q_2(\localPEV_l)
\end{gather}
where $Q_2$ is a polynomial of degree $2\efk+1$.
\label{lemma-4}
\end{lemma}

\begin{proof}
The bound \eqref{lemma-4-eq-3} follows from Lemma \ref{lemma-1}.
Let $q_i$ and $p_i$   be as in the proof of Lemma \ref{lemma-ef-bound}.

We have using the definitions of $G^\justB$ and $\tilde G^\justB$
\begin{eqnarray*}
\localEF_l(w)-\localPEF_l(w)
&=&
 \int_{\justB} \sum_{i,j}
    \delta^{ij}\partial_i G^\justB(z,w)\partial_j\localEF_l(z) -
    g^{ij}(z)\partial_i\tilde G^\justB(z,w) \partial_j\localPEF_l(z)                         dz\\
&=&
 \int_{\justB} \sum_{i,j}
    \delta^{ij}\left( \partial_i G^\justB(z,w)\partial_j\localEF_l(z) -
    \partial_i \tilde G^\justB(z,w)\partial_j\localEF_l(z)     \right)\\
&\phantom{=}&
         +
    \left( \delta^{ij}\partial_i \tilde G^\justB(z,w)\partial_j\localPEF_l(z)-
    g^{ij}(z)\partial_i\tilde G^\justB(z,w) \partial_j\localEF_l(z)          \right)                   dz\\
&=&
E^1(w) +
 \int_{\justB} \sum_{i,j}
    \left( \partial_i \tilde G^\justB(z,w)\partial_j\localEF_l(z)-
    g^{ij}(z)\partial_i\tilde G^\justB(z,w) \partial_j\localPEF_l(z)          \right)                   dz\\
&=&
E^1(w) +
\localEV_l \cdot \tilde  G^\justB * \localEF_l(w) -
\localPEV_l \cdot \tilde  G^\justB * \localPEF_l(w)\\
&=&
E^1(w) +
(\localEV_l-\localPEV_l) \cdot \tilde  G^\justB * \localEF_l (w)
 +
\localPEV_l \cdot \tilde  G^\justB * (\localEF_l-\localPEF_l)(w)\\
&=&
E^1(w) +
E^2(w) +
\localPEV_l \cdot \tilde  G^\justB * (\localEF_l-\localPEF_l)(w)\,.
\end{eqnarray*}
where we have from equation \eqref{e:Gpertubationestimate-3}
\begin{equation}
\|E^1\|_{L^\infty(\justB)}\\
\lesssim_{d,c_{\min},c_{\max},{\alphawedgeone},||g||_{\alphawedgeone}}
    \epsilon
        \localEV_l \Rresc{}(\localEV_l \Rresc{2})^{\efk}\Rresc{-\frac{d}2}
\end{equation}
and
\begin{equation}
\|E^2\|_{L^\infty(\justB)}\\
\lesssim_{d,c_{\min},c_{\max},{\alphawedgeone},||g||_{\alphawedgeone}}
    \eta \localEV_l\cdot
         \Rresc{}(\localEV_l \Rresc{2})^{\efk}\Rresc{-\frac{d}2}\,.
\end{equation}
Iterating, we have
\begin{eqnarray*}
|\localEF_l(w)-\localPEF_l(w)|
&=&
|E^1(w) +
E^2(w) +
\localPEV_l \cdot \tilde  G^\justB * (\localEF_l-\localPEF_l)(w)|\,.\\
&=&
|E^1(w) +
E^2(w) +
\localPEV_l \cdot \tilde  G^\justB * \left(E^1+E^2+\localPEV_l \cdot \tilde  G^\justB * (\localEF_l-\localPEF_l)\right)(w)|\,.\\
&=&...\\
&\leq&
\|E^1+E^2\|_{L^\infty(\justB)}
    \sum\limits_{k=0}^{\efk-1}
    \left(\localPEV_l \|\tilde G^\justB\|_{L^1(\justB)}    \right)^k
+
\localPEV_l^{\efk}
    |\underbrace{{\tilde G^\justB}*...*{\tilde G^\justB}}_{{\efk}\mathrm{\ times}}*(\localEF_l(w)-\localPEF_l(w))|\\
&\leq&
\|E^1+E^2\|_{L^\infty(\justB)}
    \sum\limits_{k=0}^{\efk-1}
    \left(\localPEV_l \|\tilde G^\justB\|_{L^1(\justB)}    \right)^k
+ \localPEV_l^{\efk}
    \|{\tilde G^\justB}\|_{L^\frac{d-1}{d-2}} \|\localEF_l-\localPEF_l \|_2\\
&\lesssim&
    2\eta
        \localEV_l \Rresc{}(\localEV_l \Rresc{2})^{\efk}\Rresc{-\frac{d}2}
            \sum\limits_{k=0}^{\efk-1}
                \localPEV_l ^k
+ \eta\cdot\localPEV_l^{\efk}=\eta Q_1 (\localPEV_l)
\end{eqnarray*}
where we require for the penultimate inequality $\epsilon<\eta$.

To prove the gradient estimate,
\begin{eqnarray}
\left|\grad (\localPEF_j-\localEF_j))(y)\right|&=&
\left|\grad_y\int \sum_{i,l}  \partial_{z_i} \tilde{G}^R(z,y) g^{il}\partial_{z_l} \localPEF_j(z)-\sum_{i,l}\partial_{z_i} G^\justB(z,y)\delta^{il}\partial_{z_l} \localEF_j(z)\right|\\
&=& \left|\grad_y\int    \tilde{G}^R(z,y)\localPEV_j \localPEF_j(z)-G^\justB(z,y)\localEV_j \localEF_j(z)\right|\\
&=&\left|\int       \grad_y\tilde{G}^R(z,y)\localPEV_j \localPEF_j(z)-\grad_y G^\justB(z,y)\localEV_j \localEF_j(z)\right|\\
&\leq& \int     \left|\grad_y(\tilde{G}^R-G^\justB)(z,y)\right| \cdot \left|\localPEV_j\localPEF_j(z)\right|\\ &&+    \left|\grad_yG^\justB(z,y)\right|\cdot \left| \localPEV_j \localPEF_j(z) - \localEV_j \localEF_j(z) \right|
\end{eqnarray}
Now using equation \eqref{lemma-4-eq-1}, Lemma \ref{lemma-1}, and Theorem \ref{GW-thm} we get equation \eqref{lemma-4-eq-2}.
\end{proof}

\subsubsection{Bounds on Eigenfunctions}
\label{s:BoundsOnEigenfunctionsManifoldCase}

The main goal of this section is to prove Proposition \ref{p:PhiLinftyEstimates}.
We note that the inequalities \eqref{e:LinftyPhij}, \eqref{e:LinftyGradPhij}, and \eqref{e:LinftyHolderGraphPhij} are invariant under scalings of the metric,
and so, once again, we assume in the proof of this Proposition and in all the Lemmata that $R=1$.
In this section all constants subsumed in $\lesssim$ and $\gtrsim$ will in general depend on $d,c_{\min},c_{\max},||g||_{\alphawedgeone},{\alphawedgeone}$.
We will need the following result.

\begin{lemma}[Lemma $3.1$ from  \cite{GruterWidman:GreenFunction}]
Suppose $h$ is a bounded solution of $\tilde\Delta^\justB h=0$ in $\justB$.
Then
\begin{equation}
|\grad h (x)|\lesssim_{d,c_{\min},c_{\max},||g||_{\alphawedgeone},{\alphawedgeone}} (1-\dist(x,z))^{-1} ||h||_{L^\infty(\justB)}\,.
\end{equation}
\label{l:GWPerturbLinfty}
\end{lemma}

\begin{lemma}
Assume that $g\in \mathcal{C}^{\alpha}$ and $\tilde\Delta^\justB h=0$ on $\justB=B_1(z)$.
Then for any $r<1$ 
\begin{equation*}
||h||_{L^\infty(B_r(z))}\lesssim_{d,c_{\min},c_{\max}} C_r ||h||_{L^2(B_{2r}(z))}\,
\end{equation*}
and
\begin{equation*}
||\grad h||_{L^\infty(B_{\frac r2}(z))}\lesssim_{d,c_{\min},c_{\max},||g||_{\alphawedgeone},{\alphawedgeone}} 
  C'_r ||h||_{L^2(B_{2r}(z))}\,.
\end{equation*}
\label{lemma-5}
\end{lemma}
\begin{proof}
Let  $r$ as above be given.
Fix  $0<a_1<a_2<1$. By the coarea formula \cite{Federer:GeometricMeasureTheory}, we have
\begin{equation*}
\begin{aligned}
\int_{a_1}^{a_2}\int_{\{x\in B_r(z): \tilde G^r(x)=t\}}&|h(x)|\,d\mathcal{H}^{d-1}(x) dt
=
\int_{\{x\in B_r(z):a_1<\tilde G^r(x)<a_2\}}|h\grad \tilde G^r| dx\\
&\le ||h||_{L^2(B_r(z))}|| |\grad \tilde G^r| ||_{L^2(\ball_{C(d,\frac{c_{\max}}{c_{\min}} a_2^{d-2})}(z) \setminus\ball_{c(d,\frac{c_{\max}}{c_{\min}} a_1^{d-2})}(z))}\\
&\lesssim ||h||_{L^2(\ball_r(z))}
\end{aligned}
\end{equation*}
by estimate \eqref{e:annulusintgradG}.
Hence there exists $t^*\in [a_1,a_2]$ such that
\begin{equation*}
\int_{\{x\in B_r(z): \tilde G^r=t^*\}}|h(x)|\,d\mathcal{H}^{d-1}(x)\lesssim ||h||_{L^2(\ball_r(z))}(a_2-a_1)^{-1}
\end{equation*}

Now, by $\tilde\Delta^\justB$ harmonicity,
\begin{equation*}
\begin{aligned}
|h(z)|&=\left|\int_{\{y\in \ball_r(z): \tilde G^{r}(y)=t^*\}} h(y)\frac{\partial \tilde G^{r}}{\partial n}(z,y)d\mathcal{H}^{d-1}(y)\right|\\
&\le ||\grad \tilde G^{r}||_{L^\infty(\{\tilde G^{r}(y)=t^*\})}\int_{\{y\in \ball_r(z): \tilde G^{r}=t^*\}} |h(y)|\,d\mathcal{H}^{d-1}(y)\\
&\lesssim  ||h||_{L^2(\ball_r(z))}\lesssim_{a_1,a_2,d,c_{\min},c_{\max}}||h||_{L^2(\ball_r(z))}\,.
\end{aligned}
\end{equation*}
Essentially the same proof holds if we replace $z$ by $w\in B_r(z))$ in the above estimates, giving the desired bound on $h$.

In order to estimate the gradient, we use Lemma \ref{l:GWPerturbLinfty}, which gives us
\begin{equation*}
||\grad h||_{L^\infty(B_{\frac r2}(z))}
\lesssim_{d,c_{\min},c_{\max}, ||g||_{\alphawedgeone},{\alphawedgeone}} C'_r ||h||_{L^\infty(B_r(z))}\,.
\end{equation*}
which implies the desired estimate.
\end{proof}

\begin{lemma}
Assume that $g\in \mathcal{C}^{\alpha}$.
Let $\localPEF_j$ and $\globalEF_k$ be as above. Then we have the estimate
\begin{equation}
\|\localPEF_j\globalEF_{k}\|_{L^{\frac{2d}{d-2}}(\justB)}\lesssim_{d,c_{\min},c_{\max},||g||_{{\alphawedgeone}},{\alphawedgeone}}
(((\localPEV_j+\globalEV_k)\Rresc{})^\frac12+\localPEV_j\Rresc{})(\localPEV_j\Rresc{2})^\efk||\globalEF_{k}||_{L^2(\justB)}\,.
\label{e:Phiontildevarphi}
\end{equation}
\label{lemma-6}
\end{lemma}
\begin{proof}
By  the Sobolev embedding Theorem it is enough to prove that
\begin{gather}\label{14-11-05}
\|\grad(\localPEF_j\globalEF_{k})\|_{L^2(\justB)}\lesssim
(((\localPEV_j+\globalEV_k)\Rresc{})^\frac12+\localPEV_j\Rresc{})(\localPEV_j\Rresc{2})^\efk||\globalEF_{k}||_{L^2(\justB)}\,.
\end{gather}
To this end, first note that we may write
\begin{equation}
\begin{aligned}
\tilde\Delta^\justB(\localPEF_j\globalEF_{k})&= \globalEF_{k}\tilde\Delta^\justB\localPEF_j + \localPEF_j\tilde\Delta^\justB\globalEF_{k} + \sum_{i,j=1}^d g^{ij}\partial_i\globalEF_{k}\partial_j\localPEF_j\\
&=(\localPEV_{j}+\globalEV_{k})\globalEF_{k}\localPEF_j + \sum_{i,j=1}^d g^{ij}\partial_i\globalEF_{k}\partial_j\localPEF_j
\end{aligned}
\label{e:deltavarphiPhieq}
\end{equation}
and so $\tilde\Delta^\justB(\localPEF_j \globalEF_{k})$ is defined as a function, and not just a distribution.
Observe that $\globalEF_l$ does not satisfy any particular boundary condition on $\partial{\justB}$, however since $\localPEF_j=0$ on $\partial{\justB}$, integration by parts gives
\begin{gather}
\langle\localPEF_j \globalEF_{k},\tilde\Delta^\justB(\localPEF_j \globalEF_{k})\rangle_{\justB} = \localgquadratic(\grad(\localPEF_j\globalEF_{k}),\grad(\localPEF_j\globalEF_{k}))\,.
\end{gather}

Now, since $g$ is a positive quadratic form,
\begin{eqnarray*}
\langle\grad(\localPEF_j \globalEF_{k}),\grad(\localPEF_j \globalEF_{k})\rangle_{g_\justB}&=& \langle \localPEF_j\grad\globalEF_k+\globalEF_k\grad\localPEF_j,\localPEF_j\grad\globalEF_k+\globalEF_k\grad\localPEF_j\rangle_{g_\justB} \\
&\geq& \langle\localPEF_j\grad\globalEF_k,\localPEF_j\grad\globalEF_k\rangle_{g_\justB}-2\left|\langle\globalEF_k\grad\localPEF_j,\localPEF_j\grad\globalEF_k\rangle_{g_\justB}\right|\\
\end{eqnarray*}
and therefore
\begin{equation*}
\begin{aligned}
\localgquadraticr(\localPEF_j\grad\globalEF_{k},\localPEF_j\grad\globalEF_{k})
&\le
\localgquadraticr(\grad(\localPEF_j\globalEF_{k}),\grad(\localPEF_j\globalEF_{k}))+
2|g_R(\globalEF_{k} \grad \localPEF_j, \localPEF_j \grad\globalEF_{k})|\\
&=
\langle\localPEF_j\globalEF_{k}, \tilde\Delta^\justB(\localPEF_j\globalEF_{k})\rangle_{\justB}+
2|g_R(\globalEF_{k} \grad \localPEF_j, \localPEF_j \grad\globalEF_{k})|\\
&\le
\langle\localPEF_j\globalEF_{k},((\localPEV_j+\globalEV_k)\localPEF_j\globalEF_{k})\rangle_{\justB}+
4|g_R(\globalEF_{k} \grad \localPEF_j, \localPEF_j \grad\globalEF_{k})|\\
&\lesssim
(\localPEV_j+\globalEV_k)||\localPEF_j^2||_\infty||\globalEF_k||^2_{L^2(\justB)}+
 ||\grad\localPEF_j||_\infty||\globalEF_k||_{L^2(\justB)}\localgquadraticr(\localPEF_j\grad\globalEF_{k},\localPEF_j\grad\globalEF_{k})^\frac12\\
&\lesssim
(\localPEV_j+\globalEV_k)(\localPEV_j\Rresc{2})^{2\efk}||\globalEF_{k}||^2_{L^2(\justB)}
+ \localPEV_j\Rresc{}(\localPEV_j\Rresc{2})^\efk ||\globalEF_{k}||_{L^2(\justB)} \localgquadraticr(\localPEF_j\grad\globalEF_{k},\localPEF_j\grad\globalEF_{k})^\frac12
\end{aligned}
\end{equation*}
giving
\begin{gather*}
\localgquadraticr(\localPEF_j\grad\globalEF_{k},\localPEF_j\grad\globalEF_{k})^\frac12=||\localPEF_j\grad\globalEF_{k}\|_{L^2(B_r(z))}
\lesssim
((\localPEV_j+\globalEV_k)^\frac12+\localPEV_j\Rresc{})(\localPEV_j\Rresc{2})^\efk||\globalEF_{k}||_{L^2(\justB)}
\end{gather*}
Finally,
\begin{eqnarray*}
\grad(\localPEF_j \globalEF_k)\leq |\globalEF_k||\grad\localPEF_j| + |\localPEF_j\grad \globalEF_k|
\end{eqnarray*}
gives equation \eqref{14-11-05}.
\end{proof}

\begin{proof}[Proof of Proposition \ref{p:PhiLinftyEstimates}]
We recall that we rescaled so that $R=1$.
Let $\psi=\sum_1^N a_j\localEF_j$ be a (finite) sum of (Euclidean) Dirichlet eigenfunctions of $\justB$ such that
\begin{gather*}
\frac12\leq \psi(x)\leq 2\,\,\,\mathrm{and\ }\,\,\,x\in B_{R/2}(z)\subsetneq \justB
\end{gather*}
and $\sum_1^N |a_j|\leq C$, $\localEV_j\leq C$, $1\leq j\leq N$.
One may obtain such a sequence by taking $\psi'\in C^\infty(\justB)$ with $0\le\psi'\le1$,
$\psi' |_{B_{R/2}(z)}=1$ and $\psi'|_{\partial \justB(z)}=0$ and then take $\psi$ to be a truncation of the eigenfunction expansion of $\psi'$.
Let $\tilde{\psi}=\sum_1^N a_j\localPEF_j$ be the sum of the corresponding Dirichlet eigenfunctions for $\justB$ with respect to $\tilde\Delta^\justB$.
By lemma \ref{lemma-4} and $|g^{ik}(x)-\delta^{ik}|<\epsilon$ (with $\epsilon$ sufficiently small), we have, for $x\in B_{R/2}(z)$,
\begin{gather*}
\frac14\leq \tilde{\psi}(x)\leq 3
\end{gather*}

By Lemma \ref{lemma-6}: 
\begin{equation}
\begin{aligned}
\|\globalEF_j\|_{L^\frac{2d}{d-2}(\justB))}\leq
\|\tilde{\psi} \globalEF_j\|_{L^\frac{2d}{d-2}(\justB))}\leq
\sum |a_i| \| \localPEF_i \globalEF_j\|_{L^\frac{2d}{d-2}(\justB)} \lesssim\\
\sum |a_i|((\localPEV_i+\globalEV_j)^\frac12+\localPEV_i\Rresc{})(\localPEV_i\Rresc{2})^\efk||\globalEF_{j}||_{L^2(\justB)}
\lesssim_C
(\globalEV_j+1)^\frac12 ||\globalEF_{j}||_{L^2(\justB)}\,.
\label{e:chaindualG}
\end{aligned}
\end{equation}

We are now ready to prove inequality \eqref{e:LinftyPhij}.
Let $r_0=R=1>r_1>r_2>\dots\ge\frac R2=\frac12$.
Write $\globalEF_j|_{\ball_{r_0}(z)}$ on  as $\globalEF_j|_{\ball_{r_0}(z)}=u+v$ where
\begin{gather}\label{v-def}
v=\tilde G^\justB(\tilde\Delta^\justB\globalEF_j)=\globalEV_j \tilde G^\justB(\globalEF_j)
\end{gather}
since  $\tilde G^\justB$ is the Green function for the Dirichlet problem on $\ball_{r_0}(z)$.
Hence $\tilde\Delta^\justB u=0$.
We  use (see below) Lemma \ref{lemma-5} in conjunction with the above decomposition, to show that $\globalEF_j\in L^\infty(\ball_{r_\infty}(z))$.
We will then (see below) get \eqref{e:LinftyGradPhij} from differentiating \eqref{v-def}  and  using Lemma \ref{lemma-5}.
Initially, by \eqref{e:chaindualG}, Theorem \ref{GW-thm},  \eqref{v-def} and Young's inequality,
with $p_0=\frac{2d}{d-2}$ and $1\leq p_1=\frac{2d}{d-6+\eta_1}$ (with $0<\eta_1<4$ of our choice, implied by the estimates on the Green function in Theorem \ref{t:WG}), we have
\begin{gather*}
\|v\|_{L^{p_1}(B_{r_0}(z))}\lesssim \globalEV_j\|\globalEF_j\|_{L^{p_0}(B_{r_0}(z))}
\end{gather*}
giving, by Lemma \ref{lemma-5} (since $p_1>p_0>2$),
\begin{gather*}
\|u\|_{L^\infty(B_{r_1}(z))}\lesssim \|u\|_{L^2(B_{r_0}(z))}\lesssim (1 + \globalEV_j)\|\globalEF_j\|_{L^2(B_{r_0}(z))}\,.
\end{gather*}
Thus, we have
\begin{gather*}
\|\globalEF_j\|_{L^{p_1}(B_{r_1}(z))}\lesssim (1+\globalEV_j)\|\globalEF_j\|_{L^{p_0}(B_{r_0}(z))}
\lesssim   (\globalEV_j+1)^\frac32 ||\globalEF_{j}||_{L^2(\justB)}\,.
\end{gather*}
Let $1\leq p_i=\frac{2d}{d-2-4i+\sum_{k\leq i }\eta_k}$ (with $0<\eta_i<4$ of our choice) and $v_i=\tilde G^{r_i}(\tilde\Delta^{r_i}\globalEF_j)$. Similarly, we have
\begin{gather*}
\|v_i\|_{L^{p_i}(B_{r_{i-1}}(z))}\lesssim \globalEV_j\|\globalEF_j\|_{L^{p_{i-1}}(B_{r_{i-1}}(z))}
\end{gather*}
and for $u_i=\globalEF_j-v_i$
\begin{gather*}
\|u_i\|_{L^\infty(B_{r_i}(z))}\lesssim  \|u\|_{L^2(B_{r_{i-1}}(z))}\lesssim (1 + \globalEV_j)\|\globalEF_j\|_{L^2(B_{r_{i-1}}(z))}\,.
\end{gather*}
Thus, we have by induction
\begin{gather*}
\|\globalEF_j\|_{L^{p_i}(B_{r_i}(z))}\lesssim (1+\globalEV_j)\|\globalEF_j\|_{L^{p_{i-1}}(B_{r_{i-1}}(z))}
\lesssim   (\globalEV_j+1)^{i+\frac12} ||\globalEF_{j}||_{L^{p_2}(\justB)}\,.
\end{gather*}

Let $\beta$ be the smallest integer larger or equal than $\frac{d-2}4$.
We may choose $\{\eta_i\}$ so that $p_\beta=\infty$.
This gives  equation \eqref{e:LinftyPhij}.

In order to upper bound $\|\grad \globalEF_j\|$ we note that (recalling that $r_\beta \sim R\sim 1$)
\begin{equation*}
\|\grad v_{\beta}\|_{L^\infty(B_{r_\beta})}
=
\|\grad \tilde G^{r_\beta}(\tilde\Delta^{r_\beta}\globalEF_j)\|_{L^\infty(B_{r_\beta})}
\leq
\globalEV_j \|\tilde G^{r_\beta}\|_{L^1(B_{r_\beta})}       \|\globalEF\|_{L^\infty(B_{r_\beta})}\,.
\end{equation*}
We also note that we have
\begin{equation*}
\|\grad u_{\beta}\|_{L^\infty(B_{\frac12r_\beta})}\lesssim \|\globalEF\|_{L^\infty(B_{r_\beta})}
\end{equation*}
from Lemma \ref{lemma-5}.
Thus combining the last two estimates, we have
\eqref{e:LinftyGradPhij}.

Finally, we prove \eqref{e:LinftyHolderGraphPhij}.
Let $\chi\in C^\infty(\R)$ be a function so that $0\le\chi\le1$, $\chi(s)|_{s\leq K_1}=0$ and $\chi(s)|_{s\geq K_2}=1$.
We  define  $\eta$,  a cutoff function, such that
$\eta|_{\ball(z,\frac14 R)}=1$ and $\eta|_{|x|\geq \frac12 R}=0$ as follows.
Define $\eta(x)=\chi(G(z,x))$, and choose $K_1,K_2$ above so that $\eta$ has the desired cutoff radius.
We get that
\begin{equation*}
\begin{aligned}
\tilde\Delta^\justB(\eta)(x)&= \tilde\Delta(\chi(G(z,x)))=
\sum_{i,j}\partial_{x_i}g^{ij}\partial_{x_j}\chi(G(z,x))
=
\sum_{i,j}\partial_{x_i}(\chi'(G(z,x)))g^{ij}\partial_{x_j} G(z,x)\\
& =
\chi''(G(z,x))\left(\sum_{i,j}\partial_{x_i} G(z,x)g^{ij} \partial_{x_j} G(z,x)) +\chi'(G(z,x)\right) \tilde\Delta_x G(z,x)\,,
\end{aligned}
\end{equation*}
where the second term in the last line is 0 as  $\tilde\Delta_x G(z,x)$ is a distribution which equals
0 on the support of $\chi'(G(z,x))$.
By choice of $\chi$ and Theorem \ref{t:WG}, this gives $\tilde\Delta(\eta)\lesssim1$.

Now, let $x,y\in B=\ball(z,\frac14 R)$.  Let $\tilde{G}=\tilde{G}^B$ be the Green's function for $B$.
\begin{equation}
\begin{aligned}
||\grad \globalEF (x)- \grad \globalEF(y)||
&=
||\grad (\globalEF\eta) (x)- \grad \globalEF\eta)(y)||\\
&=
\left\|\int \left(\grad_1 G (x,w)-\grad_1 G(y,w)\right)  \tilde\Delta^\justB (\globalEF\eta)(w) dw\right\|\\
&\leq
|\tilde\Delta^\justB (\globalEF\eta)(w)|_{L^\infty(\justB)} \int ||\grad_1 G (x,w)-\grad_1 G(y,w)||dw\,.
\end{aligned}
\label{e:gradbounddiff}
\end{equation}
We have (using uniform ellipticity as well as Proposition \ref{p:PhiLinftyEstimates})
\begin{equation}
\begin{aligned}
|\tilde\Delta^\justB (\globalEF\eta)(w)|
&\lesssim
|\eta \tilde\Delta^\justB \globalEF(w)| + | \globalEF(w)\tilde\Delta^\justB \eta | + |\grad \eta| |\grad \globalEF|\\
&\leq
\left( ||\tilde\Delta^\justB \eta||_\infty + \globalEV||\eta||_\infty\right) ||\globalEF||_{L^\infty(\justB)} +
        ||\grad\eta||_\infty ||\grad \globalEF||_{L^\infty(\justB)}\\
&\lesssim
\left( 1 + \globalEV \right) ||\globalEF||_{L^\infty(\justB)} +
         ||\grad \globalEF||_{L^\infty(\justB)}\\
&\lesssim
\left( 1 + \globalEV\right) \efpoly(\globalEV) \|\globalEF\|_{L^2(\justB)} +
        \globalEV  \gefpoly(\globalEV)  \|\globalEF\|_{L^2(\justB)}\\
&\lesssim
\left( (1 + \globalEV) \efpoly(\globalEV) +
        \globalEV  \gefpoly(\globalEV)\right)  \|\globalEF\|_{L^2(\justB)}\\
\end{aligned}
\label{e:LinftyBoundOnLaplacian}
\end{equation}
by using \eqref{e:LinftyGradPhij}.
Combining \eqref{e:gradbounddiff} with \eqref{e:LinftyBoundOnLaplacian} and \eqref{e:gradGholderestimate} we get
\begin{eqnarray}
\|\grad \globalEF (x)- \grad \globalEF(y)\|
&\lesssim&
\ggefpoly(\globalEV) \|\globalEF\|_{L^2(\justB)}  |x-y|^{\alphawedgeone}\,.
\end{eqnarray}
\end{proof}

\begin{proof}[Proof of Lemma \ref{lemma-4'}]
This follows from  Lemma 20 in \cite{GHL_RiemannianGeometry}  together with Proposition \ref{p:PhiLinftyEstimates};
we have $\mathcal{C}^{1+{\alphawedgeone}}$ functions which are close in $L^2(\justB)$ .
Hence, they are also close in $L^\infty(\justB)$, i.e. equation \eqref{lemma-4'-eq-1} holds
and so does \eqref{lemma-4'-eq-2}.
\end{proof}

\subsubsection{Heat kernel estimates}
\label{s:HeatKernelEstimatesManifold}

This subsection makes no assumptions on the finiteness of the volume of $\cM$ and the existence of $\Cweyl$ for the manifold $\cM$.  It will however use these properties for a manifold ball.

We fix a ball $B=\ball_R(x)$ for which we estimate the heat kernel $\localPK^\justB$ by comparing it to $K^\justB$.
Suppose that $\{\localEF_j\}$ is an orthonormal basis for $L^2(\tilde{B})$ (with manifold measure).
In this section all constants subsumed in $\lesssim,\gtrsim$ and $\sim$ will in general depend on $d,c_{\min},c_{\max},||g||_{\alphawedgeone},{\alphawedgeone}$.

\begin{lemma}
Let $A_1>1$ and a sufficiently small $\eta_0=\eta_0(A_1)>0$ be given.
Assume $\epsilon_0$ is sufficiently small (depending on $\eta_0$,  $A_1$, as well as the usual $d$, $c_{\min}$, $c_{\max}$, $||g||_{\alphawedgeone}$, ${\alphawedgeone}$), and $|g^{ik}(x)-\delta^{ik}|<\epsilon_0$.
For
$y\in \ball_{\frac{R}2}(x)\subset \Omega$, with  $|x-y|^2\lesssim t\sim R^2\leq 1$ in a similar fashion to   Assumption \assumptionA,
 we have
\begin{gather}\label{l:lh-2}
\sum_{\localEV_i \leq \frac{A_1}t} \localEF_i(x)\localEF_i(y)e^{-\localEV_i t}
\sim_{\eta_0,A_1,d,c_{\max},c_{\min},||g||_{\alphawedgeone},{\alphawedgeone}}
\sum_{\localPEV_i \leq \frac{A_1}t} \localPEF_i(x)\localPEF_i(y)e^{-\localPEV_i t}\,.
\end{gather}
If in addition we also have $|x-y|^2\sim t$ then
\begin{eqnarray}
\left| \sum_{\localEV_i \leq \frac{A_1}t} \localEF_i(x)\grad\localEF_i(y)e^{-\localEV_i t}
-
\sum_{\localPEV_i \leq \frac{A_1}t} \localPEF_i(x)\grad\localPEF_i(y)e^{-\localPEV_i t} \right|
\lesssim_{A_1,d,c_{\max},c_{\min},||g||_{\alphawedgeone},{\alphawedgeone}}
\ \     \eta_0\cdot \frac{R}{t}t^\frac {-d}2 \,.
\label{l:grad-lh-3}
\end{eqnarray}
The constants in \eqref{l:lh-2} 
go to 1 as $\eta_0\to 0$.
\label{large-head}
\end{lemma}
\begin{proof}
We apply Lemma \ref{lemma-4} with $J=\#\{j:\localEV_i\le A_1/t\}\le \left(\frac{A_1}t\right)^{\frac d2} R^d\sim A_1^\frac d2$ and with $\eta<\eta_0$.
Let $\epsilon_0$ be as guaranteed by Lemma \ref{lemma-4}.
Since  $\localEF_i$'s and $\localPEF_i$'s are $L^2$-normalized, Lemma \ref{lemma-ef-bound} and \ref{lemma-4} implies
for $\localEV_i \leq \frac{A_1}t$
\begin{equation*}
\begin{aligned}
&|\localEF_i(x)\localEF_i(y)e^{-\localEV_i t} - \localPEF_i(x)\localPEF_i(y)e^{-\localPEV_i t}|\\
&\leq
|\localEF_i(x)-\localPEF_i(x)||\localEF_i(y)|e^{-\localEV_i t} +
|\localEF_i(y)-\localPEF_i(y)||\localEF_i(x)|e^{-\localEV_i t} +
|\localEF_i(x)||\localEF_i(y)||e^{-\localEV_i t}- e^{-\localPEV_i t}|\\
&\lesssim Q_1(A_1 t^{-1}R^2)
\eta  \left( |\localEF_i(y)|e^{-\localEV_i t} + |\localEF_i(x)|e^{-\localEV_i t}\right) + |\localEF_i(x)||\localEF_i(y)|t\eta\localEV_i e^{-\localEV_i t}\\
&\lesssim A_1^{3\efk+1} \eta
\end{aligned}
\end{equation*}

Using Weyl's Lemma (Lemma \ref{l:weyls-lemma}) for the ball with Dirichlet boundary conditions (see Lemma \ref{l:weyls-lemma}), we have
\begin{equation*}
\begin{aligned}
\left\|\sum_{\localEV_i \leq \frac{A_1}t} \localEF_i(x)\localEF_i(y)e^{-\localEV_i t} -
\sum_{\localPEV_i \leq \frac{A_1}t} \localPEF_i(x)\localPEF_i(y)e^{-\localPEV_i t}\right\|
&\lesssim A_1^{3\efk}(1+A_1) \eta J\\
& \lesssim A_1^{3\efk+1} \eta \sum_{\localEV_i \leq \frac{A_1}t} \localEF_i(x)\localEF_i(y)e^{-\localEV_i t}\\
\end{aligned}
\end{equation*}
by the (Euclidean) estimates in the proof of Lemma \ref{l:truncated_kernel} and since $R\lesssim 1$.
We obtain the desired estimate \eqref{l:lh-2} by taking $\eta_0$ sufficiently small.
Similarly,
\begin{equation*}
\begin{aligned}
&|\localEF_i(x)\grad\localEF_i(y)e^{-\localEV_i t} - \localPEF_i(x)\grad\localPEF_i(y)e^{-\localPEV_i t}|\\
&\leq
|\localEF_i(x)-\localPEF_i(x)||\grad\localEF_i(y)|e^{-\localEV_i t} +
|\grad\localEF_i(y)-\grad\localPEF_i(y)||\localEF_i(x)|e^{-\localEV_i t} +
|\localEF_i(x)||\grad\localEF_i(y)||e^{-\localEV_i t}- e^{-\localPEV_i t}|\\
&\lesssim
\eta  \left( (A_1 t^{-1}R^2)^{\efk}|\grad\localEF_i(y)|e^{-\localEV_i t} + (A_1 t^{-1}R^2)^{\efk+1} r^{-1}|\localEF_i(x)|e^{-\localEV_i t}\right) + |\localEF_i(x)||\grad\localEF_i(y)|t\eta\localEV_i e^{-\localEV_i t}\\
&\lesssim A_1^{3\efk+2} \eta R^{-1}\,.
\end{aligned}
\end{equation*}
Thus, equation \eqref{l:grad-lh-3} also clearly follows by $\eta_0$ sufficiently small.
\end{proof}

\begin{lemma}
\label{1-7-07}
Let $\eta_0>0$ be given and assumed to be sufficiently small.
Assume $\epsilon_0$ is sufficiently small (depending on $\eta_0$, as well as the usual   $d$, $c_{\min}$, $c_{\max}$, $||g||_{\alphawedgeone}$, ${\alphawedgeone}$), and
$|g^{ik}(x)-\delta^{ik}|<\epsilon_0$.
For $y\in B_{\frac{R}2}(x)\subset \Omega$ with $ |x-y|^2\lesssim t\sim R^2\leq1$ (in a similar fashion to Assumption \assumptionA) and $s\leq t$,
\begin{gather}\label{large-tilde-K}
\localPK^B_t(x,y)
\sim_{\eta_0,d,c_{\min},c_{\max},||g||_{\alphawedgeone},{\alphawedgeone}} %
K^B_t(x,y)\,,
\end{gather}
\begin{gather}\label{s_large-tilde-K}
\localPK^B_s(x,y)
\lesssim_{\eta_0,d,c_{\min},c_{\max},||g||_{\alphawedgeone},{\alphawedgeone}} %
K^B_s(x,y)\,,
\end{gather}
and
\begin{gather}\label{s-large-grad-tilde-K2}
\|\grad \localPK^B_s(x,y)\|
\lesssim_{\eta_0,d,c_{\min},c_{\max},||g||_{\alphawedgeone},{\alphawedgeone}}
\frac{R}{s}(sR^{-2})^{-2\efk-1}s^\frac{-d}2\,.
\end{gather}
If in addition we have $|x-y|^2\sim t$ then
\begin{equation}\label{e:ball-kernels-close}
\left\|\grad \localPK^B_t(x,y) - \grad K^B_t(x,y)\right\|
\lesssim_{d,c_{\min},c_{\max},||g||_{\alphawedgeone},{\alphawedgeone}}
\ \ \eta_0 \cdot\frac{R}{t}t^\frac {-d}2\,.
\end{equation}
The constants in \eqref{large-tilde-K} 
go to $1$ as $\eta_0\to 0$.
\end{lemma}
\begin{proof}
We estimate the tail :
\begin{equation*}
\begin{aligned}
\left\|\sum_{\localEV_i \geq \frac{A_1}t} \localEF_i(x)\localEF_i(y)e^{-\localEV_i t}\right\|
&\le
e^{-\frac12 A_1}
\left\|\sum_{\localEV_i \geq \frac{A_1}t} \localEF_i(x)\localEF_i(y)e^{-\frac12\localEV_i t}\right\| \\
&\le
e^{-\frac12 A_1} \localPK^B_{\frac14t}(x,x)\localPK^B_{\frac14t}(y,y)
\underbrace{\lesssim}_{{\rm using\  }\cite{Davies:SpectraPropertiesChangesMetric}}
e^{-\frac12 A_1} t^{-\frac{d}2}\,.
\end{aligned}
\end{equation*}
This, combined with  \eqref{l:lh-2}, for $A_1$ large enough, gives \eqref{large-tilde-K}. From \cite{Davies:SpectraPropertiesChangesMetric} we also get  \eqref{s_large-tilde-K}.
We also have
\begin{eqnarray*}
\left\|\grad_x
    \sum_{\localPEV \geq \frac{A_1}s} \localPEF_i(x)\localPEF_i(y)e^{-\localPEV_i s} \right\|
&=&
\left\|\int_{\justB} \grad_x \tilde{G}^\justB(x,w)
    \sum_{\localPEV \geq \frac{A_1}s} \tilde\Delta^\justB \localPEF_i(w)\localPEF_i(y)
                e^{-\localPEV_i s} \right\|\\
&=&
\left\|\int_{\justB} \grad_x \tilde{G}^\justB(x,w)
    \sum_{\localPEV \geq \frac{A_1}s} \localPEV_i\localPEF_i(w)\localPEF_i(y)
                e^{-\localPEV_i s} \right\|\\
&=&
s^{-1} \left\|\int_{\justB} \grad_x \tilde{G}^\justB(x,w)
    \sum_{\localPEV \geq \frac{A_1}s} \localPEF_i(w)\localPEF_i(y)
                (\localPEV_ise^{-\localPEV_i s}) \right\|\\
&\lesssim&
s^{-1} \int_{\justB} ||\grad_x \tilde{G}^\justB(x,w)||
    \sum_{\localPEV \geq \frac{A_1}s} |\localPEF_i(w)| |\localPEF_i(y)|
                e^{-\frac12 \localPEV_i s} \\
&\lesssim&
e^{-\frac14 A_1} s^{-1} \int_{\justB} ||\grad_x \tilde{G}^\justB(x,w)||
        \localPK^B_{s/8}(w,w)^\frac12\localPK^B_{s/8}(y,y)^\frac12\\
&\lesssim&
e^{-\frac14 A_1} s^{-\frac{d}2 -1} \int_{\justB} ||\grad_x \tilde{G}^\justB(x,w)||\lesssim e^{-\frac14 A_1} s^{-\frac{d}2 -1} R\,,
\end{eqnarray*}
since by \eqref{e:annulusintgradG} we have $||\grad \tilde{G}^\justB(x,\cdot)||_{L^1(B_\Rrescb)} \lesssim R$.
If we now take $s=t$ then, by the Euclidean estimates and \eqref{l:grad-lh-3}, for $A_1$ large enough, we obtain both the lower and upper bounds
\eqref{e:ball-kernels-close}.

To prove estimate \eqref{s-large-grad-tilde-K2}, we use the above estimate and notice that we also have (from Lemma \ref{lemma-ef-bound} and Weyl's Lemma for the ball
with Dirichlet boundary conditions)
\begin{equation*}
\begin{aligned}
\left\|\sum_{\localPEV \leq \frac{A_1}s} \grad \localPEF_i(x)\localPEF_i(y)e^{-\localPEV_i s}\right\|
&\lesssim
R^{-d}\sum_{\localPEV \leq \frac{A_1}s} \localPEV_i R(\localPEV_i R^2)^{2\efk+1}e^{-\localPEV_i s}\\
&=
R^{-d}(sR^{-2})^{-2\efk-2}R^{-1}\sum_{\localPEV \leq \frac{A_1}s} (\localPEV_i s)^{2\efk+2}e^{-\localPEV_i s}\\
&\lesssim
R^{-d}(sR^{-2})^{-2\efk-2}R^{-1}\sum_{\localPEV \leq \frac{A_1}s}1 \\
&\lesssim
A_1^\frac{d}2 R^{-1}(sR^{-2})^{-2\efk-2}s^\frac{-d}2\\
&\lesssim
\frac{R}s (sR^{-2})^{-2\efk-1}s^\frac{-d}2\,.
\end{aligned}
\end{equation*}
\end{proof}

Lemma \ref{1-7-07} will be used to get  Proposition \ref{p:non-smooth-kernel_estimates} for the case of a manifold.
We will  need to improve estimate \eqref{s-large-grad-tilde-K2}, which in turn requires the following:
\begin{lemma}
Let $|y|<\frac{R}4$, $r<\frac{R}4$ and $s^\frac12 \leq r$.
Let $B^y(s')$ be Brownian motion started at $y$.  Then
$$P(\sup\limits_{0\leq s'\leq s} |B^y(s')|>|y|+r)\lesssim_{d,c_{\min},c_{\max}} e^{-c'(d,c_{\min},c_{\max}) \frac{r}{\sqrt{s}}}\,.$$
\label{exp-decay-time}
\end{lemma}
\begin{proof}
This follows from Lemma \ref{Ptauleqt}.
\end{proof}

\begin{proof}[Proof of Proposition \ref{p:non-smooth-kernel_estimates}; case of  $\localTangentBall$ with metric at least $\mathcal{C}^2$]
By rescaling we may assume  that $R\leq 1$.
%
We upper bound $\delta_0$ so that $|g^{ik}(x)-\delta^{ik}|<\epsilon_0$ where $\epsilon_0$ is as prescribed by Lemma \ref{1-7-07} (this is done as in \eqref{e:epsilon_0-is-used}).

Estimates \eqref{e:kernel_estimates} and the first part of  \eqref{e:kernel_estimates_s} follow  from the Euclidean case and estimates \eqref{large-tilde-K} and \eqref{s_large-tilde-K}.
Estimate \eqref{e:gradient_kernel_estimates} and
estimate \eqref{e:gradient_kernel_estimates_q} follow from \eqref{e:ball-kernels-close} and Euclidean ball estimates.

We now turn to the second and third parts in   \eqref{e:kernel_estimates_s}.
Without loss of generality we identify $z=0$.
Let $a$ be such that $a\sum\limits_{j=1}^\infty\frac1{j^2}=\frac14$.
Define stopping times $\tau_1,\tau_2,...$ by
$$\tau_n=\inf\{s':|B^{z}(s')|=aR\sum\limits_{j=1}^n\frac1{j^2}\}\,.$$

For $n>1$, define the set of paths
\begin{equation*}
B_n=\{\omega \in {\bf \Omega}: \tau_n(\omega)\leq (1-2^{-n})s \}\,.
\end{equation*}
For $n>1$ define $G_n\subset B_n$ as
\begin{equation*}
G_n=B_n\smallsetminus B_{n-1}
\end{equation*}
and
\begin{equation*}
G_1=B_1
\end{equation*}
We estimate using Lemma \ref{exp-decay-time}:
\begin{equation*}
P(G_1)\le
\exp\left(-c' \frac{a}{1^2} Rs^{-\frac12}\right)\\
\end{equation*}
and for $n>1$
\begin{equation*}
P(G_n)
\leq  \exp\left(-c'  \frac{a2^\frac{n-1}2}{n^2} Rs^{-\frac12}\right)\,.
\end{equation*}
We need another lemma:
\begin{lemma}
The set $\{\omega\in{\bf \Omega}:  \tau_n \leq s\ \forall (n\geq1),\quad \omega\notin \cup G_n\}$ has probability 0.
\end{lemma}
\begin{proof}
\begin{eqnarray*}
\{\omega\in{\bf \Omega}:  \tau_n \leq s\ \forall (n\geq1),\quad \omega\notin \cup G_n\}
&=&\{\omega\in{\bf \Omega}:  \tau_n \leq s\ \forall (n\geq1),\quad \omega\notin \cup B_n\}\\
&=&
\{\omega\in{\bf \Omega}:  s\geq \tau_{n}\geq (1-2^{-n})s \ \forall(n>1)\}\\
&\subset&
\{\omega\in{\bf \Omega}:  \tau_n-\tau_{n-1}\leq 2^{-n}s \ \forall(n>1)\}
\end{eqnarray*}
However the set $\{\omega\in{\bf \Omega}:  \tau_n-\tau_{n-1}\leq 2^{-n}s \}$
has probability decaying super-exponentially in $n$ by Lemma \ref{exp-decay-time}.
\end{proof}
We now continue with the proof of Proposition \ref{p:non-smooth-kernel_estimates}; case of
$\localTangentBall$ with metric at least $\mathcal{C}^2$.
Define $H_n=G_n\setminus (\cup_1^{n-1}G_i)$.
We now have a disjoint partition of $\{\omega\in {\bf \Omega}:\tau_n \leq s\ \forall (n\geq1)\}$ (up to measure $0$) by the collection $\{H_i\}$.
Set $\tilde K^D_s(\cdot,\cdot):=\localPK^\localTangentBall_s(\cdot,\cdot)$.
For $|y|> \frac{R}4$
 we have
\begin{equation*}
\localPK^D_s(z,y_0)=\sum\limits_{n=1}^\infty \E_\omega(\chi_{H_n} \localPK^D_{s-\tau_n}(B^{z}(\tau_n),y))
\end{equation*}
Taking gradient and using equation \eqref{s-large-grad-tilde-K2}  we get
\begin{equation*}
\begin{aligned}
|\grad_x\localPK^D_s(z,y)|
&=
|\sum\limits_{n=1}^\infty \E_\omega(\chi_{H_n} \grad_x\localPK^D_{s-\tau_n}(B^{z}(\tau_n),y))| \\
&\lesssim_{d,c_{\min},c_{\max},||g||_{\alphawedgeone},{\alphawedgeone}} %
\sum\limits_{n=1}^\infty \frac{R}s (2^{-n}sR^{-2})^{-2\efk-1}s^\frac{-d}2 P(H_n)\\
&\lesssim_{d,c_{\min},c_{\max},||g||_{\alphawedgeone},{\alphawedgeone}} %
\sum\limits_{n=1}^\infty \frac{R}s (2^{-n/2}s^\frac12R^{-1})^{-4\efk-2}s^\frac{-d}2  \exp\left(-c'  \frac{a2^{-1/2}}{n^2} 2^{n/2}s^{-\frac12}R\right)\\
&\lesssim_{d,c_{\min},c_{\max},||g||_{\alphawedgeone},{\alphawedgeone}} %
\sum\limits_{n=1}^\infty \frac{R}t(2^{-n/2}t^\frac12R^{-1})^{-4\efk-2}t^\frac{-d}2  \exp\left(-c'  \frac{a2^{-1/2}}{n^2} 2^{n/2}t^{-\frac12}R\right)\\
&\leq
\frac{R}t t^\frac{-d}2\sum\limits_{n=1}^\infty (2^{-n/2}t^\frac12R^{-1})^{-4\efk-2}  \exp\left(-c'  \frac{a2^{-1/2}}{n^2} 2^{n/2}t^{-\frac12}R\right)\\
&\lesssim_{d,c_{\min},c_{\max},||g||_{\alphawedgeone},{\alphawedgeone}} %
\frac{R}t t^\frac{-d}2
\end{aligned}
\end{equation*}
where we may replace $s$ with $t$ above, since $s\leq t$, and each of the summands is increasing in $s$  as long as it is sufficiently small with respect to $R^2$ (independently of $n$ when $n>1$).
This proves the second and third parts of \eqref{e:kernel_estimates_s} for $\localTangentBall$ with metric at least $\mathcal{C}^2$.
\end{proof}

\begin{remark}
The proof below makes no assumption on the volume of $\mathcal{M}$, and works for the case of $\mathcal{M}$ having infinite volume as well.
\end{remark}
\begin{proof}[Proof of Proposition \ref{p:non-smooth-kernel_estimates} for the heat kernel of $\mathcal{M}$, with metric at least $\mathcal{C}^2$]
As for the Neumann heat kernel, the starting point is Proposition \ref{p:LocalToGlobalHeatKernel}, which allows us to localize.
We use Proposition \ref{p:non-smooth-kernel_estimates} for the ball $B_{2\delta_0R_z}(z)$  with metric at least $\mathcal{C}^2$.
For this proof, we denote by  $C_2[B]$ be the $C_2$ constant for the Dirichlet ball case, and
set $ K^D_s(\cdot,\cdot):=K^{2\delta_0R_z}_s(\cdot,\cdot)$,
the heat kernel for the ball $\ball(z,2\delta_0R_z)$ with Dirichlet boundary conditions.
For $s\le t$,
\begin{equation*}
\begin{aligned}
|K_s(x,y)- \tilde K_s^D(x,y)|&=
\left|\sum_{n=1}^{+\infty}\mathbb{E}_\omega\left[ \tilde K_{s-\tau_n}^D(x_n(\omega),y)|\tau_n<s\right]P_{\omega}(\tau_n(\omega)<s)\right|\\
&\lesssim_{C_2[B]}
\sum_{n=1}^\infty t^{-\frac d2} \underbrace{e^{-n\frac{\left(\frac{\delta_0R_z}2\right)^2}{Ms}}}_{\mathrm{eqn.\ } \eqref{e:taunexpbound}}\\
& \lesssim_{C_2[B],\delta_0,\delta_1} t^{-\frac d2} e^{-\frac{\left(\frac{\delta_0R_z}2\right)^2}{Ms}}
\end{aligned}
\end{equation*}
This proves \eqref{e:kernel_estimates} and the first part of \eqref{e:kernel_estimates_s} (see Remark \ref{r:expsmall}).
For the gradient estimates, i.e.  \eqref{e:gradient_kernel_estimates}, \eqref{e:gradient_kernel_estimates_q},  and the second and third part of \eqref{e:kernel_estimates_s},
\begin{equation*}
\begin{aligned}
\|\grad_y K_s(x,y)- \grad_y  \tilde K_s^D(x,y)\|
&\le \sum_{n=1}^{+\infty}\left\|\grad_y \mathbb{E}_\omega\left[ \tilde K_{s-\tau_n}^D(x_n(\omega),y)|\tau_n<s\right]\right\|P_{\omega}(\tau_n(\omega)<s)\\
&\lesssim_{C_2'[B]}
\sum_{n=1}^\infty t^{-\frac d2}\frac{\delta_0R_z}t \underbrace{e^{-n\frac{\left(\frac{\delta_0R_z}2\right)^2}{Ms}}}_{\mathrm{eqn.\ } \eqref{e:taunexpbound}}\\
& \lesssim_{C_2'[B],\delta_0,\delta_1}  t^{-\frac d2}\frac{\delta_0R_z}t e^{-\frac{\left(\frac{\delta_0R_z}2\right)^2}{Ms}}
\end{aligned}
\end{equation*}
giving us $C_9$.
By Remark \ref{r:expsmall} the exponential term from equation \eqref{e:taunexpbound} can be made small enough so that we obtain estimate \eqref{e:gradient_kernel_estimates}
as well as the second and third parts of \eqref{e:kernel_estimates_s}.
\end{proof}

\section{The proof of Theorem \ref{t:heatkernelmapping}}
We remind the reader of Remark \ref{r:inf-vol} which notes that the proof
 of Proposition \ref{p:non-smooth-kernel_estimates} for the heat kernel of $\mathcal{M}$, with metric at least $\mathcal{C}^2$ (appearing  at the end of section \ref{s:HeatKernelEstimatesManifold}),
 made no assumptions on the finiteness of the volume of $\cM$ and the existence of $\Cweyl$.

 \subsection{The case $g\in\mathcal{C}^2$}

We  appropriately choose heat kernels $\{K_t(z,y_i)\}_{i=1,\dots,d}$, with $t\sim R_z^2$,
that provide a local coordinate chart with the properties claimed in the Theorem \ref{t:heatkernelmapping}:

\begin{proof}[Proof of Theorem \ref{t:heatkernelmapping}  for $g\in\mathcal{C}^2$]
Without loss of generality we may assume $\rho=R_z=1$, and thus, by Remark \ref{r:inf-vol}, we may apply Proposition \ref{p:non-smooth-kernel_estimates}.
Let us consider the Jacobian $\tilde J(x)$, for $x\in\ball_{c_1R_z}(z)$, of the map
\begin{equation*}
\tilde \Phi := R_z^{-d} t^{d/2}(t/R_z^2) \Phi\,.
\end{equation*}
By 
\eqref{e:gradient_kernel_estimates_q} we have $|\tilde J_{ij}(x)-C_2'\langle p_i,\frac{x-y_j}{||x-y_j||}\rangle R_z^{-1}|\le C_9 R_z^{-1}$.
As dictated by Proposition \ref{p:non-smooth-kernel_estimates},
by choosing $\delta_0,\delta_1$ appropriately (and, correspondingly, $c_1$ and $c_6$), we can make the constant $C_9$ smaller than any chosen $\epsilon$,
for all entries, and for all $x$ at distance no greater than $c_1R_z$
from $z$, where we use $t=t_z=c_6R_z^2$ for $\tilde \Phi$.
Therefore for $c_1$ small enough compared to $c_4$ we can write $R_z \tilde J(x)=G_d+E(x)$ where $G_d$ is the Gramian matrix $\langle p_i,p_j\rangle$ (indepedent of $x$!),
and $|E_{ij}(x)|<\epsilon$, for $x\in\ball_{c_1R_z}(z)$.
This implies that $R_z^{-1}(\sigma_{\min}-C_d\epsilon)||v||\le ||\tilde J(x)v||\le R_z^{-1}(\sigma_{\max}+C_d\epsilon)||v||$, with $C_d$ depending linearly on $d$,
where $\sigma_{\max}$ and $\sigma_{\min}$ are the largest and, respectively, smallest eigenvalues of $G_d$.
At this point we choose $\epsilon$ small enough, so that the above bounds imply that the Jacobian is essentially constant in $\ball_{c_1R_z}(z)$, and
by integrating along a path from $x_1$ to $x_2$ in $\ball_{c_1R_z}(z)$, we obtain
the Theorem  ($\Phi$ and $\tilde \Phi$ differ only by scalar multiplication).
We note that  $\epsilon\sim\frac1d$ suffices.
\end{proof}

We discuss the proof for $g\in\mathcal{C}^\alpha$, and $\mathcal{M}$ has possibly infinite volume in Section \ref{s:infinitevolume}.
Such proof is based on approximation arguments via heat kernels corresponding to smooth metrics on finite volume submanifolds.

\subsection{The case $g\in\mathcal{C}^\alpha$}
\label{s:infinitevolume}

In this section we discuss heat kernel estimates and the heat kernel triangulation Theorem in the case when $\mathcal{C}^\alpha$.
The key ingredient for the proof of Theorem \ref{t:heatkernelmapping} for the case of $g\in\mathcal{C}^\alpha$, are the heat kernel estimates similar to those of  Proposition \ref{p:non-smooth-kernel_estimates}.

Before
we  turn to the proof
of Theorem \ref{t:heatkernelmapping} for the case $g\in\mathcal{C}^\alpha$,
we need one more statement about the case $g\in\mathcal{C}^2$.
Consider the following   variant of Proposition \ref{p:LocalToGlobalHeatKernel}.

\begin{proposition}[Variant of Proposition \ref{p:LocalToGlobalHeatKernel}]
Assume $g\in\mathcal{C}^2$.
Let $w\in\Omega$ and $R_w\le\mathrm{dist}(w,\partial\Omega)$, or $w\in\mathcal{M}$ and $R_w\le \injRad(w)$.
Let $z$ and $R_z$ be similarly defined.
Assume 
$z\notin\ball_{ R_w}(w)$,
and
$w\notin\ball_{R_z}(z)$.
For each path $B_\omega^z$ (starting at $z$), we define $\tau_1(\omega)\le \tau_2(\omega)\le\dots$ as follows.
Let $\tau_1(\omega)$ be the first time that $B_\omega^z$ enters $\ball(w,\frac34 R_w)$
(if this does not happen, let $\tau_1(\omega)=+\infty$). Let $z_1=B_\omega^z(\tau_1)$.
By induction, for $n>1$ let $\tau_n(\omega)$ be the first time after $\tau_{n-1}(\omega)$ that $B_\omega^z$ re-enters $\ball(w,\frac34 R_w)$ after having exited $\ball(w,\frac12 R_w)$,
or $+\infty$ otherwise. Let $z_n(\omega)=B_\omega^z(\tau_n)$. If $\tau_n(\omega)=+\infty$, let $\tau_{n+k}(\omega)=+\infty$ for all $k\ge0$.
Then
\begin{equation}
K_s(z,w)= \sum_{n=1}^{+\infty}
    \mathbb{E}_\omega\left[K_{s-\tau_n(\omega)}^D(z_n(\omega),w)\Bigg|\tau_n<s\right]P_\omega(\tau_n<s)\,,
\label{e:variant-PetersMagic}
\end{equation}
where
\begin{equation*}
K^D_s=K^{Dir(\ball_{\frac12 R_w}(w))}_s\,.
\end{equation*}
Moreover there exists an $M=M(c_{\min},c_{\max})$ such that
\begin{equation}
P(\tau_n<s)\lesssim_{d,M,c_{\min},c_{\max}}
\exp\{-(n-1)\left(\frac{R_w}{8}\right)^2(2Ms)^{-1} -{\left(\frac{R_z}8\right)^2}(2Ms)^{-1} \}\,.
\label{e:variant-Ptaunlessthant}
\end{equation}
\label{p:variant-LocalToGlobalHeatKernel}
\end{proposition}
The proof of this Proposition is along the same lines as that of Proposition \ref{p:LocalToGlobalHeatKernel}.
%


\begin{proof}[Proof of Theorem \ref{t:heatkernelmapping} for $g\in\mathcal{C}^\alpha(\mathcal{M})$ with $\vol{\mathcal{M}}\leq\infty$]
Consider a sequence of metrics $\{g_k\}\subseteq\mathcal{C}^2(\mathcal{M})$, with increasing compact supports $\{\mathcal{M}_k\}$, converging to $g$ in $\mathcal{C}^\alpha$
(and therefore bounded in $\mathcal{C}^\alpha$),  and such that $g_k$ is uniformly elliptic with constants $\frac12c_{\min},2c_{\max}$ (which is possible
since $c_{\min}$ and $c_{\max}$ are continuous functions of the components of the metric tensor).
Let $K_k$ be the heat kernel associated with $g_k$. Note that for this heat kernel and its gradient we have bounds, from above 
with
constants uniform in $k$ for any fixed compact $\mathcal{E}$ away from $\partial \mathcal{M}$.
We proceed as in the proof of Theorem II.3.1 in \cite{Stroock:DiffusionSemigroups}.
The key ingredients are uniform (in $k$, for a fixed compact) upper bounds on $K_k$ (which follow from Propositions \ref{p:LocalToGlobalHeatKernel} and \ref{p:variant-LocalToGlobalHeatKernel}), and that $\{K_k\}$ is equicontinuous, which follows from the uniform upper bounds on the gradient of $K_k$
(for a fixed compact we have uniform lower bounds on $R_z$ and $R_w$ and estimate \eqref{e:variant-Ptaunlessthant}).
It could also made
follow from Stroock's paper (Nash-Moser estimates that say the $K_k$ is H\"older of order and with constants depending only the ellipticity constants)).
The proof of Theorem II.3.1 in \cite{Stroock:DiffusionSemigroups} then implies that $K_k\rightarrow K_\mathcal{M}$ as $k\rightarrow+\infty$, uniformly on compacts.
Therefore the uniform (in $k$) bi-Lipschitz bounds on the map
$$x\to (K_{k,t}(x,y_1),...,K_{k,t}(x,y_d))$$
on $\localTangentBall$, imply the same bounds for
$$x\to (K_{\mathcal{M},t}(x,y_1),...,K_{\mathcal{M},t}(x,y_d))\,.$$
\end{proof}

\section{Examples}
\label{s:examples}

\subsection{Localized eigenfunctions}\label{s:localized-ef}

\begin{figure}[h]
\centering
\scalebox{0.5}{\includegraphics*[0in,0in][8in,8in]{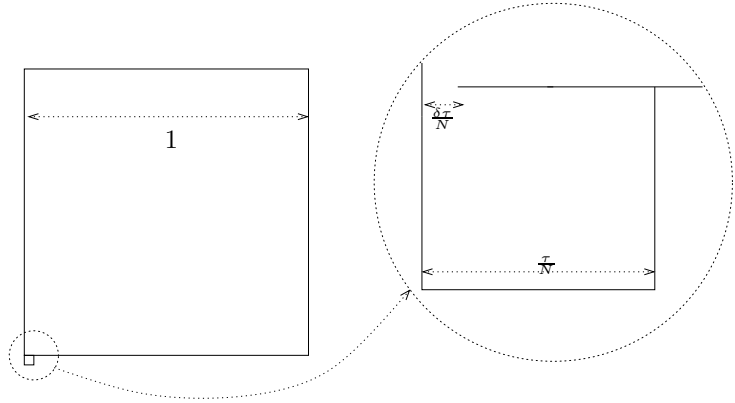}}
\put(-230,95){$1$}
\put(-90,50){\tiny $\frac{\tau}{N}$}
\put(-130,105){\tiny $\frac{\delta\tau}{N}$}
\caption{Example of localization}
\label{f:example-localization}
\end{figure}

The following example shows that the factors $\gamma_1,\dots,\gamma_d$ in Theorems \ref{t:datheorem} and \ref{t:datheoremmanifold} may in fact be required to be as small as
$R_z^{\frac d2}$.

Let $\tau$ below be the golden ratio.
Consider the domain $\Omega_\delta$ as in Figure \ref{f:example-localization}, with Dirichlet boundary conditions.
We will let $z$ to be the center of the small square.
Let $\globalEV_j^\delta$ and $\globalEF_j^\delta$ be the eigenvalues and eigenfunctions on $\Omega_\delta$.
Fix $A>CN$.
Let $\Lambda^\delta_A=\{j:\globalEV_j^\delta\le A^2\}$.
The cardinality of $\Lambda^\delta_A$ is uniformly bounded above by Weyl Lemma.
It is also bounded below, since in this case we can easily obtain a reverse Weyl Lemma. To see this, let
$(\partial\Omega_\delta)_A=\{x\in\Omega: d(x,\partial\Omega)\le a/A\}$, so that the heat kernel estimates in Lemma \ref{l:truncated_kernel} hold for
$t=b A^{-2}$, and observe that
\begin{equation*}
\begin{aligned}
\#\{j:\lambda_j^\delta\le A^2\}&\ge e^{+b}\int_{\Omega_\delta\setminus(\partial\Omega_\delta)_A} \sum_{\lambda_j^\delta\le A} |\globalEF_j^\delta(x)|^2e^{-\lambda_j^\delta t}\\
&\ge e^{+b}\left(\int_{\Omega_\delta\setminus(\partial\Omega_\delta)_A} K_t^\delta(x,x)-\int_{\Omega_\delta\setminus(\partial\Omega_\delta)_A} K_{t/2}^\delta(x,x)e^{-\frac{A^2t}2}\right)\\
&\gtrsim e^{+b} \left(\frac12-\frac aA\right)^2\left( C_1t^{-1}-C_2 e^{-\frac b2}t^{-1}\right)\gtrsim
e^{+b} b^{-1} A^2
\end{aligned}
\end{equation*}
where we choose $b$ so that the last inequality holds, and thus determine $a$ and $A$, which are chosen so that Lemma \ref{l:truncated_kernel} and Proposition \ref{p:non-smooth-kernel_estimates} hold.

Fix $j=1$. The sequence $\{\globalEV_j^\delta\}_{\delta>0}$ is bounded and hence the families $\{\globalEF_j^\delta\}_{\delta>0}$ and $\{\grad\globalEF_{j}^\delta\}_{\delta\ge0}$ are equicontinuous  (by
Proposition \ref{p:PhiLinftyEstimates}), therefore there exists a sequence $\delta_k\rightarrow0$ such
that $\globalEF_j^{\delta_k}\rightarrow\varphi_j$, $\grad\globalEF_j^{\delta_k}\rightarrow\grad\varphi_j$ and $\globalEV_j^{\delta_k}\rightarrow\globalEV_j$.
We can repeat this argument for any $j\le \liminf_k \#\{j:\lambda_j^{\delta_k}\le A^2\}:=j_{\max}(A)$, which is strictly positive and tending to $+\infty$ as $A\rightarrow+\infty$,
by the above.
By a diagonal argument, we can find a subsequence $\delta_l$ such that for any $j\le j_{\max}(A)$,
$\globalEF_j^{\delta_l}\rightarrow\varphi_j$, $\grad\globalEF_j^{\delta_l}\rightarrow\grad\varphi_j$.
Let us look at some properties of $\varphi_j$.
Clearly, $\varphi_j$ is an eigenfunction for $\Delta$ with Dirichlet boundary conditions on $\Omega_0$.
Since $(\frac{\tau}N n_1)^2+(\frac{\tau}N n_2)^2$ is irrational for any $n_1,n_2\in\mathbb{Z}$, every $\varphi_j$ is supported in either the small square, or the big square.
Recall that $z$ is the center of the small square. For any $j\le j_{\max}(A)$ if $\varphi_j$ has support in $S_0$ then $||\grad\varphi_j||\gtrsim (\tau/N)^{-2}$.
Let $\delta_l$ be small enough so that $||\grad\varphi_j^{\delta_l}||\gtrsim (\tau/N)^{-2}$, for all $j\le j_{\max}(A)$.
By choosing $A$ larger than $c_5(\tau/{2N})$, where $c_5$ is as in Theorem \ref{t:datheorem}, all possible eigenfunctions that may get chosen in
the Theorem will correspond to $j\le j_{\max}(A)$, and therefore the lower bound for the $\gamma_i$ is sharp.

See \url{http://pmc.polytechnique.fr/pagesperso/dg/recherche/localization_e.htm} for nice demonstrations of the above example.

\subsection{Non-simply connected domain}

\begin{figure}[h!]
\begin{minipage}[t]{\textwidth}
\includegraphics[width=0.32\textwidth,height=0.4\textwidth]{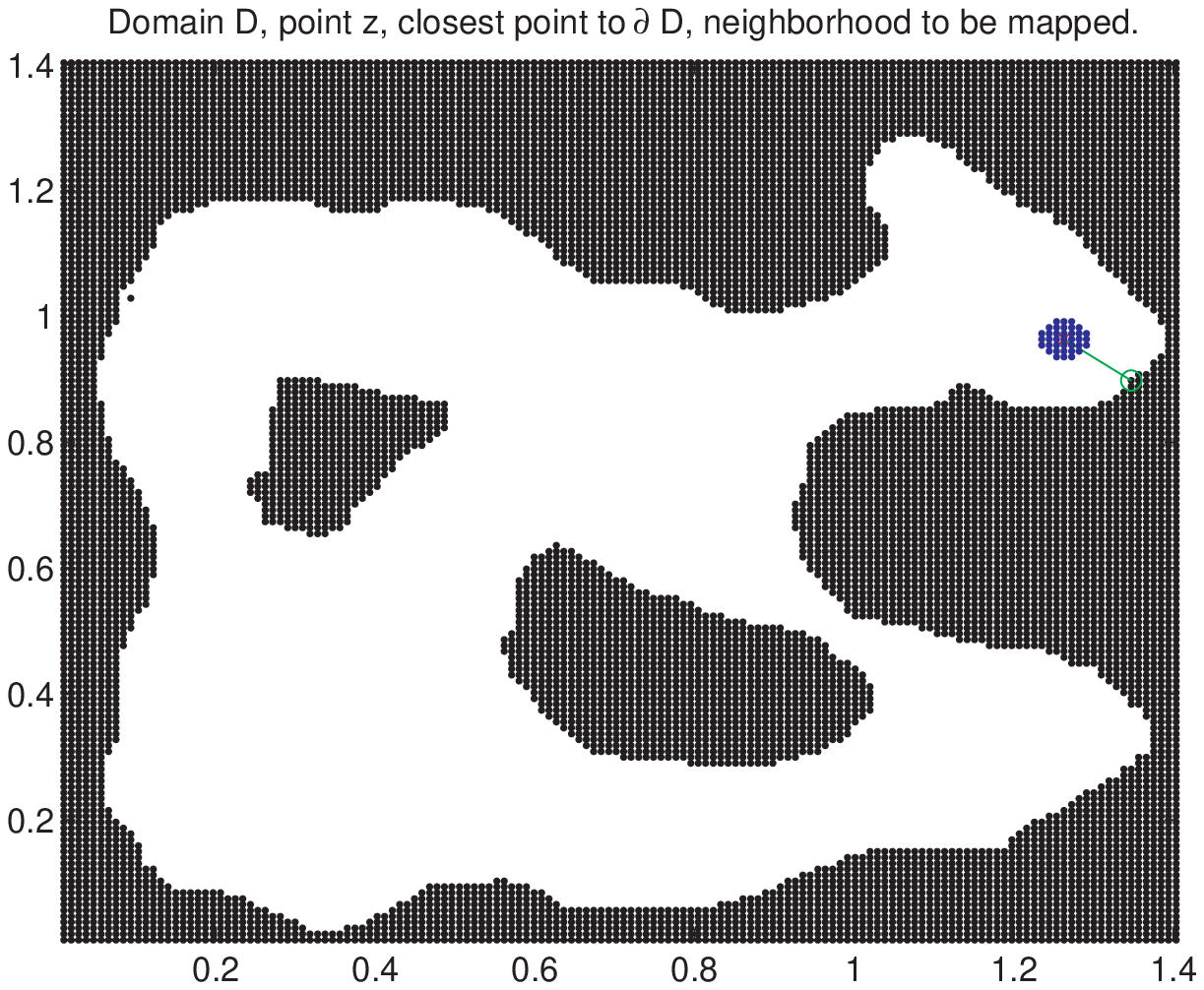}
\includegraphics[width=0.64\textwidth,height=0.4\textwidth]{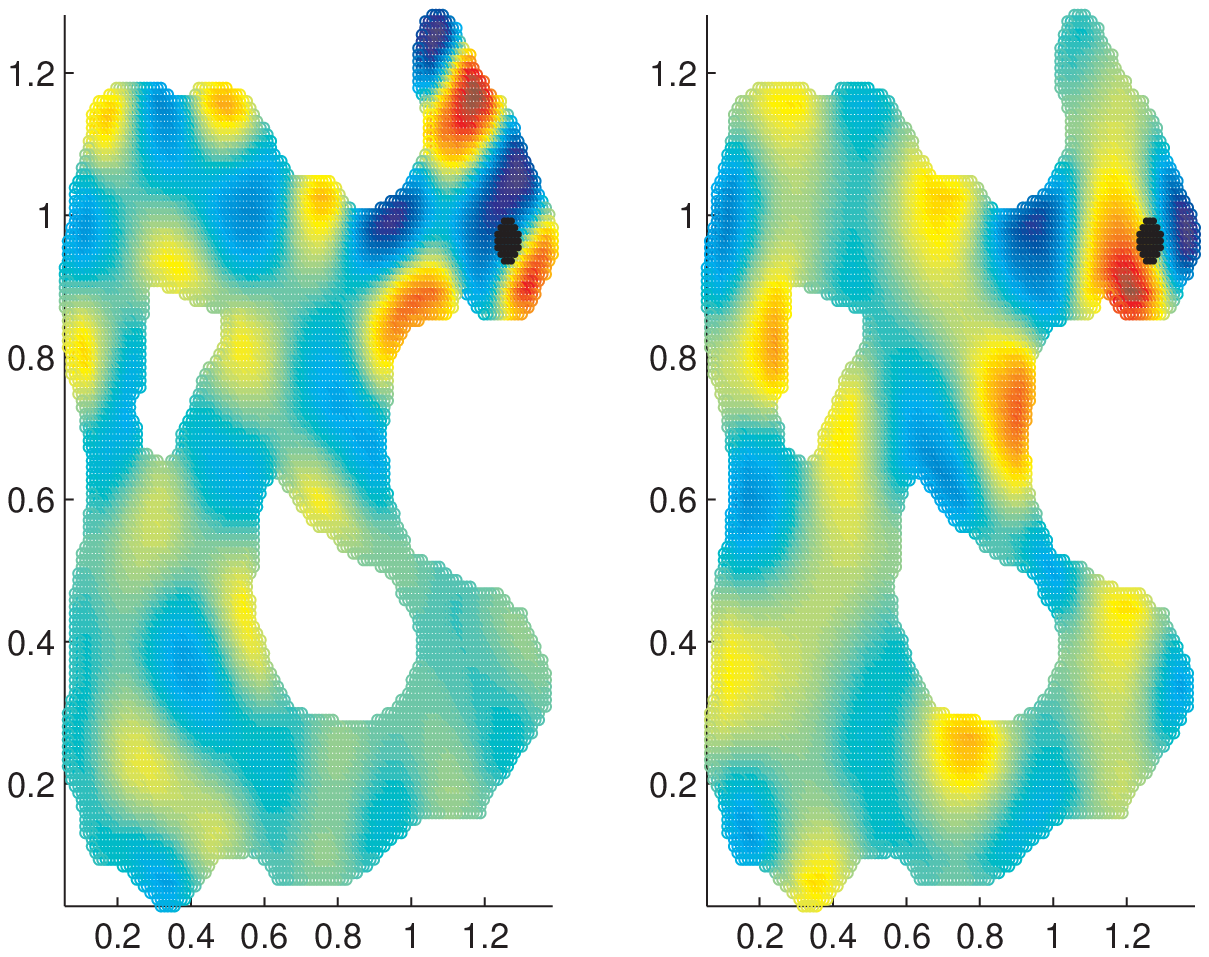}
\end{minipage}
\caption{Top left: a non-simply connected domain in $\mathbb{R}^2$, and the point $z$ with its neighborhood to be mapped. Top right: the image of the neighborhood under the map.
Bottom: Two eigenfunctions for mapping.}
\end{figure}

%
\newpage

\section{Appendix}\label{appendix-table-notation}
\subsection{Table of Notation/Symbols}\label{appendix_notation-table}
Below is a table of notation/symbols.  For each element, we list the first time it is defined mentioned,  repeat the definition if it is short, or explain what the symbol is typically used for.

\medskip

\begin{tabular}{l@{\extracolsep{.6in}}l}
{\bf Symbol}& {\bf Location}\\
\\
$A,\ A'$ & see \eqref{d:Lambda-L-H-definitions}    \\
$b_1$ & see Proposition \ref{p:PhiLinftyEstimates}\\
$\beta$, $\efk$ & see Proposition \ref{p:PhiLinftyEstimates} or Lemma \ref{lemma-ef-bound} \\
$\ball_r(x)$, $\ball(x,r)$ & open or closed metric ball of radius $r$ around the point $x$. \\
$c_0$ & see \eqref{d:Lambda-E-definitions}\\
$c_1,\ c_2,\ c_3,\ c_4,\ c_5,\ c_6$ & see Theorem \ref{t:datheorem} or \ref{t:datheoremmanifold}   \\
$C_1,\ C_2, C_1', C_2'$ & see Proposition \ref{p:non-smooth-kernel_estimates}   \\
$C_3,\ C_4$  & see Lemma \ref{l:truncated_kernel}   \\
$C_5,\ C_6$  & see Lemma \ref{l:truncated_truncated_kernel}   \\
$C_7,C_8$  & see Lemma \ref{l:bigpartials}   \\
$C_9$ &    see Proposition \ref{p:non-smooth-kernel_estimates} or \ref{p:non-smooth-kernel_estimates}\\
$c_{\max},\ c_{\min}$ & see \eqref{e:uniformlyelliptic}    \\
$\Cweyl$ & see \eqref{e:weyl-bounds-omega} \\
$\delta_0,\ \delta_1$ & preceding Proposition \ref{p:non-smooth-kernel_estimates}     \\
$d$ &  the dimension of the domain or manifold  \\
$\Delta$ &Euclidean Laplacian\\
$\Delta^\justB$ &Euclidean Laplacian with Dirichlet boundary\\
&conditions on a ball of radius $R$\\
$ \tilde\Delta^\justB$ & Laplace Beltrami operator  with Dirichlet boundary\\
& conditions on a ball of radius $R$\\
$\Delta_\mathcal{M}$ & Laplace Beltrami operator on $\mathcal{M}$\\
$\gamma_1,...,\gamma_d$ &  see Theorem \ref{t:datheorem} or \ref{t:datheoremmanifold} as well as equation \eqref{e:gamma-def}  \\
$\delta^{ij},\ g^{ij}$ & Euclidean and non-Euclidean metric tensor \\
$g_r(\cdot,\cdot)$ & see equation \eqref{e:gquadratic}\\
$G^r,\ \tilde  G^r$ &   Green function for the Euclidean or non-Euclidean\\
& ball of radius $r$(center is some fixed point).  In many\\
& place $r$ is $R$ which is assumed to have value $1$.\\
$\localEV,\ \localEF $ & eigenvalue and eigenfunction of a Euclidean ball\\
& (usually $\localTangentBall$ with $R=1$)     \\
$ \localPEV,\ \localPEF $ & eigenvalue and eigenfunction of a non-Euclidean ball \\
&(usually $\localTangentBall$ with $R=1$)     \\
$\globalEV,\ \globalEF$ &  eigenvalue and eigenfunction of the whole (Euclidean)\\
&  domain or manifold  \\
$\globalPEV,\ \globalPEF$ &  eigenvalue and eigenfunction of a the whole\\
&  domain or manifold,  with a perturbed metric\\
$\Lambda_L,\ \Lambda_H\ \Lambda_E$ & see \eqref{d:Lambda-L-H-definitions}, \eqref{d:Lambda-L-H-definitions} and    \eqref{d:Lambda-E-definitions}\\
$\Omega, \ \mathcal{M}$ &  a domain or a manifold  \\
$\efpoly, \ \gefpoly, \ \ggefpoly$&  see Proposition \ref{p:PhiLinftyEstimates}\\
$K^{Dir(B)},\ K^B$ &  heat kernel with Dirichlet boundary conditions on $B$,\\
& context dependent   \\
$K^{r}$ &  heat kernel with Dirichlet boundary conditions on a ball of radius $r$\\
&(with some fixed center),  context dependent   \\
$K^{\mathcal{N}}$ &heat kernel with Neumann boundary conditions,\\
& context dependent     \\
$K$ &   heat kernel, context dependent \\
$\partial_\dir K$ &   partial with respect to the second variable of a heat kernel\\
$\rho$ & see Theorem \ref{t:datheorem} and Theorem  \ref{t:datheoremmanifold}   \\
$R_z$ & written for $\rho$, as in  Theorem \ref{t:datheorem} and Theorem  \ref{t:datheoremmanifold}   \\
$\injRad$ &   see equation \eqref{e:rM} \\
$s$ &a time (for the heat kernel), usually smaller than $t$ (see\\
& below)    \\
$t$ &  a time (for the heat kernel), usually comparable to the\\
& square of the radius of a ball being discussed \\
$$ &    \\
$$ &    \\
$$ &    \\
$$ &    \\
$$ &    \\
\end{tabular}

\bibliographystyle{plain}
\bibliography{../../bibliography/MyPublications,../../bibliography/DiffusionBib,../../bibliography/bib-file-1}

\begin{thebibliography}{10}

\bibitem{BN:SpectralTechniquesEmbeddingClustering}
M.~Belkin and P.~Niyogi.
\newblock Laplacian eigenmaps and spectral techniques for embedding and
  clustering.
\newblock In {\em Advances in Neural Information Processing Systems 14 (NIPS
  2001)}, pages 585--591. MIT Press, Cambridge, 2001.

\bibitem{BNeigenmaps}
M.~Belkin and P.~Niyogi.
\newblock Laplacian eigenmaps for dimensionality reduction and data
  representation.
\newblock {\em Neural Computation}, 6(15):1373--1396, June 2003.

\bibitem{BN}
M.~Belkin and P.~Niyogi.
\newblock Using manifold structure for partially labelled classification.
\newblock {\em Advances in NIPS}, 15, 2003.

\bibitem{NB}
M.~Belkin and P.~Niyogi.
\newblock Semi-supervised learning on {R}iemannian manifolds.
\newblock {\em Machine Learning}, 56(Invited Special Issue on
  Clustering):209--239, 2004.
\newblock TR-2001-30, Univ. Chicago, CS Dept., 2001.

\bibitem{BBG:EmbeddingRiemannianManifoldHeatKernel}
P.~B\'erard, G.~Besson, and S.~Gallot.
\newblock Embedding {R}iemannian manifolds by their heat kernel.
\newblock {\em Geom. and Fun. Anal.}, 4(4):374--398, 1994.

\bibitem{ganesh}
M.~Chaplain, M.~Ganesh, and I.Graham.
\newblock Spatio-temporal pattern formation on spherical surfaces: numerical
  simulation and application to solid tumor growth.
\newblock {\em J. Math. Biology}, 42:387--423, 2001.

\bibitem{Cheeger70}
J~Cheeger.
\newblock A lower bound for the smallest eigenvalue of the laplacian.
\newblock In RC~Gunning, editor, {\em Problems in Analysis}, pages 195--199.
  Princeton Univ. Press.

\bibitem{Chung}
F.~R.~K. Chung.
\newblock {\em Spectral graph theory}, volume~92 of {\em CBMS Regional
  Conference Series in Mathematics}.
\newblock Published for the Conference Board of the Mathematical Sciences,
  Washington, DC, 1997.

\bibitem{DiffusionPNAS}
R.~R. Coifman, S.~Lafon, A.~B. Lee, M.~Maggioni, B.~Nadler, F.~Warner, and
  S.~W. Zucker.
\newblock Geometric diffusions as a tool for harmonic analysis and structure
  definition of data: Diffusion maps.
\newblock {\em PNAS}, 102(21):7426--7431, 2005.

\bibitem{DiffusionPNAS2}
R.~R. Coifman, S.~Lafon, A.~B. Lee, M.~Maggioni, B.~Nadler, F.~Warner, and
  S.~W. Zucker.
\newblock Geometric diffusions as a tool for harmonic analysis and structure
  definition of data: Diffusion maps.
\newblock {\em PNAS}, 102(21):7432--7438, 2005.

\bibitem{CKLMN:DiffusionMapsReductionCoordinates}
R.R. Coifman, I.G. Kevrekidis, S.~Lafon, M.~Maggioni, and B.~Nadler.
\newblock Diffusion maps, reduction coordinates and low dimensional
  representation of stochastic systems.
\newblock {\em to appear Siam J.M.M.S.}, 2008.

\bibitem{DiffusionMaps}
R.R. Coifman and S.~Lafon.
\newblock Diffusion maps.
\newblock {\em submitted to Applied and Computational Harmonic Analysis}, 2004.

\bibitem{CLAcha1}
R.R. Coifman and S.~Lafon.
\newblock Diffusion maps.
\newblock {\em Appl. Comp. Harm. Anal.}, 21(1):5--30, 2006.

\bibitem{CLAcha2}
R.R. Coifman and S.~Lafon.
\newblock Geometric harmonics: a novel tool for multiscale out-of-sample
  extension of empirical functions.
\newblock {\em Appl. Comp. Harm. Anal.}, 21(1):31--52, 2006.

\bibitem{CMDiffusionWavelets}
R.R. Coifman and M.~Maggioni.
\newblock Diffusion wavelets.
\newblock {\em Appl. Comp. Harm. Anal.}, 21(1):53--94, July 2006.
\newblock (Tech. Rep. YALE/DCS/TR-1303, Yale Univ., Sep. 2004).

\bibitem{Davies:SpectraPropertiesChangesMetric}
E.~B. Davies.
\newblock Spectral properties of compact manifolds and changes of metric.
\newblock {\em Amer. J. Math.}, 112(1):15--39, 1990.

\bibitem{DaviesHeatKernels}
E.B. Davies.
\newblock {\em Heat kernels and spectral theory}.
\newblock Cambridge University Press, 1989.

\bibitem{Donnelly-Fefferman}
H.~Donnelly and C.~Fefferman.
\newblock Growth and geometry of eigenfunctions of the {L}aplacian.
\newblock In {\em Analysis and partial differential equations}, volume 122 of
  {\em Lecture Notes in Pure and Appl. Math.}, pages 635--655. Dekker, New
  York, 1990.

\bibitem{DoGri:WhenDoesIsoMap}
D.~L. Donoho and C.~Grimes.
\newblock When does isomap recover natural parameterization of families of
  articulated images?
\newblock Technical Report Tech. Rep. 2002-27, Department of Statistics,
  Stanford University, August 2002.

\bibitem{DG_HessianEigenmaps}
D.~L Donoho and Carrie Grimes.
\newblock Hessian eigenmaps: new locally linear embedding techniques for
  high-dimensional data.
\newblock {\em Proc. Nat. Acad. Sciences}, pages 5591--5596, March 2003.
\newblock also tech. report, Statistics Dept., Stanford University.

\bibitem{donoho2}
D.~L. Donoho, O.~Levi, J.-L. Starck, and V.~J. Martinez.
\newblock Multiscale geometric analysis for 3-d catalogues.
\newblock Technical report, Stanford Univ., 2002.

\bibitem{Doob_PotentialTheory}
Joseph~L Doob.
\newblock {\em Classical potential theory and its probabilistic counterpart}.
\newblock Springer-Verlag, New York, 1984.

\bibitem{Federer:GeometricMeasureTheory}
H.~Federer.
\newblock {\em Geometric Measure Theory}.
\newblock Springer-Verlag, 1969.

\bibitem{GHL_RiemannianGeometry}
S.~Gallot, D.~Hulin, and J.~Lafontaine.
\newblock {\em Riemannian Geometry}.
\newblock Springer-Verlag, Berlin, 1987.

\bibitem{GruterWidman:GreenFunction}
M.~Gr\"uter and K.-O. Widman.
\newblock The {G}reen function for uniformly elliptic equations.
\newblock {\em Man. Math.}, 37:303--342, 1982.

\bibitem{niyogi}
X.~He, S.~Yan, Y.~Hu, P.~Niyogi, and H.-J. Zhang.
\newblock Face recognition using laplacianfaces.
\newblock {\em IEEE Trans. pattern analysis and machine intelligence},
  27(3):328--340, 2005.

\bibitem{Simon:NeumannEssentialSpectrum}
R.~Hempel, L.~Seco, and B.~Simon.
\newblock The essential spectrum of {N}eumann {L}aplacians on some bounded
  singular domains, 1991.

\bibitem{isomap}
V.~De~Silva J.~B.~Tenenbaum and J.~C. Langford.
\newblock A global geometric framework for nonlinear dimensionality reduction.
\newblock {\em Science}, 290(5500):2319--2323, 2000.

\bibitem{jms:UniformizationEigenfunctions}
P.W. Jones, M.~Maggioni, and R.~Schul.
\newblock Manifold parametrizations by eigenfunctions of the {L}aplacian and
  heat kernels.
\newblock {\em Proc. Nat. Acad. Sci.}, 2007.
\newblock to appear.

\bibitem{Kellogg:PotentialTheory}
O.D. Kellogg.
\newblock {\em Foundations of potential theory}.
\newblock Berlin, 1929.

\bibitem{LThesis}
S.~Lafon.
\newblock {\em Diffusion maps and geometric harmonics}.
\newblock PhD thesis, Yale University, Dept of Mathematics \& Applied
  Mathematics, 2004.

\bibitem{pclo}
P-C. Lo.
\newblock Three dimensional filtering approach to brain potential mapping.
\newblock {\em IEEE Tran. on biomedical engineering}, 46(5):574--583, 1999.

\bibitem{smkfsomm:SimLearningReprControlContinuous}
S.~Mahadevan, K.~Ferguson, S.~Osentoski, and M.~Maggioni.
\newblock Simultaneous learning of representation and control in continuous
  domains.
\newblock In {\em AAAI}. AAAI Press, 2006.

\bibitem{smmm:ValueFunction}
S.~Mahadevan and M.~Maggioni.
\newblock Value function approximation with diffusion wavelets and laplacian
  eigenfunctions.
\newblock In {\em University of Massachusetts, Department of Computer Science
  Technical Report TR-2005-38; Proc. NIPS 2005}, 2005.

\bibitem{Netrusov:SharpRemainderEstimatesNeumannPlanar}
Yu. Netrusov.
\newblock Sharp remainder estimates in the weyl formula for the neumann
  laplacian on a class of planar regions.
\newblock {\em Journal of Functional Analysis}, 250(1):21--41, 2007.

\bibitem{Safarov:WeylAsymptoticFormula}
Yu. Netrusov and Yu. Safarov.
\newblock Weyl asymptotic formula for the {L}aplacian on domains with rough
  boundaries.
\newblock {\em Comm. Math. Phys.}, 253(2):481--509, 2005.

\bibitem{ng01spectral}
A.~Ng, M.~Jordan, and Y.~Weiss.
\newblock On spectral clustering: Analysis and an algorithm, 2001.

\bibitem{NMB}
P.~Niyogi, I.~Matveeva, and M.~Belkin.
\newblock Regression and regularization on large graphs.
\newblock Technical report, University of Chicago, Nov. 2003.

\bibitem{Oksendal-book}
Bernt \"Oksendal.
\newblock {\em Stochastic differential equations}.
\newblock Universitext. Springer-Verlag, 1985.
\newblock An introduction with applications.

\bibitem{pommerenke}
Christian Pommerenke.
\newblock {\em Univalent functions}.
\newblock Vandenhoeck \& Ruprecht, G\"ottingen, 1975.
\newblock With a chapter on quadratic differentials by Gerd Jensen, Studia
  Mathematica/Mathematische Lehrb\"ucher, Band XXV.

\bibitem{RSLLE}
ST~Roweis and LK~Saul.
\newblock Nonlinear dimensionality reduction by locally linear embedding.
\newblock {\em Science}, 290:2323--2326, 2000.

\bibitem{Saul:SpectralMethodsDimReduction}
L.K. Saul, K.Q. Weinberger, F.H. Ham, F.~Sha, and D.D. Lee.
\newblock {\em Spectral methods for dimensionality reduction}, chapter
  Semisupervised Learning.
\newblock MIT Press, 2006.

\bibitem{Saul:AnalysisExtensionSpectralMethods}
F.~Sha and L.K. Saul.
\newblock Analysis and extension of spectral methods for nonlinear
  dimensionality reduction.
\newblock {\em Proc. ICML}, pages 785--792, 2005.

\bibitem{Shen:fMRIDiffusionMaps}
X.~Shen and F.G. Meyer.
\newblock Analysis of event-related fmri data using diffusion maps.
\newblock In {\em Proc. IPMI}, pages 652--663, 2005.

\bibitem{shi-malik:pami}
J.~Shi and J.~Malik.
\newblock Normalized cuts and image segmentation.
\newblock {\em IEEE PAMI}, 22:888--905, 2000.

\bibitem{smith:sharplplq}
H.F. Smith.
\newblock Sharp $l^2-l^q$ bounds on spectral projectors for low regularity
  metrics.
\newblock {\em Math. Res. Lett.}, 13(6):967--974, 2006.

\bibitem{smith:c11metrics}
H.F. Smith.
\newblock Spectral cluster estimates for $c^{1,1}$ metrics.
\newblock {\em Amer. Jour. Math.}, 128:1069--1103, 2006.

\bibitem{sogge}
C.D. Sogge.
\newblock Eigenfunction and bochner--riesz estimates on manifolds with
  boundary.
\newblock {\em Mathematical Research Letters}, 9:205--216, 2002.

\bibitem{Spielman:SpectralPartitioningWorks}
D.A. Spielman and S.H. Teng.
\newblock Spectral partitioning works: Planar graphs and finite element meshes.
\newblock {\em FOCS}, 1996.

\bibitem{stroock:PDEforProbabilists}
Daniel~W. Stroock.
\newblock {\em Partial differential equations for probabilists}, volume 112 of
  {\em Cambridge Studies in Advanced Mathematics}.
\newblock Cambridge University Press, Cambridge, 2008.

\bibitem{Stroock:DiffusionSemigroups}
D.W. Stroock.
\newblock Diffusion semigroups corresponding to uniformly elliptic divergence
  form operators.
\newblock {\em S\'eminaire de probabilit\'es}, 22:316--347, 1988.

\bibitem{SMC:GeneralFrameworkAdaptiveRegularization}
A.D. Szlam, M.~Maggioni, and R.R. Coifman.
\newblock A general framework for adaptive regularization based on diffusion
  processes on graphs.
\newblock Technical Report YALE/DCS/TR1365, submitted, Yale Univ, July 2006.

\bibitem{MSCB:MultiscaleManifoldMethods}
A.D. Szlam, M.~Maggioni, R.R. Coifman, and J.C.~Bremer Jr.
\newblock Diffusion-driven multiscale analysis on manifolds and graphs:
  top-down and bottom-up constructions.
\newblock volume 5914-1, page 59141D. SPIE, 2005.

\bibitem{TSL}
J.B. Tenenbaum, V.~de~Silva, and J.C. Langford.
\newblock A global geometric framework for nonlinear dimensionality reduction.
\newblock {\em Science}, 290:2319--2323, 2000.

\bibitem{Weinberger:MaximumVarianceUnfolding}
K.~Q. Weinberger and L.~K. Saul.
\newblock An introduction to nonlinear dimensionality reduction by maximum
  variance unfolding.
\newblock In {\em Proc. AAAI}, 2006.

\bibitem{Saul:LearningKernelMatrixNonlineadDimReduction}
K.Q. Weinberger, F.~Sha, and L.K. Saul.
\newblock Leaning a kernel matrix for nonlinear dimensionality reduction.
\newblock {\em Proc. ICML}, pages 839--846, 2004.

\bibitem{Wiener:Series}
N.~Wiener.
\newblock {\em Journ. Math. and Phys. M.I.T.}, 3:24--51,127--146, 1924.

\bibitem{Xu}
Xiangjin Xu.
\newblock New proof of the {H}\"ormander multiplier theorem on compact
  manifolds without boundary.
\newblock {\em Proc. Amer. Math. Soc.}, 135(5):1585--1595 (electronic), 2007.

\bibitem{ZhaZha}
Z.~Zhang and H.~Zha.
\newblock Principal manifolds and nonlinear dimension reduction via local
  tangent space alignement.
\newblock Technical Report CSE-02-019, Department of computer science and
  engineering, Pennsylvania State University, 2002.

\end{thebibliography}

\end{document}